\documentclass{article}
\usepackage[utf8]{inputenc}

\usepackage[a4paper,top=2cm,bottom=2cm,left=2cm,right=2cm,marginparwidth=1cm]{geometry}

\usepackage{upgreek}
\usepackage{amsmath,amssymb,amsthm,amsfonts}
\usepackage{mathtools}
\usepackage{dsfont}
\usepackage{graphicx}
\usepackage{subcaption}
\usepackage{float}
\usepackage{bbm}
\usepackage{bm}
\usepackage{enumitem}
\usepackage{pifont}
\usepackage{comment}

\usepackage[colorinlistoftodos]{todonotes}
\usepackage[allcolors=black]{hyperref}

\usepackage[square,numbers]{natbib}
\newcommand{\PP}{{\mathbb{P}}}

\newcommand{\EE}{{\mathbb{E}}}
\newcommand{\LL}{{\mathcal{L}}}
\newcommand{\eps}{{\epsilon}}
\newcommand{\HH}{{\mathbb{H}}}

\newcommand{\Lj}{{\operatorname{Int}(L_j)}}
\newcommand{\Lk}{{\operatorname{Int}(L_k)}}
\newcommand{\CR}{{\operatorname{CR}}}
\newcommand{\Area}{{\operatorname{Area}}}
\newcommand{\der}{{\operatorname{d}}}

\newtheorem{theorem}{Theorem}[section]
\newtheorem{corollary}[theorem]{Corollary}
\newtheorem{lemma}[theorem]{Lemma}

\newtheorem{proposition}[theorem]{Proposition}
\newtheorem{claim}[theorem]{Claim}
\theoremstyle{definition}
\newtheorem{definition}[theorem]{Definition}
\theoremstyle{remark}
\newtheorem{remark}[theorem]{Remark}
\theoremstyle{remark}

\title{Massive SLE$_4$, massive CLE$_4$ and the massive planar GFF}
\author{L\'eonie Papon \thanks{Durham University}}
\date{\today}

\begin{document}

\maketitle

\begin{abstract}
We construct a coupling between a massive GFF and a random curve in which the curve can be interpreted as the level line of the field and has the law of massive SLE$_4$. This coupling is obtained by reweighting the law of the standard coupling GFF-SLE$_4$ and our result can be seen as a conditional version of the path-integral formulation of the massive GFF. We then show that by reweighting the law of the coupling GFF-CLE$_4$ in a similar way, one obtains a coupling between a massive GFF and a random countable collection of simple loops, that we call massive CLE$_4$. Using this coupling, we relate massive CLE$_4$ to the massive Brownian loop soup with intensity $1/2$, thus proving a conjecture of Camia. As the law of the massive GFF, the laws of massive SLE$_4$ and massive CLE$_4$ are conformally covariant.
\end{abstract}

\section{Introduction}

\subsection{Main results}

Chordal SLE$_{\kappa}$ curves are a one-parameter family of random planar curves that can be characterized by their conformal invariance and a certain Markov property \cite{Schramm_SLE}. They have a rich interplay with the Gaussian free field (GFF), a random Markovian field defined on the plane which is conformally invariant. The prime example of this interplay is the existence of a coupling between a GFF with appropriate boundary conditions and an SLE$_4$ curve \cite{contour_line, Dub_SLE, IG_1}. In this coupling, the curve is measurable with respect to the field and can be seen as its level line. This is perhaps surprising, since the GFF is too rough to be a pointwise defined function.

Recently, this coupling was shown to also hold for massive versions of these objects: a massive SLE$_4$ curve can be coupled with a massive Gaussian free field with appropriate boundary conditions as its level line \cite{mHE}. In \cite{mHE}, the massive SLE$_4$ curve is defined via its driving function and the proof relies on the explicit expression for this function. However, in a simply connected domain $D \subset \mathbb{C}$ and under suitable conditions on the mass, the massive GFF with mass $m$ is absolutely continuous with respect to the GFF with Radon-Nikodym derivative given by
\begin{equation} \label{RN_intro}
    \frac{\der \PP_{\operatorname{mGFF}}}{\der \PP_{\operatorname{GFF}}}(h)=\frac{1}{\mathcal{Z}} \exp \bigg( -\frac{1}{2} \int_{D} m^2(z) :h^2(z):dz \bigg)
\end{equation}
where $\mathcal{Z}$ is a normalization constant. In \eqref{RN_intro}, $:h^2:$ is the Wick square of the GFF, which is a renormalized version of the square of the GFF. With this renormalization, the change of measure \eqref{RN_intro} can be seen as a rigorous path-integral formulation of the massive GFF.

In view of \eqref{RN_intro}, it is natural to wonder whether the law of the coupling massive GFF-massive SLE$_4$ can be obtained by reweighting the law of the coupling GFF-SLE$_4$. This idea, albeit under a different form, was already suggested in the physics literature \cite[Section~4.6]{BBC} and the next theorem rigorously formalizes it. 

In everything that follows, we set $\lambda:=\sqrt{\pi/8}$, and if $D \subset \mathbb{C}$ is an open, bounded and simply connected domain, with $a,b \in \partial D$ (in the sense of prime ends), we denote by $\partial D^{+}$, respectively $\partial D^{-}$, the clockwise-oriented, respectively counterclockwise-oriented, boundary arc $(ab)$. We also define $F^{(D,a,b)}:\partial D\to \mathbb{R}$ to be equal to $\lambda$ on $\partial D^{+}$ and $-\lambda$ on $\partial D^{-}$.

\begin{theorem} \label{theorem_mSLE_4}
 Let $D,a,b$ be as above, and let $\phi: D \to \mathbb{R}$ be the unique harmonic function in $D$ with boundary values $F^{(D,a,b)}$. Denote by $\PP$ the law of the coupling $(h+\phi, \gamma)$ between a GFF $h+\phi$ with boundary conditions $\phi$ in $D$ and an SLE$_4$ curve $\gamma$ in $D$ from $a$ to $b$. Let $m:D \to \mathbb{R}_{+}$ be of the form $\vert \varphi' \vert \hat m \circ \varphi$, where $\varphi: D \to \hat D$ is a conformal isomorphism, $\hat D \subset \mathbb{C}$ is a bounded and simply connected domain and $\hat m:\hat D \to \mathbb{R}_{+}$ is a bounded and continuous function. Define a new probability measure $\tilde{\PP}$ by
\begin{equation} \label{RN_mSLE4}
    \frac{\der \tilde \PP}{\der \PP}((h+\phi,\gamma)) := \frac{1}{\mathcal{Z}} \exp \bigg( -\frac{1}{2} \int_{D} m^2(z) :(h+\phi)^2(z): dz \bigg)
\end{equation}
where $\mathcal{Z}$ is a normalization constant. Then, under $\tilde \PP$,
\begin{itemize}\setlength{\itemsep}{0em}
    \item the marginal law of $h+\phi$ is that of a massive GFF in $D$ with mass $m$ and  
    boundary conditions $F^{(D,a,b)}$;
    \item the marginal law of $\gamma$ is that of a massive SLE$_4$ curve with mass $m$ in $D$ from $a$ to $b$, as defined in Section \ref{sec_def_mSLE4}.
\end{itemize}
Moreover, let $t \in [0,\infty)$. Then, under $\tilde \PP$, conditionally on $\gamma([0,t])$, $h+\phi=h_t+\phi_t$ where $h_t+\phi_t$ has the law of a massive GFF in $D \setminus \gamma([0,t])$ with mass $m$ and with boundary conditions $F^{(D\setminus \gamma[0,t],\gamma(t),b)}$. The same result holds at $t=\infty$; with $h_t+\phi_t$ being a sum of independent massive GFFs with mass $m$ and boundary conditions $\pm \lambda$ on either side of the curve.
\end{theorem}

Theorem \ref{theorem_mSLE_4} implies that the law of massive SLE$_4$ is absolutely continuous with respect to that of ordinary SLE$_4$, and that in the massive GFF-massive SLE$_4$ coupling (i.e. the law of $(h+\phi,\gamma)$ under $\tilde{\mathbb{P}}$), the curve is a measurable function of the field.\footnote{If this measurability had already been proven in \cite{mHE}, then Theorem \ref{theorem_mSLE_4} would follow from the results of \cite{mHE}. However, this fact was not straightforward to prove in the set-up of \cite{mHE}.} 
Observe also that by conformal invariance of SLE$_4$ and conformal covariance of the random variable on the right-hand side of \eqref{RN_mSLE4}, we can deduce from Theorem \ref{theorem_mSLE_4} that massive SLE$_4$ is conformally covariant, in a sense that will be made precise below.

Finally, Theorem \ref{theorem_mSLE_4} combined with Le Jan's isomorphism theorem \cite{LeJan, LeJan_2} yield the following expression for the Radon-Nikodym derivative of the law $\PP_{\operatorname{mSLE}_4}^{(D,a,b)}$ of massive SLE$_4$ with respect to the law $\PP_{\operatorname{SLE}_4}^{(D,a,b)}$ of SLE$_4$.

\begin{corollary} \label{cor_RN_mSLE4}
Let $D \subset \mathbb{C}$, $a,b \in \partial D$ and $m: D \to \mathbb{R}_{+}$ be as in Theorem \ref{theorem_mSLE_4}. Then, for $t \geq 0$,
\begin{align*}
    \frac{\der \PP_{\operatorname{mSLE}_4}^{(D,a,b)}}{\der \PP_{\operatorname{SLE}_4}^{(D,a,b)}} (\gamma) \bigg \vert _{\sigma(\gamma(s), s \leq t)} = &\frac{1}{\mathcal{Z}}\exp \bigg( \frac{1}{2}\mu_t\big( e^{-\langle \ell, m^2 \rangle} + \langle \ell, m^2 \rangle -1 \big) -\frac{1}{2} \int_{D_t} m^2(z) \phi_t(z) \phi_t^m(z) dz \bigg)\\ 
    &\times \exp \bigg(\int_{D_t} \frac{m^2(z)}{4\pi} \log \frac{\CR(z,\partial D)}{\CR(z,\partial D_t)} dz \bigg)
\end{align*}
where $\mathcal{Z}$ is as in the definition \eqref{RN_mSLE4} of $\tilde \PP$ in Theorem \ref{theorem_mSLE_4}. Above, $D_t := D \setminus \gamma([0,t])$ and $\phi_t$, respectively $\phi_t^m$, is the harmonic, respectively massive harmonic, function in $D_t$ with boundary conditions $F^{(D_t,\gamma(t),b)}$.

Here $\mu_t$ denotes the Brownian loop measure in $D_t$, and for a loop $\ell$ , $\langle \ell, m^2 \rangle := \int_{0}^{\tau(\ell)} m^2(\ell(t)) dt$, where $\tau(\ell)$ is the lifetime of $\ell$. For $z \in D$, $\CR(z,\partial D)$, respectively $\CR(z,\partial D_t)$ is the conformal radius of $z$ in $D$, respectively $D_t$.
\end{corollary}

SLE$_{\kappa}$ curves for $\kappa \in (8/3,8)$ have loop versions, called conformal loop ensembles (CLE). CLE$_{\kappa}$ are a one-parameter family of random countable collections of planar loops characterized by their conformal invariance and a certain Markovian property \cite{CLE_Markovian}. There also exists a coupling between a CLE$_4$ and a GFF with Dirichlet boundary conditions \cite{slides_CLE, BTLS} in which the loops of CLE$_4$ are in a certain sense, similarly to SLE$_4$, level lines of the GFF. One may wonder whether a massive version of this coupling exists. To date, no massive version of CLE$_4$, or of CLE$_{\kappa}$ for other values of $\kappa$, has been defined. However, using the same ideas as those motivating Theorem \ref{theorem_mSLE_4}, we can construct a coupling between a massive GFF and a random collection of planar loops.

\begin{theorem} \label{theorem_mCLE_4}
Let $D \subset \mathbb{C}$ and $m:D \to \mathbb{R}_{+}$ be as in Theorem \ref{theorem_mSLE_4}. Denote by $\PP$ the law of the coupling $(h, \Gamma)$ between a GFF $h$ in $D$ with Dirichlet boundary conditions and a CLE$_4$ $\Gamma$ in $D$. Define a new probability measure $\tilde \PP$ by
\begin{equation} \label{RN_mCLE4}
    \frac{\der \tilde \PP}{\der \PP}((h,\Gamma)) := \frac{1}{\mathcal{Z}} \exp \bigg(-\frac{1}{2} \int_{D} m^2(z) :h^2(z): dz \bigg)
\end{equation}
where $\mathcal{Z}$ is a normalization constant. Then, under $\tilde \PP$, the marginal law of $h$ is that of a massive GFF in $D$ with mass $m$ and with Dirichlet boundary conditions. Moreover, under $\tilde \PP$, conditionally on $\Gamma$, 
\begin{equation*}
    h=\sum_j h_j+\xi_j
\end{equation*}
where the sum runs over the loops $(L_j)_j$ of $\Gamma$ and
\begin{itemize}\setlength{\itemsep}{0em}
    \item for each $j$, $h_{j}+\xi_j$ has the law of a massive GFF in the interior of the loop $L_j$ with mass $m$ and boundary conditions $\xi_j$;
    \item the fields $(h_j+\xi_j)_j$ are independent;
    \item the random variables $(\xi_j)_j$ are independent and
    $
        \tilde \PP(\xi_j=-2\lambda \vert \Gamma) = \tilde \PP(\xi_j=2\lambda \vert \Gamma)=1/2 \; \forall j;
   $ 
    \item the random variables $(\xi_j)_j$ are measurable with respect to the fields $(h_j+\xi_j)_j$.
\end{itemize}
\end{theorem}

\begin{definition}
We define massive CLE$_4$ to be the law of the collection of loops $\Gamma$ under $\tilde{\mathbb{P}}$, and we denote this law by $\PP_{\operatorname{mCLE}_4}^{D}$.
\end{definition}

Again it follows from the above theorem and definition that the law of massive CLE$_4$ is absolutely continuous with respect to the law of CLE$_4$ in $D$. In particular, massive CLE$_4$ defines a countable collection of planar loops in $D$ that are almost surely simple and almost surely do not touch each other or the boundary of $D$. Moreover, we can deduce from Theorem \ref{theorem_mCLE_4} that the law of massive CLE$_4$ is conformally covariant.

\begin{corollary} \label{cor_conformal_cov_mCLE4}
Let $D \subset \mathbb{C}$ and $m:D \to \mathbb{R}_{+}$ be as in Theorem \ref{theorem_mCLE_4}. Let $\tilde \varphi: D \to \tilde D$ be a conformal map. If $\Gamma$ has the law of massive CLE$_4$ in $D$ with mass $m$, then $\tilde \varphi(\Gamma)$ has the law of massive CLE$_4$ in $\tilde D$ with mass given by, for $w \in \tilde D$,
\begin{equation*}
    \tilde m^2(w) = \vert (\tilde \varphi^{-1})'(w) \vert^2 m^2(\tilde \varphi^{-1}(w)).
\end{equation*}
\end{corollary}

Observe also that the coupling of Theorem \ref{theorem_mCLE_4} can be iterated in the interiors of the loops to obtain a loop decomposition of the massive GFF similar to that of the GFF \cite{BTLS}. However, in this decomposition, the loops are not identically distributed: this is because massive CLE$_4$ is only conformally covariant and not conformally invariant.

As in the case of massive SLE$_4$, a consequence of Theorem \ref{theorem_mCLE_4} is an explicit expression for the Radon-Nikodym derivative of $\PP_{\operatorname{mCLE}_4}^{D}$ with respect to the law $\PP_{\operatorname{CLE}_4}^{D}$ of CLE$_4$ in $D$.

\begin{corollary} \label{cor_RN_mCLE4}
Let $D \subset \mathbb{C}$ and $m:D \to \mathbb{R}_{+}$ be as in Theorem \ref{theorem_mCLE_4}. Then
\begin{align*}
    \frac{\der \PP_{\operatorname{mCLE}_4}^{D}}{\der \PP_{\operatorname{CLE}_4}^{D}}(\Gamma) = \frac{1}{\mathcal{Z}} \prod_j &\exp \bigg(\frac{1}{2}\mu_j \big( e^{-\langle \ell, m^2 \rangle} + \langle \ell, m^2 \rangle -1\big) -2\lambda^2 \int_{\Lj} m^2(z) H_j^{(m)}(z, L_j) dz \bigg) \\
    &\times \exp \bigg( \int_{\Lj} \frac{m^2(z)}{4\pi} \log \frac{\CR(z,\partial D)}{\CR(z, L_j)} dz \bigg)
\end{align*}
where $\mathcal{Z}$ is as in the definition \eqref{RN_mCLE4} of $\tilde \PP$ in Theorem \ref{theorem_mCLE_4} and where the product runs over the loops $(L_j)_j$ of $\Gamma$. For each $j$, $\mu_j$ is the Brownian loop measure in the interior $\Lj$ of the loop $L_j$ while, for $z \in \Lj$, $H_j^{(m)}(z, L_j)$ is the massive harmonic measure of $L_j$ seen from $z$.
\end{corollary}

\begin{remark}
In the massless case, the partition function of SLE$_{\kappa}$ can be related to the $\zeta$-regularized determinant of the Laplacian \cite{Dub_SLE, LDP_SLE_1}. One may wonder if massive SLE$_4$ is related in some way to the $\zeta$-regularized determinant of the massive Laplacian. Using the computations of \cite[Appendix~A]{BBC}, one can relate the exponential of the loop measure term of Corollary \ref{cor_RN_mSLE4} and Corollary \ref{cor_RN_mCLE4} to
\begin{equation*}
    \bigg(\frac{\det_{\zeta}(-\Delta_t+m^2)}{\det_{\zeta}(-\Delta_t)} \bigg)^{-1/2}, \quad \bigg(\frac{\det_{\zeta}(-\Delta_j+m^2)}{\det_{\zeta}(-\Delta_j)} \bigg)^{-1/2}
\end{equation*}
where $\det_{\zeta}$ denotes the $\zeta$-regularized determinant and $\Delta_t$, respectively $\Delta_j$, is the Laplacian in $D_t$, respectively the interior of the loop $L_j$, with Dirichlet boundary conditions.
\end{remark}

Although Corollary \ref{cor_RN_mCLE4} gives an explicit expression for the Radon-Nikodym derivative of $\PP_{\operatorname{mCLE_4}}^{D}$ with respect to $\PP_{\operatorname{CLE}_4}^{D}$, this does not provide a direct construction of massive CLE$_4$. Such a construction exists in the massless case \cite{CLE_Markovian}: if one first samples a Brownian loop soup with intensity $1/2$ in $D$ and keeps only the outer boundaries of the outermost clusters of loops of the loop soup, then one obtains a collection of loops in $D$ that has the law of CLE$_4$ in $D$. In \cite{Camia_BLS}, Camia introduced a massive version of the Brownian loop soup and conjectured that the outer boundaries of its outermost clusters at intensity $1/2$ are distributed as the level lines of the massive GFF. Thanks to Theorem \ref{theorem_mCLE_4} which shows that the loops of massive CLE$_4$ are the level lines of the massive GFF and to the coupling of \cite[Proposition~5]{Qian_GFF_BLS}, we can prove this conjecture and give an explicit construction of massive CLE$_4$.

\begin{theorem} \label{prop_mCLE4_mBLS}
Let  $D \subset \mathbb{C}$ and $m:D \to \mathbb{R}_{+}$ be as in Theorem \ref{theorem_mSLE_4}. Then a massive CLE$_4$ $\Gamma$ in $D$ with mass $m$ and a massive Brownian loop soup $\mathcal{L}$ in $D$ with mass $m$ and intensity $1/2$ can be coupled together in such way that the outer boundaries of the outermost clusters of $\LL$ are the loops of $\Gamma$. 
\end{theorem} 

Let us now give an overview of the proof of Thereom \ref{theorem_mSLE_4}, Theorem \ref{theorem_mCLE_4} and Theorem \ref{prop_mCLE4_mBLS}. We will then discuss questions that arise from these results and their proof.

\subsection{Outline of the proof of Theorem \ref{theorem_mSLE_4} and Theorem \ref{theorem_mCLE_4}}

The proof of Theorem \ref{theorem_mSLE_4}, given in Section \ref{sec_mSLE4}, is composed of two main steps. Note that by definition, under $\tilde \PP$ as defined via \eqref{RN_mSLE4}, the marginal law of $h+\phi$ is that of a massive GFF in $D$ with mass $m$ and boundary conditions $F^{(D,a,b)}$. In the first step, we identify the conditional law under $\tilde \PP$ of $h+\phi$ given $\gamma([0,t])$, for $t \geq 0$. This involves some technical arguments but the main idea is simple. It can be seen as a ``conditional version" of the path-integral formalism for the massive GFF, as we now explain. Let $f$ be a smooth function with compact support and let $t \geq 0$. It suffices to compute the characteristic function of $(h+\phi, f)$ given $\gamma([0,t])$ under $\tilde \PP$. If the field $h+\phi$ were a pointwise defined function and thus no renormalization were needed to define its square, we would have that
\begin{align} \label{informal_cond_law}
    \tilde \EE [\exp(i(h+\phi, f)) \vert \gamma([0,t])] 
    &= \frac{1}{\mathcal{Z}_t} \EE [ \exp(i(h+\phi, f)) \exp ( -\frac{1}{2} \int_{D} m^2(z) (h+\phi)^2(z) dz )\vert \gamma([0,t]) ] \nonumber \\
    &= \frac{1}{\mathcal{Z}_t} \EE [ \exp(i(h_t+\phi_t, f)) \exp ( -\frac{1}{2} \int_{D_t} m^2(z) (h_t+\phi_t)^2(z) dz ) \vert \gamma([0,t]) 
    ],
\end{align}
using the fact that under $\PP$, $h+\phi$ and $\gamma$ are coupled in such way that, conditionally on $\gamma([0,t])$, $h+\phi=h_t+\phi_t$ where $h_t+\phi_t$ is a GFF in $D_t := D \setminus \gamma([0,t])$ with boundary conditions $F^{(D_t,\gamma(t),b)}$. This coupling will be recalled more precisely in Section \ref{sec_GFF}. Similarly, we would have that
\begin{align} \label{informal_Zt}
    \mathcal{Z}_t &= \EE [ \exp ( -\frac{1}{2} \int_{D} m^2(z) (h+\phi)^2(z) dz )\vert \gamma([0,t])]
    = \EE [ \exp ( -\frac{1}{2} \int_{D_t} m^2(z) (h_t+\phi_t)^2(z) dz ) \vert \gamma([0,t])].
\end{align}
Now, observe that, since conditionally on $\gamma([0,t])$, $h_t+\phi_t$ is a GFF in $D_t$ with boundary conditions $F^{D_t,\gamma(t),b}$, according to \eqref{RN_intro}, the random variable
\begin{equation*}
    \frac{1}{\mathcal{Z}_t}\exp \bigg( -\frac{1}{2} \int_{D_t} m^2(z) (h_t+\phi_t)^2(z) dz \bigg)
\end{equation*}
is the Radon-Nikodym derivative of the law of the massive GFF in $D_t$ with mass $m$ and boundary conditions $F^{(D_t,\gamma(t),b)}$, with respect to the law of the GFF in $D_t$ with the same boundary boundary conditions. Going back to \eqref{informal_cond_law}, this would complete the proof. 

Of course, these computations are only formal since $h$ is not a pointwise defined function. Nevertheless, when introducing the renormalization of $h^2$ mentioned in the discussion below \eqref{RN_intro}, they can be made rigorous. The equalities \eqref{informal_cond_law} and \eqref{informal_Zt} actually hold true when considering the renormalized squares $:(h+\phi)^2:$ and $:(h_t+\phi_t)^2:$ of $h+\phi$ and $h_t+\phi_t$, but up to a correction term due to boundary effects. This term is the same in both equalities and thus is cancelled out.

The second step of the proof of Theorem \ref{theorem_mSLE_4} is devoted to the identification of the marginal law of $\gamma$ under $\tilde \PP$. This relies on a martingale characterization of massive SLE$_4$ shown in \cite{mHE} and on the fact that under $\tilde \PP$, $h+\phi$ and $\gamma$ are coupled as described in the statement of Theorem \ref{theorem_mSLE_4}. This fact is used to express the martingale characterizing massive SLE$_4$ as an observable of the field $h+\phi$ conditioned on $\gamma([0,t])$, for $t \geq 0$. This is possible thanks to the Markov property of the massive GFF and the explicit expression for its covariance.

The strategy for the proof of Theorem \ref{theorem_mCLE_4} is the same as that of Theorem \ref{theorem_mSLE_4}, except that we must deal with more technicalities. This is due to the fact that under $\PP$ and when regularizing the GFF $h$ to define its square, conditionally on $\Gamma$, the regularized GFFs $(h_{j,\eps})_j$ obtained via the decomposition of $h$ are not supported in the interior of the loops of $\Gamma$. This creates correlations and we must ensure that they disappear when we remove the regularization. To identify the conditional law of the variables $(\xi_j)_j$ given $\Gamma$ under $\tilde \PP$, we simply use the symmetry of the Wick square of the GFF and of the conditional laws of $(h_j)_j$ and $(\xi_j)_j$ given $\Gamma$ under $\PP$.

Finally, to prove Theorem \ref{prop_mCLE4_mBLS}, we first compute the Radon-Nikodym derivative of the law of the massive Brownian loop soup with respect to that of the Brownian loop soup. We then use this Radon-Nikodym derivative to reweight the coupling of \cite{Qian_GFF_BLS} between a CLE$_4$, a GFF with Dirichlet boundary conditions and a Brownian loop soup with intensity $1/2$ in a similar way as in Theorem \ref{theorem_mCLE_4}. We show that we thus obtain a coupling between a massive CLE$_4$, a massive GFF with Dirichlet boundary conditions and a massive Brownian loop soup with intensity $1/2$, which yields Theorem \ref{prop_mCLE4_mBLS}.

We will prove Theorem \ref{theorem_mSLE_4}, Theorem \ref{theorem_mCLE_4} and Theorem \ref{prop_mCLE4_mBLS} under the assumptions that the domain $D$ is bounded and that the mass $m:D \to \mathbb{R}_{+}$ is a bounded and continuous function. Conformal covariance then extends the results to the stated generality. However, our set of allowable masses is probably not the optimal one. This will be discussed in Section \ref{sec_assumptions}.

\subsection{Open problems}

\begin{enumerate}[leftmargin=*]

    \item The path-integral formalism can be used to define other field theories than the massive GFF, called interacting quantum field theories (QFT). In this approach, an interacting QFT with interaction potential $V$ is formally defined via 
    \begin{equation} \label{RN_PV}
        \frac{\der \PP_{\operatorname{V}}}{\der \PP_{\operatorname{GFF}}}(h)=\frac{1}{\mathcal{Z}_V} \exp \bigg( -\int  V(h(z)) dz \bigg).
    \end{equation}
    The above definition of $\PP_{\operatorname{V}}$ is only formal since the GFF is not defined pointwise and thus, in order to make sense of $V(h(z))$, one must introduce a renormalization procedure. However, note the analogy between \eqref{RN_intro} and \eqref{RN_PV}. This raises the following question: can the same ideas as those behind Theorem \ref{theorem_mSLE_4} and Theorem \ref{theorem_mCLE_4} be used to construct the level lines of interacting QFTs, provided that \eqref{RN_PV} can be defined rigorously by renormalization? Examples include the sine-Gordon field in the subcritical regime $\beta < 4\pi$ \cite{sG_1}, the $\Phi_{2}^4$-theory \cite{SimonEQFT} and the $\exp(\Phi)_2$-theory \cite{QFT}.

    \item Some fields of interest, such as the sine-Gordon field with $\beta < 6\pi$ \cite{max_sG} or those constructed via stochastic quantization in the regime where the Da Prato--Debussche trick applies \cite{DaPrato_SQ}, can be decomposed as a masssive GFF plus a more regular field, at least when these fields are defined on the torus. This more regular field is typically non-Gaussian and coupled to the massive GFF. Can we extend the massive GFF level line coupling to such regular perturbations?

    \item In the massless case, an SLE$_{\kappa}$ curve can be coupled to a GFF with appropriate boundary conditions for $\kappa \in (0,8)$ \cite{Dub_SLE, IG_1}. Can one construct a family of massive SLE$_{\kappa}$, $\kappa \in (0,8)$, that can be coupled to a massive GFF with appropriate boundary conditions? At a formal level, for each $\kappa$, it is not difficult to obtain an expression for the SDE satisfied by the driving function of this variant of SLE$_{\kappa}$ but one must still show that this SDE has a unique (weak) solution. If this can be achieved, then one can ask: does this family of massive SLE$_{\kappa}$ satisfy a large deviation principle similar to that of SLE$_{\kappa}$, $\kappa \in (0,8)$ \cite{LDP_SLE_1,LDP_SLE_2}?

    \item In a similar spirit, what are the massive CLE$_{\kappa}$ for $\kappa \neq 4$? Outermost boundaries of outermost clusters of a massive Brownian loop soup with intensity with $c(\kappa) \in (0,1/2]$ give a natural candidate family for $\kappa \in (8/3,4]$. Alternatively, in analogy to the non-massive case \cite{CLE_Markovian}, do a conformal covariance property and a restriction property characterize a family of probability measures on loops that can be seen as massive CLE$_{\kappa}$, for $\kappa \in (8/3,8)$? This family should be large: this is related to the fact that there is more than one way to define the massive version of an SLE$_{\kappa}$ curve \cite{off_SLE}.
\end{enumerate}

\paragraph{Acknowledgements.} The author is indebted to Avelio Sep\'ulveda for pointing out that Theorem \ref{theorem_mCLE_4} should yield a proof of Camia's conjecture. The author is grateful to Ellen Powell for her guidance and her careful reading of an earlier version of the paper. The author also thanks Tyler Helmuth for discussions at various stages of this project.

\section{Background} \label{sec_background}

\subsection{Schramm-Loewner evolutions} \label{sec_SLE}

\subsubsection{The Loewner equation and Schramm-Loewner evolutions}

Denote by $\HH$ the complex upper-half plane $\{ z \in \mathbb{C}: \Im(z) > 0 \}$ and let $\gamma: [0,\infty) \to \overline{\mathbb{H}}$ be a non-self-crossing curve targeting $\infty$ and such that $\gamma(0)=0$. For $t \geq 0$, let $K_t$ be the hull generated by $\gamma([0,t])$, that is $\HH \setminus K_t$ is the unbounded connected component of $\HH \setminus \gamma([0,t])$. In the case where $\gamma([0,t])$ is non-self-touching, $K_t$ is simply given by $\gamma([0,t])$. For each $t \geq 0$, it is easy to see that there exists a unique conformal $g_t: \HH \setminus K_t \to \HH$ satisfying the normalization $g_t(\infty) = \infty$ and such that $\lim_{z \to \infty} (g_t(z) - z) = 0$. It can then be proved that $g_t$ satisfies the asymptotic
\begin{equation*}
    g_t(z) = z + \frac{a_1(t)}{z} + O(\vert z \vert^{-2}) \quad \text{as } \vert z \vert \to \infty.
\end{equation*}
The coefficient $a_1(t)$ is equal to $\text{hcap}(K_t)$, the half-plane capacity of $K_t$, which, roughly speaking, is a measure of the size of $K_t$ seen from $\infty$. Moreover, one can show that $a_1(0)=0$ and that $t \mapsto a_1(t)$ is continuous and strictly increasing. Therefore, the curve $\gamma$ can be reparametrized in such a way that at each time $t$, $a_1(t) = 2t$. $\gamma$ is then said to be parameterized by half-plane capacity.

In this time-reparametrization and with the normalization of $g_t$ just described, it is known that there exists a unique real-valued function $t \mapsto W_t$, called the driving function, such that the following equation, called the Loewner equation, is satisfied:
\begin{equation} \label{eq_Loewner}
    \partial_t g_t(z) = \frac{2}{g_t(z) - W_t},  \quad g_0(z)=z, \quad \text{for all $z \in \HH \setminus K_t$}.
\end{equation}
Indeed, it can be shown that $g_t$ extends continuously to $\gamma(t)$ and setting $W_t = g_t(\gamma(t))$ yields the above equation, see e.g. \cite[Chapter~4]{book_Lawler} and \cite[Chapter~4]{book_SLE}.

Conversely, given a continuous and real-valued function $t \mapsto W_t$, one can construct a locally growing family of hulls $(K_t)_t$ by solving the equation \eqref{eq_Loewner}. Under additional assumptions on the function $t \mapsto W_t$, the family of hulls obtained using \eqref{eq_Loewner} is generated by a curve, in the sense explained above \cite{slit_Loewner}.

Schramm-Loewner evolutions, or SLE for short, are random Loewner chains introduced by Schramm \cite{Schramm_SLE}. For $\kappa \geq 0$, SLE$_\kappa$ is the Loewner chain obtained by considering the Loewner equation \eqref{eq_Loewner} with driving function $W_t = \sqrt{\kappa}B_t$, where $(B_t, t \geq 0)$ is a standard one-dimensional Brownian motion. As such, SLE$_{\kappa}$ is defined in $\HH$ but, thanks to the conformal invariance of the Loewner equation, SLE$_{\kappa}$ can be defined in any simply connected domain $D \subset \mathbb{C}$ with two marked boundary points $a, b \in \partial D$ by considering a conformal map $\varphi: D \to \HH$ with $\varphi(a)=0$ and $\varphi(b) = \infty$ and taking the image of SLE$_{\kappa}$ in $\HH$ by $\varphi^{-1}$. In particular, SLE$_{\kappa}$ is conformally invariant and it turns out that this conformal invariance property together with a certain domain Markov property characterize the family (SLE$_{\kappa}, \kappa \geq 0)$. In what follows, we will be interested in the special case $\kappa =4$. SLE$_4$ can be shown to be almost surely generated by a simple and continuous curve that is transient and whose Hausdorff dimension is $3/2$. For a proof of these facts, we refer the reader to \cite[Chapter~5]{book_SLE} and references therein.

\subsubsection{Massive SLE$_4$} \label{sec_def_mSLE4}

Massive SLE$_4$ is a massive version of SLE$_4$ originally introduced in \cite{off_SLE}. There are many possible ways to define massive versions of SLE$_4$ but we will use the name massive SLE$_4$ for this particular one, as in \cite{off_SLE}.

When mapped to the upper-half plane via a conformal map, a random curve whose law is massive SLE$_4$ can be described via its driving function as follows. Let $D \subset \mathbb{C}$ be a bounded, open and simply connected domain and let $a,b \in \partial D$. Denote by $\partial D^{+}$, respectively $\partial D^{-}$, the clockwise, respectively counterclockwise, oriented boundary arc from $a$ to $b$. Let $m: D \to \mathbb{R}_{+}$ be a bounded and continuous function. Let $\varphi: D \to \HH$ be a conformal map such that $\varphi(a)=0$ and $\varphi(b)=\infty$. Then, a random curve $\gamma$ has the law of massive SLE$_4$ in $D$ from $a$ to $b$ with mass $m$ if the driving function $(W_t, t \geq 0)$ of $\varphi(\gamma)$, when parametrized by half-plane capacity, satisfies the SDE
\begin{equation} \label{SDE_mSLE4}
    dW_t  = 2dB_t - 2\pi \bigg( \int_{D_t} m^2(z)P_t^m(z)h_t(z) dz \bigg)dt, \quad W_0=0,
\end{equation}
where $(B_t, t \geq 0)$ is a one-dimensional standard Brownian motion. Above, for $z \in D_t$, we have set
\begin{equation*}
    P_t^m(z) := \frac{1}{\pi}\Im\bigg( \frac{-1}{g_t(\varphi(z))-W_t}\bigg) - \int_{D_t} m^2(w) \frac{1}{\pi}\Im\bigg( \frac{-1}{g_t(\varphi(w))-W_t}\bigg) G_t^m(z,w) dw
\end{equation*}
where $G_t^m$ is the massive Green function with mass $m$ in $D_t$. In \eqref{SDE_mSLE4}, the function $h_t$ is the unique harmonic function with boundary values $1/2$ on $\partial D^{-}$ and the left side of $\gamma([0,t])$ and $-1/2$ on $\partial D^{-}$ and the right side of $\gamma([0,t])$. It was shown in \cite{mHE} that the SDE \eqref{SDE_mSLE4} has a unique weak solution whose law is absolutely continuous with respect to that of $(2B_t, t \geq 0)$. This implies that the law of massive SLE$_4$ in $D$ from $a$ to $b$ with mass $m$ is absolutely continuous with respect to that of SLE$_4$ in $D$ from $a$ to $b$.

Moreover, using the conformal covariance of the drift term in the SDE \eqref{SDE_mSLE4}, one can extend the definition of massive SLE$_4$ to pairs $(\tilde D, \tilde m)$ where $\tilde D \subset \mathbb{C}$ is the image under a conformal map $\varphi$ of a bounded, open and simply connected domain $D \subset \mathbb{C}$ and, for $w \in \tilde D$, $\tilde m(w)^2=\vert (\varphi^{-1})'(w)\vert ^2 m(\varphi^{-1}(w))^2$, where $m: D \to \mathbb{R}$ is a bounded and continuous function.

In the proof of Theorem \ref{theorem_mSLE_4}, we will not use this explicit definition of massive SLE$_4$ via its driving function but rather the fact that massive SLE$_4$ can be characterized as the unique random curve for which a certain observable is a martingale. This characterization will be recalled in the course of the proof of Theorem \ref{theorem_mSLE_4}, see Section \ref{sec_marginal_law_gamma}.

\subsection{Conformal loop ensembles} \label{sec_CLE}

Conformal loop ensembles (CLE) are a family of probability distributions on countable ensembles of non-nested loops (closed curves) in open and simply connected domains of the complex plane \cite{CLE_She, CLE_Markovian}. As in \cite[Section~2.1]{CLE_Markovian}, we define a simple loop in the plane to be the image of the unit circle under a continuous and injective map. With this definition, a loop $\ell$ with time-duration $\tau(\ell)$ is equivalent to the loop $\tilde \ell$ with time-duration $\tau(\tilde \ell)$ if there exists a bijective map $\psi: [0,\tau(\ell)] \to [0,\tau(\tilde \ell)]$ such that for any $t \in [0,\tau(\ell)]$, $\ell(\psi(t)) = \tilde \ell(t)$. This defines a space of loops, that we denote by $\operatorname{Loop}(\mathbb{C})$, and we endow it with the $\sigma$-field $\Sigma$ generated by all the events of the form $\{ O \subset \operatorname{Int}(\ell) \}$, where $O$ spans the set of open subsets of $\mathbb{C}$ and for a loop $\ell$, $\operatorname{Int}(\ell)$ denotes its interior. Note that $\operatorname{Loop}(\mathbb{C})$ can also be turned into a metric space, for example by equipping it with the distance induced by the supremum norm. A countable collection $\Gamma=(\ell_j)_{j \in J}$ of loops in $\operatorname{Loop}(\mathbb{C})$ can be identified with the point-measure
\begin{equation*}
    \mu_{\Gamma} = \sum_{j \in J} \delta_{\ell_j}.
\end{equation*}
The space of countable collections of loops is then naturally equipped with the $\sigma$-field generated by the sets $\{ \Gamma: \mu_{\Gamma}(A) = k \}$ where $A \in \Sigma$ and $k \geq 0$.

CLE$_{\kappa}$ are random countable collections of loops indexed by a parameter $\kappa \in (8/3,8)$. $\operatorname{CLE}_{\kappa}$ is connected to $\operatorname{SLE}_{\kappa}$ via the so-called branching tree construction \cite{CLE_She}. The geometry of the loops of a $\operatorname{CLE}_{\kappa}$ in an open and simply connected domain $D$ depends on the value of $\kappa$: when $\kappa \in (8/3,4]$, these loops are almost surely simple loops that do not intersect each other or the boundary of $D$; on the contrary, when $\kappa \in (4,8)$, they are almost surely non-simple but non-self-crossing and they may touch (but not cross) the boundary of $D$ and each other. Another important property of CLE$_\kappa$ is their conformal invariance in law: if $\varphi: D \to \tilde D$ is a conformal map between two open and simply connected domains of $\mathbb{C}$ and $\Gamma$ is a CLE$_{\kappa}$ in $D$, then $\varphi(\Gamma)$ has the law of a CLE$_\kappa$ in $\tilde D$.

\subsection{Level lines of the Gaussian free field} \label{sec_GFF}

\subsubsection{Definition of the GFF} \label{sec_def_GFF}

Let $D \subset \mathbb{C}$ be an open and simply connected domain. The Gaussian free field (GFF) in $D$ with Dirichlet boundary conditions is a centered Gaussian process $h$ indexed by smooth functions with compact support in $D$ whose covariance is given by, for $f$ and $g$ two such functions,
\begin{equation*}
    \EE [(h,f)(h,g)] = \int_{D \times D} f(z)G_{D}(z,w)g(w) dzdw.
\end{equation*}
Above, $G_{D}$ is the Green function in $D$: this is the inverse (in the sense of distributions) of the operator $-\Delta$ in $D$ with Dirichlet boundary conditions. As the Green function blows up on the diagonal, the GFF is not defined pointwise but is instead a generalized function.

One can also define a GFF with specified boundary conditions. Let $D \subset \mathbb{C}$ be an open and simply connected domain. Let $f : \partial D \to \mathbb{R}$ be a bounded function and let $\phi_{f}$ be its unique harmonic extension in $D$. That is, $\phi_{f}$ is the unique solution to the boundary value problem
\begin{equation*}
    \begin{cases}
        &-\Delta u(z)=0, \quad z \in D \\
        &u(z)=f(z), \quad z \in \partial D.
    \end{cases}
\end{equation*}
Then, a GFF in $D$ is said to have boundary conditions given by $f$ if it has the same law as $h+\phi_{f}$ where $h$ has the law of a GFF in $D$ with Dirichlet boundary conditions.

\subsubsection{Coupling between an SLE$_4$ curve and a GFF} \label{sec_coupling_SLE}

Set $\lambda := \sqrt{\pi/8}$. Let $D \subset \mathbb{C}$ be an open and simply connected domain and let $a,b \in \partial D$. Denote by $\partial D^{+}$, respectively $\partial D^{-}$, the clockwise-oriented, respectively counterclockwise-oriented, boundary arc $(ab)$. Let $\phi: D \to \mathbb{R}$ be the unique harmonic function in $D$ with boundary conditions $\lambda$ on $\partial D^{+}$ and $-\lambda$ on $\partial D^{-}$. Then, there exists a coupling $(h+\phi, \gamma)$ between a GFF $h+\phi$ in $D$ with boundary conditions $\phi$ and an SLE$_4$ curve in $D$ from $a$ to $b$ \cite{Dub_SLE, IG_1}. In this coupling, for any $t \geq 0$, conditionally on $\gamma([0,t])$, $h+\phi=h_t+\phi_t$ where $h_t$ is a GFF with Dirichlet boundary conditions in $D_t:=D \setminus \gamma([0,t])$ and $\phi_t$ is the unique harmonic function in $D_t$ with boundary conditions $\lambda$ on $\partial D^{+}$ and the left side of $\gamma([0,t])$ and $-\lambda$ on $\partial D^{-}$ and the right side of $\gamma([0,t])$. A similar decomposition also holds when conditioning on $\gamma([0,\tau])$, where $\tau$ is a stopping time for the filtration generated by $\gamma$. 

When $t=\infty$, since $\gamma([0,\infty))$ is almost surely a simple curve which does intersect the boundary of $D$ except at its endpoints, $D \setminus \gamma([0,\infty))$ is almost surely composed of two simply connected components. Denote by $D_1$, respectively $D_2$, the component on the left, respectively right, of $\gamma([0,\infty))$. Then, conditionally on $\gamma([0,\infty))$, $h=h_1+\phi_1+h_{2}+\phi_2$ where $h_1$ and $h_2$ are independent GFFs with Dirichlet boundary conditions in $D_1$ and $D_2$ and $\phi_1$, respectively $\phi_2$ is almost surely equal to $\lambda$, respectively $0$, in $D_1$ and to $0$, respectively $-\lambda$, in $D_2$.

These decompositions for $t<\infty$ and $t=\infty$ explain why in this coupling, the curve $\gamma$ is sometimes called the level line of the field $h+\phi$. In this coupling, it can also be shown that the curve $\gamma$ is in fact measurable with respect to the GFF $h$, see \cite{IG_1}.

\subsubsection{Coupling between an CLE$_4$ and a GFF} \label{sec_coupling_CLE}

As before, set $\lambda := \sqrt{\pi/8}$. Let $D \subset \mathbb{C}$ be an open and simply connected domain. Then, there exists a coupling $(h,\Gamma)$ between a GFF $h$ in $D$ with Dirichlet boundary conditions and a CLE$_4$ $\Gamma$ in $D$ \cite{BTLS}. In this coupling, conditionally on $\Gamma$, $h=\sum_j h_j + \xi_j$ where the sum runs over the loops $(L_j)_j$ of $\Gamma$ and
\begin{itemize}
    \item $(h_j)_j$ are independent GFFs with Dirichlet boundary conditions in the interiors of  $(L_j)_j$;
    \item $(\xi_j)_j$ are independent random variables with $\PP(\xi_j=2\lambda \vert \Gamma)=\PP(\xi_j=-2\lambda \vert \Gamma)=1/2 \; \forall j$;
    \item $(h_j)_j$ and $(\xi_j)_j$ are independent.
\end{itemize}
This decomposition of $h$ given $\Gamma$ is the reason why the loops of $\Gamma$ are sometimes called the level lines of the field $h$. In this coupling, it can also be shown that the CLE$_4$ $\Gamma$, the fields $(h_j)_j$ and the labels $(\xi_j)_j$ are in fact measurable with respect to the GFF $h$ \cite{slides_CLE, BTLS}.

\subsection{The Brownian loop measure and Le Jan's isomorphism theorem} \label{sec_BLS_LeJan}

\subsubsection{The Brownian loop measure and the Brownian loop soup} \label{sec_BLS}

Let $D \subset \mathbb{C}$ be a bounded, open and simply connected domain. We denote by $\PP_t^{z,w}$ the bridge probability from $z$ to $w$ of length $t$ for a two-dimensional standard Brownian motion $(B_s, s \geq 0)$. We also let $\tau_{\partial D}$ denote the first hitting time of $\partial D$ by $(B_s, s \geq 0)$. The rooted Brownian loop measure $\mu_{R,D}$ in $D$, introduced in \cite{Werner_BLS},  is defined as
\begin{equation*}
    \mu_{R,D}(\cdot) := \int_{0}^{\infty} \int_{D} \frac{1}{t} \frac{1}{2\pi t} \PP_t^{z,z}(\cdot, \tau_{\partial D} > t) dz dt.
\end{equation*}
This is a measure on rooted loops. One can define an equivalence relation between such loops as follows. For $\ell$ a rooted loop, denote by $\tau(\ell)$ its lifetime. Then, for $u \in \mathbb{R}$, $\theta_u\ell: t \mapsto \ell(u+t \, \operatorname{mod} \tau(\ell))$ is again a loop. We then say that two loops $\ell$ and $\tilde \ell$ are equivalent if there exists $u \in \mathbb{R}$ such that $\ell = \theta_u \tilde \ell$. The quotient of the rooted loop measure $\mu_{R, D}$ under this equivalence relation is the Brownian loop measure $\mu_{D}$ in $D$. This measure satisfies two important properties:
\begin{itemize}
    \item the restriction property: if $\tilde D \subset D$, then $\mu_{\tilde D}(\cdot) = \mu_{D}(\cdot \mathbb{I}_{\ell \subset \tilde D})$;
    \item conformal invariance: if $f: \Omega \to D$ is a conformal map, then $f \circ \mu_{\Omega} = \mu_{D}$.
\end{itemize}
For $\alpha >0$, one can then sample a Poisson point process with intensity $\alpha \mu_{D}$ which gives rise to a countable collection of unrooted loops in $D$. This process, introduced in \cite{Werner_BLS}, is called a Brownian loop soup in $D$ with intensity $\alpha$. The Brownian loop soup turns out to be intimately connected to CLE$_{\kappa}$ for $\kappa \in (8/3,4]$: a CLE$_\kappa$ can be constructed out a Brownian loop soup with the correct intensity \cite{CLE_Markovian}. In particular, a CLE$_4$ $\Gamma$ can be coupled to a Brownian loop soup $\LL$ with intensity $1/2$ in such way that the outer boundaries of the outermost clusters of loops of $\LL$ are the loops of $\Gamma$. Here, we say that two loops $\ell$ and $\tilde \ell$ belong to the same cluster of loops of $\LL$ if there exists a sequence $(\ell_k)_{k=0}^{n}$ of loops in $\LL$ such that $\ell_0=\ell$, $\ell_n=\tilde \ell$ and for any $k$, $\ell_k \cap \ell_{k+1}\neq \emptyset$.

An important observable of a Brownian loop soup $\LL$ in $D$ with intensity $\alpha > 0$ is its renormalized occupation-time field $:L:$ which, in some sense, fully characterizes $\LL$. For $z \in D$, $:L(z):$ should be thought of as the total amount of time that loops in $\LL$ spend at $z$. However, since this quantity is in fact almost surely infinite, $:L:$ must be constructed via renormalization and is not defined pointwise. The construction of $:L:$ goes as follows. Let $\eps > 0$ and set, for a bounded function $f: D \to \mathbb{R}$,
\begin{equation*}
	\int_{D} :L_{\eps}(z): f(z) dz := \sum_{\ell \in \LL: \tau(\ell) \geq \eps} \int_{0}^{\tau(\ell)}f(\ell(t)) dt - \alpha\mu_{D}\bigg(\mathbb{I}_{\tau(\ell) \geq \eps} \int_{0}^{\tau(\ell)}f(\ell(t)) dt\bigg) 
\end{equation*}
where we note that $\alpha\mu_{D}(\mathbb{I}_{\tau(\ell) \geq \eps}\int_{0}^{\tau(\ell)}f(\ell(t)) dt) = \EE[\sum_{\ell \in \LL: \tau(\ell) \geq \eps} \int_{0}^{\tau(\ell)}f(\ell(t)) dt]$. Then, by \cite[Section~10.2]{LeJan}, as $\eps$ tends to $0$, $\langle :L_{\eps}(\LL):, f \rangle$ converges in $L^2(\PP)$ and almost surely to a random variable that we denote by $\int_{D} f(z) :L(z): dz$. One can then make sense of $:L:$ as a random field indexed by bounded functions in $D$ and this field is what we call the renormalized ocupation-time field of $\LL$.

\subsubsection{Wick square of the Gaussian free field} \label{sec_Wick}

As explained in Section \ref{sec_def_GFF}, the GFF is only a generalized function and its square is thus a priori ill-defined. However, an interesting object that corresponds in some sense to its square can still be obtained by an appropriate regularization and renormalization procedure. To describe it, let us first introduce the circle-average approximation of the GFF. Let $D \subset \mathbb{C}$ be an open, bounded and simply connected domain and let $h$ be a GFF in $D$ with Dirichlet boundary conditions. For $z \in D$ and $\eps > 0$, denote by $B(z,\eps)$ the ball of radius $\eps$ centered at $z$ and let $\rho_{\eps}^z$ be the uniform measure on $\partial B(z,\eps)$ seen from $z$. We define the random variable $h_{\eps}(z)$ by
\begin{equation*}
    h_{\eps}(z) := (h, \rho_{\eps}^z).
\end{equation*}
Even though $\rho_{\eps}^z$ is not a smooth function, this random variable is always well-defined and the process $(h_{\eps}(z), z \in D, \eps > 0)$ is in fact jointly H\"{o}lder continuous in $(z,\eps)$, see \cite[Section~3.3.4]{book_GFF}.

A natural way to define the square of $h$ would be to take the limit of $h_{\eps}(z)^2$ as $\eps$ tends to $0$. However, $h_{\eps}(z)^2$ has expectation of order $-\log \eps$, which blows up as $\eps$ tends to $0$. Therefore, to obtain an interesting object when taking the limit, a divergent counterterm must be introduced. This counterterm is chosen such that the limit field, denoted by $:h^2:$, is a centered field. Here, we will not need to construct the field $:h^2:$ itself. Instead, we must only be able to make sense of the random variable $(:h^2:,m^2)$, or more generally of $(:(h+\phi)^2:, m^2)$, where $\phi$ will correspond to the boundary conditions of $h$ and $m^2$ will be the mass of the massive GFF. This is achieved thanks to the following lemma. Below, for $\phi:D \to \mathbb{R}$ a bounded harmonic function, $z \in D$ and $\eps > 0$, we use the notation
\begin{equation} \label{notation_h+phi_square}
    :(h_{\eps}(z)+\phi_{\eps}(z))^2: \; \overset{\operatorname{not}}{=} h_{\eps}(z)^2 + 2h_{\eps}(z)\phi_{\eps}(z) +\phi_{\eps}(z)^2 - \EE[h_{\eps}(z)^2]
\end{equation}
where $\phi_{\eps}(z)$ the circle average approximation of $\phi(z)$, that is
\begin{equation} \label{def_phi_eps}
    \phi_{\eps}(z) := \int \phi(w) \rho_{\eps}^z(dw).
\end{equation}
Notice that if $B(z,\eps) \subset D$, then, by harmonicity, $\phi_{\eps}(z)=\phi(z)$. We also denote by $\operatorname{dim}_H(A)$ the Hausdorff dimension of the set $A$.

\begin{lemma} \label{lemma_L2_cvg_Wick_boundary}
Let $D \subset \mathbb{C}$ be a bounded, open and simply connected domain and let $m: D \to \mathbb{R}_{+}$ be a bounded and continuous function. Let $h$ be a GFF with Dirichlet boundary conditions in $D$ and $\phi: D \to \mathbb{R}$ be a harmonic function bounded by some constant $b_{\phi} > 0$. Then
\begin{equation*}
    \lim_{\eps \to 0} \int_{D} m^2(z) :(h_{\eps}(z)+\phi_{\eps}(z))^2:dz = \int_{D} m^2(z):h(z)^2:dz + 2(h,m^2\phi) + \int_{D} m^2(z)\phi(z)^2dz
\end{equation*}
where the limit is in $\operatorname{L}^2(\PP)$. Above, $(h,m^2\phi)$ denotes the GFF $h$ tested against the function $z \mapsto m^2(z)\phi(z)$ and
\begin{equation*}
    \int_{D} m^2(z):h(z)^2:dz = \lim_{\eps \to 0} \int_{D} m^2(z) :h_{\eps}(z)^2: dz
\end{equation*}
where the limit is in $\operatorname{L}^2(\PP)$. Moreover, for any $b \in (0,2-\operatorname{dim}_{H}(\partial D))$, there exists $C>0$ such that, for any $0<\eps < 1/2$,
\begin{equation} \label{rate_cvg_L2_Wick_boundary}
    \bigg \| \int_{D} m^2(z) :(h+\phi)^2(z): dz - \int_{D} m^2(z) :(h_{\eps}(z)+\phi_{\eps}(z))^2: dz \bigg\|_{\operatorname{L}^2(\PP)} \leq C\eps^{b}.
\end{equation}
\end{lemma}

The proof of Lemma \ref{lemma_L2_cvg_Wick_boundary} is given in Appendix \ref{sec_app_1}. The assumption that $\phi$ is a harmonic function is not really needed but it simplifies the proof and will always be satisfied in the cases in which we will be interested. Let us also stress that in the above statement, the notation $\int_{D} m^2(z) :h(z)^2: dz$ is only formal since $:h^2:$ is not defined pointwise. 

For convenience, as in \eqref{rate_cvg_L2_Wick_boundary}, we will sometimes use the formal notation
\begin{equation} \label{notation_Wick_boundary}
    \int_{D} m^2(z) :(h+\phi)^2(z): dz := \int_{D} m^2(z):h(z)^2:dz + 2(h,m^2\phi) + \int_{D} m^2(z)\phi(z)^2dz.
\end{equation}
Also, if the function $m: D \to \mathbb{R}_{+}$ is of the form $\vert \varphi' \vert (\hat m \circ \varphi)$ where $\varphi: D \to \hat D$ is a conformal isomorphism, $\hat D \subset \mathbb{C}$ is an open, bounded and simply connected domain and $\hat m: \hat D \to \mathbb{R}_{+}$ is a bounded and continuous function, we set
\begin{equation*}
    \int_{D} m^2(z) :(h+\phi)^2(z): dz = \int_{\hat D} \hat m^2(w) :(\hat h+ \hat \phi)^2(w): dw
\end{equation*}
with $h+\phi= (\hat h + \hat \phi) \circ \varphi^{-1}$.

\begin{remark} \label{remark_variance_Wick}
Thanks to the $L^2(\PP)$ convergence established in Lemma \ref{lemma_L2_cvg_Wick_boundary}, the variance of $(:h^2:,m^2)$ can be explicitly computed: as shown in Lemma \ref{lemma_variance_Wick} in Appendix \ref{sec_app_1},
\begin{equation} \label{EE_square_GFF}
    \EE[(:h^2:,m^2)^2] = 2\int_{D} m^2(z)G_{D}(z,w)^2m^2(w) dzdw
\end{equation}
where, as in Section \ref{sec_GFF}, $G_{D}$ is the Green function in $D$.
\end{remark}

\subsubsection{Le Jan's isomorphism theorem} \label{sec_LeJan}

The Brownian loop soup in $D$ is related to the Gaussian free field in $D$ with Dirichlet boundary conditions via Le Jan's isomorphism theorem \cite{LeJan, LeJan_2}, which relates the renormalized occupation-time field $:L:$ of the loop soup (see Section \ref{sec_BLS}) to the Wick square of the Gaussian free field in $D$.

\begin{proposition}[Le Jan's isomorphism] \label{prop_LeJanIso}
The renormalized occupation-time field $:L:$ of a Brownian loop soup in $D$ with intensity $1/2$ has the same law as half the Wick square $1/2 :h^2:$ of a Gaussian free field $h$ in $D$ with Dirichlet boundary conditions.
\end{proposition}

We also recall the following result, stated as Theorem 8 in \cite{LeJan}, that will be needed for the proof of Corollary \ref{cor_RN_mSLE4} and Corollary \ref{cor_RN_mCLE4}.

\begin{proposition} \label{prop_RN_loopGFF}
For a bounded and continuous function $m: D \to \mathbb{R}_{+}$ and a loop $\ell: [0,\tau(\ell)] \to D$, set $\langle \ell, m^2 \rangle := \int_{0}^{\tau(\ell)} m^2(\ell(t)) dt$. Then
\begin{equation*}
    \EE_{\operatorname{GFF}} \bigg[ \exp\bigg( -\frac{1}{2} \int_{D} m^2(z) :h^2(z):dz \bigg) \bigg] = \exp \bigg(\frac{1}{2} \mu_{D}(e^{-\langle \ell, m^2 \rangle} + \langle \ell, m^2 \rangle -1) \bigg) < \infty
\end{equation*}
where under $\EE_{\operatorname{GFF}}$, $h$ has the law of a GFF in $D$ with Dirichlet boundary conditions.
\end{proposition}

\begin{remark}
In \cite{LeJan}, Proposition \ref{prop_RN_loopGFF} is stated under the additional assumption that the function $m$ is smooth and has compact support in $D$ but the proof extends to the case where $m$ is a bounded function and $D$ a bounded domain.
\end{remark}

\subsubsection{Coupling between a GFF, a CLE$_4$ and a Brownian loop soup}
In view of Le Jan's isomorphism, the construction of CLE$_4$ using a Brownian loop soup with intensity $1/2$ explained in Section \ref{sec_BLS} and the coupling between a GFF and a CLE$_4$ described in Section \ref{sec_coupling_CLE}, it is natural to wonder whether there exists a coupling between a GFF, a CLE$_4$ and a Brownian loop soup with intensity $1/2$ such that all these relations hold. A positive answer to this question was given in \cite[Proposition~5]{Qian_GFF_BLS}. As this result is the starting point of the proof of Theorem \ref{prop_mCLE4_mBLS}, let us state it.

\begin{proposition}[Qian-Werner] \label{prop_coupling_GFF_BLS_CLE}
One can couple a Brownian loop soup $\LL$ in $D$ with intensity $1/2$, a CLE$_4$ $\Gamma$ in $D$ and a GFF $h$ in $D$ with Dirichlet boundary conditions in such a way that:
\begin{itemize}
	\item the loops of $\Gamma$ are the level lines of $h$ as in Section \ref{sec_coupling_CLE};
	\item the loops of $\Gamma$ are the outer boundaries of the outermost clusters of $\LL$;
	\item the renormalized occupation-time field $:L:$ of $\LL$ is exactly $\frac{1}{2}:h^2:$.
\end{itemize}
\end{proposition}

\subsection{The massive Gaussian free field} \label{sec_mGFF}

\subsubsection{Definition of the massive GFF} \label{sec_def_mGFF}

Let $D \subset \mathbb{C}$ be an open, bounded and simply connected domain and let $m: D \to \mathbb{R}_{+}$ be a bounded and continuous function. The massive GFF in $D$ with mass $m$ and Dirichlet boundary conditions is a centered Gaussian process $h$ indexed by smooth functions with compact support in $D$ whose covariance is given by, for $f$ and $g$ two such functions,
\begin{equation*}
    \EE [(h,f)(h,g)] = \int_{D \times D} f(z)G_{D}^m(z,w)g(w) dzdw.
\end{equation*}
Above, $G_{D}^m$ is the massive Green function in $D$ with mass $m$: this is the inverse (in the sense of distributions) of the operator $-\Delta+m^2$ in $D$ with Dirichlet boundary conditions. Like the GFF, the massive GFF is not defined pointwise but is instead a random generalized function. 

From a statistical mechanics point of view, the massive GFF can be seen as a GFF ``perturbed by a mass". The presence of this mass breaks the conformal invariance of the field, which is one of the main features of the GFF: the massive GFF is not conformally invariant but only conformally covariant, in the following sense. Let $\varphi: D \to \tilde D$ be conformal map. If $h$ is a massive GFF in $D$ with mass $m$ and Dirichlet boundary conditions, then the pushforward of $h$ via $\varphi$ is a massive GFF in $\tilde D$ with Dirichlet boundary conditions and mass $\tilde m$ given by, for $z \in \tilde D$,
\begin{equation} \label{mass_conformal}
    \tilde m^2(z) = \vert (\varphi^{-1})'(z) \vert^2 m^2(\varphi^{-1}(z)).
\end{equation}
This simply follows from the conformal covariance of the massive Green function $G_D^m$.

One can also define a massive GFF with specified boundary conditions. Let $D \subset \mathbb{C}$ be an open, bounded and simply connected domain and let $m: D \to \mathbb{R}_{+}$ be a bounded and continuous function. Let $f : \partial D \to \mathbb{R}$ be a bounded function and let $\phi_{f}^m$ be its unique massive harmonic extension in $D$. That is, $\phi_{f}^m$ is the unique solution to the boundary value problem
\begin{equation*}
    \begin{cases}
        &(-\Delta + m^2(z))u(z)=0, \quad z \in D \\
        &u(z)=f(z), \quad z \in \partial D.
    \end{cases}
\end{equation*}
Then, a massive GFF in $D$ with mass $m$ is said to have boundary conditions given by $f$ if it has the same law as $h+\phi_{f}^m$ where $h$ has the law of a massive GFF in $D$ with mass $m$ and Dirichlet boundary conditions. Note that since massive harmonic functions are conformally covariant, with the mass changing as in \eqref{mass_conformal} when conformally mapping the domain to another one, a massive GFF with specified boundary conditions is also conformally covariant, in the same sense as for a massive GFF with Dirichlet boundary conditions.

\subsubsection{Absolute continuity with respect to the GFF} \label{sec_abscont_mGFF}

Let $D \subset \mathbb{C}$ be an open, bounded and simply connected domain and let $m: D \to \mathbb{R}_{+}$ be a bounded and continuous function. Then the massive GFF in $D$ with mass $m$ and Dirichlet boundary conditions is absolutely continuous with respect to the GFF in $D$ with Dirichlet boundary conditions. The corresponding Radon-Nikodym derivative has an explicit expression, as we now detail. Let $\PP_{\operatorname{GFF}}^D$, respectively $\PP_{\operatorname{mGFF}}^D$, denote the law of the GFF in $D$ with Dirichlet boundary conditions, respectively of the massive GFF in $D$ with mass $m$ and Dirichlet boundary conditions. The Radon-Nikodym derivative of $\PP_{\operatorname{mGFF}}^D$ with respect to $\PP_{\operatorname{GFF}}^D$ is given by
\begin{equation} \label{RN_mGFF_Dirichlet}
    \frac{\der \PP_{\operatorname{mGFF}}^D}{\der \PP_{\operatorname{GFF}}^D}(h) = \frac{1}{\mathcal{Z}} \exp \bigg( -\frac{1}{2} \int_{D} m^2(z) :h^2(z): dz \bigg)
\end{equation}
where $\mathcal{Z}$ is a normalization constant and $\int_{D} m^2(z) :h^2(z): dz$ is as defined in Section \ref{sec_Wick}. See \cite[Lemma~3.10]{mGFF_RN} for a detailed proof of this fact. Let us nevertheless point out that the fact that the random variable on the right-hand side is $\sigma(h)$-measurable and almost strictly positive can be justified using Lemma \ref{lemma_L2_cvg_Wick} (see Lemma \ref{lemma_non0} in Appendix \ref{sec_app_1}). Moreover, by Proposition \ref{prop_RN_loopGFF}, this random variable is in $L^1(\PP_{\operatorname{GFF}}^D)$, so $\mathcal{Z}$ is finite and $\PP_{mGFF}^D$ is absolutely continuous with respect to $\PP_{GFF}^D$.

Note also that, by Proposition \ref{prop_RN_loopGFF} applied with $f=m^2$,
\begin{equation} \label{explicit_Z}
    \mathcal{Z} = \EE_{\operatorname{GFF}}^{D} \bigg[ \exp \bigg( -\frac{1}{2} \int_{D} m^2(z) :h^2(z): dz \bigg) \bigg] = \exp \bigg(\frac{1}{2} \mu_{D}\big(e^{-\langle \ell, m^2 \rangle} + \langle \ell, m^2 \rangle -1\big) \bigg)
\end{equation}
where, as in Section \ref{sec_BLS_LeJan}, $\mu_D$ is the Brownian loop measure in $D$.

Similarly, a massive GFF in $D$ with mass $m$ and specified boundary conditions is absolutely continuous with respect to the GFF in $D$ with the same boundary conditions. More precisely, let $f: \partial D \to \mathbb{R}$ be a bounded function and let $\phi_{f}$ be its unique harmonic extension in $D$. Then, if we define
\begin{equation} \label{RN_mGFF_phi}
    \frac{\der \PP_{\operatorname{mGFF}}^D}{\der \PP_{\operatorname{GFF}}^{D,f}}(h+\phi_{f}) = \frac{1}{\mathcal{Z}_f} \exp \bigg( -\frac{1}{2} \int_{D} m^2(z) :(h+\phi_{f})^2(z): dz \bigg)
\end{equation}
where $\mathcal{Z}_f$ is a normalization constant, the law of $h+\phi_f$ under $\PP_{mGFF}^{D,f}$ is that of a massive GFF in $D$ with mass $m$ and boundary conditions $f$. As above, the fact that the random variable on the right-hand side of \eqref{RN_mGFF_phi} is almost surely strictly positive and measurable with respect to $h+\phi$ follows from Lemma \ref{lemma_L2_cvg_Wick_boundary} and Lemma \ref{lemma_non0} in Appendix \ref{sec_app_1}.

To prove that under the measure $\PP_{\operatorname{mGFF}}^D$ defined via the change of measure \eqref{RN_mGFF_phi}, $h+\phi_{f}$ has the same law as $h+\phi_f^m$ where $h$ has the law of a massive GFF in $D$ with mass $m$ and Dirichlet boundary conditions and $\phi_f^m$ is the unique massive harmonic extension of $f$ in $D$, one can compute the Laplace transform of $h+\phi_{f}$ under $\PP_{\operatorname{mGFF}}^D$. Although this computation is rather straightforward, let us detail it as we could not find any reference in the existing literature.

\begin{proof}[Proof of \eqref{RN_mGFF_phi}]
Let $g$ be a smooth function with compact support in $D$. Then, we have that
\begin{align} \label{Laplace_mGFF_phi}
    & \EE_{\operatorname{mGFF}}^{D,f} \big[ \exp((h+\phi_f, g))\big] 
    = \frac{1}{\mathcal{Z}_f} \EE_{\operatorname{GFF}}^{D} \bigg[ \exp\bigg((h+\phi_f, g)-\frac{1}{2} \int_{D} m^2(z) :(h+\phi_{f})^2(z): dz \bigg) \bigg] \nonumber \\
    &= \frac{1}{\mathcal{Z}_f} \exp\bigg(\int_{D}\phi_{f}(z)g(z)-\frac{1}{2}m^2(z)\phi(z)^2 dz \bigg)   \EE_{\operatorname{GFF}}^{D} \bigg[ \exp\bigg((h, g-m^2\phi_{f}) -\frac{1}{2} \int_{D} m^2(z) :h^2(z): dz\bigg) \bigg] \nonumber \\
    &= \frac{\mathcal{Z}}{\mathcal{Z}_f} \exp\bigg(\int_{D}\phi_{f}(z)g(z)-\frac{1}{2}m^2(z)\phi(z)^2 dz \bigg) \EE_{\operatorname{mGFF}}^{D} \big[ \exp((h, g-m^2\phi_{f})) \big].
\end{align}
The fact that the two integrals above are finite follows from the fact that $D$ is a bounded domain and that $\phi_{f}$ and $m$ are bounded functions in $D$. Similarly, we have that
\begin{equation} \label{expression_Zf}
   \mathcal{Z}_f =
   \mathcal{Z} \exp \bigg( -\frac{1}{2} \int_{D} m^2(z)\phi_{f}(z)^2 dz \bigg) \EE_{\operatorname{mGFF}}^{D} \big[ \exp((h, m^2\phi_{f})) \big].
\end{equation}
Going back to \eqref{Laplace_mGFF_phi}, this yields that
\begin{align*}
    & \EE_{\operatorname{mGFF}}^{D,f} \big[ \exp((h+\phi_f, g))\big] \\
    & = \exp \bigg( \int_{D}g(z) \bigg( \phi_{f}(z) - \int_{D} m^2(w)\phi_f(w)G_D^m(z,w) dw \bigg)dz + \frac{1}{2} \int_{D \times D} g(z)G_D^m(z,w)g(w) dz dw \bigg) \\
    &= \exp \bigg( \int_{D} g(z)\phi_f^m(z) dz + \frac{1}{2} \int_{D \times D} g(z)G_D^m(z,w)g(w) dz dw \bigg)
\end{align*}
where the last equality uses Fubini theorem. This exactly the Laplace transform of $(h+\phi_f^m,g)$ when $h$ has the law of a massive GFF in $D$ with mass $m$ and Dirichlet boundary conditions.
\end{proof}

For later reference, let us record the computation of the normalization constant $\mathcal{Z}_{f}$ in the following lemma. This result uses the explicit computation of the normalization constant $\mathcal{Z}$ given in \eqref{explicit_Z}.

\begin{lemma} \label{lemma_Zf}
Let $D \subset \mathbb{C}$ be an open, bounded and simply connected domain. Let $\phi: D \to \mathbb{R}$ be a bounded harmonic function in $D$ and let $m:D \to \mathbb{R}_{+}$ be a bounded and continuous function. Then
\begin{align*}
    \mathcal{Z}_{f}&= \EE_{\operatorname{GFF}}^{D} \bigg[ \exp \bigg( -\frac{1}{2} \int_{D} m^2(z) :(h+\phi_{f})^2(z): dz \bigg) \bigg]\\
    &= \exp \bigg( -\frac{1}{2}\int_{D} m^2(z)\phi_f(z)\phi_f^m(z) dz + \frac{1}{2}\mu_{D}\
    \big(e^{-\langle \ell, m^2 \rangle}+\langle \ell, m^2 \rangle-1 \big) \bigg)
\end{align*}
where the integral is finite.
\end{lemma}

\begin{proof}
By \eqref{expression_Zf} and using the explicit expression \eqref{explicit_Z} for $\mathcal{Z}$, we have that
\begin{equation*}
    \mathcal{Z}_{f} = \exp \bigg( -\frac{1}{2}\int_{D} m^2(z)\phi_f(z)^2 dz + \frac{1}{2}\mu_{D}\
    \big(e^{-\langle \ell, m^2 \rangle}+\langle \ell, m^2 \rangle-1 \big) \bigg) \EE_{\operatorname{mGFF}}^{D} \big[ \exp((h, m^2\phi_{f})) \big]
\end{equation*}
where
\begin{align} \label{equality_phi_mass}
    \EE_{\operatorname{mGFF}}^{D} \big[ \exp((h, m^2\phi_{f})) \big] &= \exp \bigg( \frac{1}{2} \int_{D \times D} m^2(z)\phi_{f}(z)G_D^m(z,w)m^2(w)\phi_f(w) dwdz \bigg) \nonumber \\
    &=  \exp \bigg( -\frac{1}{2}\int_{D} m^2(z)\phi_f(z)\phi_f^m(z) dz + \frac{1}{2} \int_{D} m^2(z)\phi_f(z)^2 dz \bigg).
\end{align}
Above, we used Fubini theorem whose application can easily be justified. Indeed, $\phi_{f}$ is a bounded function in $D$ and the massive Green function $G_D^m$ is integrable in $D \times D$ since $D$ is a bounded domain. This yields the expression of Lemma \ref{lemma_Zf} for $\mathcal{Z}_f$. The finiteness of the two integrals in \eqref{equality_phi_mass} simply follows from the fact that $m^2$, $\phi_f$ and $\phi_f^m$ are bounded functions and that $D$ is a bounded domain.
\end{proof}

Note that the definition of the massive GFF and the changes of measure \eqref{RN_mGFF_Dirichlet} and \eqref{RN_mGFF_phi} can be extended to unbounded domains, provided that in such a domain $D$, the mass function $m: D \to \mathbb{R}_{+}$ is of the form $\vert \varphi' \vert (\hat m \circ \varphi)$ where $\varphi: D \to \hat D$ is a conformal isomorphism, $\hat D$ is a bounded domain and $\hat m: \hat D \to \mathbb{R}_{+}$ is a bounded and continuous function.

\subsubsection{Modes of convergence of the Radon-Nikodym derivative of the massive GFF with respect to the GFF} \label{sec_cvg_RN}

In this section, we collect some technical results on the mode of convergence as $\eps$ goes to $0$ of the random variables
\begin{align*}
    &\exp \bigg( -\frac{1}{2} \int_{D} m^2(z)\big(h_{\eps}(z)^2 - \EE[h_{\eps}(z)^2]\big) dz \bigg) \quad
    \text{and} \quad \exp \bigg( -\frac{1}{2} \int_{D} m^2(z)\big((h_{\eps}(z)+\phi_{\eps}(z))^2 - \EE[h_{\eps}(z)^2]\big) dz \bigg)
\end{align*}
where $h$ is a GFF with Dirichlet boundary conditions in $D$ and $\phi: D \to \mathbb{R}$ is a bounded harmonic function. These random variables approximate the Radon-Nikodym derivatives \eqref{RN_mGFF_Dirichlet} and \eqref{RN_mGFF_phi} of the massive GFF with respect to the GFF. Their convergence in an $\operatorname{L}^p(\PP)$ sense, which is shown below, will be useful in the proof of Theorem \ref{theorem_mSLE_4} and Theorem \ref{theorem_mCLE_4}. Indeed, there, we will need to take conditional expectations of the Radon-Nikodym derivatives \eqref{RN_mGFF_Dirichlet} and \eqref{RN_mGFF_phi}, which will be well behaved since conditioning is a contraction in $L^p$.

In the statement below, for $\eps>0$, the random variable $\int_{D} m^2(z) :(h_{\eps}(z)+\phi_{\eps}(z))^2:dz$ is defined as in \eqref{notation_h+phi_square} and we use the notation \eqref{notation_Wick_boundary} for the random variable $\int_{D} m^2(z) :(h+\phi)^2(z): dz$.

\begin{lemma} \label{lemma_Lp_cvg_boundary}
Let $D \subset \mathbb{C}$ be a bounded, open and simply connected domain and let $m: D \to \mathbb{R}_{+}$ be a bounded and continuous function. Let $h$ be a GFF with Dirichlet boundary conditions in $D$ and let $\phi: D \to \mathbb{R}$ be a harmonic function bounded by some constant $b_{\phi}>0$. Then, for any $p \in [1,\infty)$,
\begin{equation*}
    \lim_{\eps \to 0} \exp \bigg(-\frac{1}{2} \int_{D} m^2(z) :(h_{\eps}(z)+\phi_{\eps}(z))^2:dz \bigg) = \exp \bigg(-\frac{1}{2} \int_{D} m^2(z) :(h+\phi)^2(z): dz \bigg)
\end{equation*}
where the limit is in $\operatorname{L}^{p}(\PP)$.
\end{lemma}

The proof of Lemma \ref{lemma_Lp_cvg_boundary} is given in Appendix \ref{sec_app_1}. This is for the proof of this lemma that the rate of convergence \eqref{rate_cvg_L2_Wick_boundary} obtained in Lemma \ref{lemma_L2_cvg_Wick_boundary} is important.

Note that by taking $\phi \equiv 0$ in Lemma \ref{lemma_Lp_cvg_boundary}, we obtain that the random variables approximating the Radon-Nikodym derivative \eqref{RN_mGFF_Dirichlet} converge in $\operatorname{L}^{p}(\PP)$, for any $p \in [1,\infty)$.

\subsubsection{Assumptions on the mass and on the domain of definition of the massive GFF} \label{sec_assumptions}

So far, in our discussion about the massive GFF, the domain $D \subset \mathbb{C}$ and the mass $m: D \to \mathbb{R}_{+}$ have always been chosen to satisfy the assumptions of Theorem \ref{theorem_mSLE_4} and Theorem \ref{theorem_mCLE_4}. However, these assumptions are not optimal, as we now explain.

In Theorem \ref{theorem_mSLE_4}, we will show that a massive GFF can be coupled to a random curve which has the law of massive SLE$_4$. Massive SLE$_4$ is in fact well-defined and absolutely continuous with respect to SLE$_4$ provided that the domain $D$ is bounded and that the mass $m: D \to \mathbb{R}_{+}$ satisfies
\begin{align}
    &\int_{D} m^2(z) dz < \infty \label{m_L2} \\
    &\int_{D \times D} m^2(z)m^2(w) G_{D}(z,w) dz dw < \infty \label{m_Green},
\end{align}
where $G_D$ is the Green function in $D$. For the massive GFF to be absolutely continuous with respect to the GFF with Radon-Nikodym derivative \eqref{RN_mGFF_Dirichlet}, the mass $m$ must also be such that
\begin{equation}
    \int_{D \times D} m^2(z)m^2(w) G_{D}(z,w)^2 dz dw < \infty \label{m_Green_square}.
\end{equation}
Indeed, the condition \eqref{m_Green_square} ensures that the random variable $\int_{D} m^2(z) :h(z)^2: dz$ can be constructed as the limit in $L^2(\PP)$ of $\int_{D} m^2(z) :h_{\eps}(z)^2: dz$ as $\eps \to 0$. This is because if this condition is not satisfied, then the variance of $\int_{D} m^2(z) :h_{\eps}(z)^2: dz$ blows up as $\eps \to 0$.

Observe that the conditions \eqref{m_L2}, \eqref{m_Green} and \eqref{m_Green_square} on $m$ are conformally invariant in the following sense. If $\varphi: D \to \tilde D$ is a conformal map and that $m$ satisfies \eqref{m_L2}, \eqref{m_Green} and \eqref{m_Green_square}, then the mass $\tilde m: \tilde D \to \mathbb{R}$ defined via, for $w \in \tilde D$, $\tilde m(w)^2 = \vert (\varphi^{-1})'(w) \vert ^2 m(\varphi^{-1}(w))^2$ also satisfies \eqref{m_L2}, \eqref{m_Green} and \eqref{m_Green_square}. However, in their current version, the proofs of Theorem \ref{theorem_mSLE_4} and Theorem \ref{theorem_mCLE_4} require stronger assumptions on the domain $D \subset \mathbb{C}$ and the function $m: D \to \mathbb{R}_{+}$. These assumptions are helpful to prove Lemma \ref{lemma_Lp_cvg_boundary}.

We remark that the question of what happens when, for example, $D=\HH$ and the mass $m^2$ is a constant $m^2>0$ is interesting. The massive GFF in $\HH$ with mass $m^2$ does exist, since the massive Green function in $\HH$ with mass $m^2$ is symmetric and positive-definite. However, this field is only locally absolutely continuous with respect to the GFF in $\HH$ and thus we cannot define its level lines via a change of measure similar to \eqref{RN_mSLE4} or \eqref{RN_mCLE4}.

\section{Massive SLE$_4$ via change of measure} \label{sec_mSLE4}

In this section, we prove Theorem \ref{theorem_mSLE_4} and its corollary, Corollary \ref{cor_RN_mSLE4}. The proof of Theorem \ref{theorem_mSLE_4} is given in Section \ref{sec_mSLE4_GFF}. There, we first identify the conditional law of $h+\phi$ given $\gamma([0,t])$ under $\tilde \PP$ and then use the knowledge of this conditional law to show that the marginal law of $\gamma$ under $\tilde \PP$ is that of massive SLE$_4$. We will explain how to use these results to show Theorem \ref{theorem_mSLE_4}.  In the proof, we will assume that the domain $D$ is bounded and that the mass $m: D \to \mathbb{R}_{+}$ is a bounded and continuous function. Conformal covariance then enables us to extend the results to pairs $(D,m)$ satisfying the more general assumptions of Theorem \ref{theorem_mSLE_4}.

Finally, Section \ref{sec_cor_SLE4} is devoted to the proof of Corollary \ref{cor_RN_mSLE4}. Conformal covariance of massive SLE$_4$ is also established there.

\subsection{A level line of the massive GFF} \label{sec_mSLE4_GFF}

In this section, we seek to better understand the probability measure $\tilde \PP$ defined via the change of measure \eqref{RN_mSLE4} in order to prove Theorem \ref{theorem_mSLE_4}. Let us start with some preliminary remarks. Observe first that the change of measure \eqref{RN_mSLE4} of Theorem \ref{theorem_mSLE_4} is a valid change of measure. Indeed, this is just the change of measure \eqref{RN_mGFF_phi}, which extends to a change of measure on $(h+\phi,\gamma)$ since $\gamma$ is measurable with respect to $h$ under $\PP_{\operatorname{GFF}}$. This implies in particular that the marginal law of $h+\phi$ under $\tilde \PP$ is that of a massive GFF in $D$ with mass $m$ and boundary conditions $\lambda$ on $\partial D^{+}$ and $-\lambda$ on $\partial D^{-}$.

To prove Theorem \ref{theorem_mSLE_4}, we must now understand what happens to the field $h+\phi$ under $\tilde \PP$ when conditioning on $\gamma([0,t])$, for $t \geq 0$. Denote by $\mathcal{S}'(\mathbb{R}^2)$ the space of Schwartz distributions on $\mathbb{R}^2$ and observe that the field $h+\phi$ can be seen as a random element of $\mathcal{S}'(\mathbb{R}^2)$ by extending it to $0$ outside of $D$. By definition of $\tilde \PP$ via the change of measure \eqref{RN_mSLE4}, the conditional law of $h+\phi$ under $\tilde \PP$ given $\gamma$ is such that, for $t \geq 0$ and $F: \mathcal{S}'(\mathbb{R}^2) \to \mathbb{R}$ a bounded and continuous function, almost surely,
\begin{equation} \label{RN_gamma_t}
    \tilde \EE[F(h+\phi) \,\vert\, \gamma([0,t])] = \frac{1}{\mathcal{Z}_t}  \EE \bigg[ F(h+\phi) \exp \bigg( -\frac{1}{2}\int_{D} m^{2}(z):(h+\phi)^2(z): dz \bigg) \, \vert \, \gamma([0,t]) \bigg]
\end{equation}
where we have set
\begin{equation} \label{def_Ztm}
    \mathcal{Z}_t := \EE \bigg[\exp \bigg( -\frac{1}{2}\int_{D} m^{2}(z) :(h+\phi)^2(z): dz \bigg) \,\vert\, \gamma([0,t]) \bigg].
\end{equation}
Observe that in \eqref{RN_gamma_t}, the random partition function $\mathcal{Z}_t$ is almost surely strictly positive: this follows from the fact that if $X$ is an almost surely strictly positive random variable and $\mathcal{F}$ is a $\sigma$-algebra, then $\EE[X \vert \mathcal{F}]$ is almost surely strictly positive. This implies in particular that the right-hand side of \eqref{RN_gamma_t} is well-defined and thus that also the conditional law of $h+\phi$ given $\gamma([0,t])$ under $\tilde \PP$ is well-defined.

Since under $\PP$, conditionally on $\gamma([0,t])$, $h+\phi=h_t+\phi_t$ where $h_t$ and $\phi_t$ are as described in Section \ref{sec_coupling_SLE}, one expects the random partition function $\mathcal{Z}_t$, or more generally weighted conditional expectations of the form
\begin{equation} \label{weighted_cond_EE}
    \EE \bigg[\exp(i(h,f)) \exp \bigg( -\frac{1}{2}\int_{D} m^{2}(z) :(h+\phi)^2(z): dz \bigg) \,\vert\, \gamma([0,t]) \bigg]
\end{equation}
where $f:D \to \mathbb{R}$ is a bounded and continuous function, to have an expression in terms of the field $h_t$ and the harmonic function $\phi_t$. As a first step toward the identification of the conditional law of $h+\phi$ given $\gamma([0,t])$ under $\tilde \PP$, let us show that this is indeed the case.

\subsubsection{The weighted conditional expectation} \label{sec_rdm_partition_SLE}

Here and in the rest of Section \ref{sec_mSLE4}, if $\gamma : [0,\infty) \to D$ is a simple and continuous curve in $D$ that does not touch the boundary of $D$ except at its endpoints, we set, for $t \geq 0$, $D_t := D \setminus \gamma([0,t])$. We also denote by $\partial D_{t}^{+}$, respectively $\partial D_{t}^{-}$, the boundary arc of $D_t$ formed by the union of $\partial D^{+}$ and the left side of $\gamma([0,t])$, respectively the union of $\partial D^{-}$ and the right side of $\gamma([0,t])$. For $t \geq 0$, we also let $\phi_t$, respectively $\phi_{t}^m$, be the unique harmonic function, respectively massive harmonic function with mass $m$, with boundary conditions $\lambda$ on $\partial D_t^{+}$ and $-\lambda$ on $\partial D_t^{-}$.

Observe that by absolute continuity of $\tilde \PP$ with respect to $\PP$, under $\tilde \PP$, the marginal law of $\gamma$ is such that $\gamma$ is almost surely a simple and continuous curve that does not touch the boundary of $D$ except at its endpoints. Indeed, under $\PP$, the marginal law of $\gamma$ is that of SLE$_4$ in $D$ from $a$ to $b$, so that these properties are satisfied by $\gamma$ under $\PP$.

With these preliminary remarks made, let us now state and prove how to express weighted conditional expectations of the form \eqref{weighted_cond_EE} in terms of the field $h_t$ and the harmonic function $\phi_t$. We first start with the case $t \in (0,\infty)$, that is when $D_t$ is made of a single simply connected component.

\begin{proposition} \label{prop_rep_SLE}
Under the same assumptions as in Theorem \ref{theorem_mSLE_4}, let $f: D \to \mathbb{R}$ be a bounded and continuous function. Then, for any $t \in (0,\infty)$, almost surely,
\begin{align*}
    &\EE \bigg[\exp(i(h+\phi,f)) \exp \bigg( -\frac{1}{2}\int_{D} m^{2}(z) :(h+\phi)^2(z): dz \bigg) \,\vert\, \gamma([0,t]) \bigg] \\ 
    &= \begin{aligned}[t]&\EE \bigg[ \exp(i(h_t+\phi_t,f)) \exp \bigg(-\frac{1}{2} \int_{D_t} m^2(z) :(h_t+\phi_t)^2(z): dz \bigg) \vert \gamma([0,t])\bigg]\\
    &\times \exp \bigg(\int_{D_t} \frac{m^2(z)}{4\pi} \log \frac{\CR(z,\partial D)}{\CR(z,\partial D_t)} dz \bigg)\end{aligned}
\end{align*}
where $\CR(z,\partial D)$, respectively $\CR(z,\partial D_t)$, is the conformal radius of $z$ in $D$, respectively $D_t$. Above, the random variable $\int_{D_t} m^2(z) :(h_t+\phi_t)^2(z): dz$ is as defined via the notation \eqref{notation_Wick_boundary}.
\end{proposition}

\begin{proof}
Let $f: D \to \mathbb{R}$ be a bounded and continuous function and let $t \in (0,\infty)$. By Lemma \ref{lemma_Lp_cvg_boundary} and since $\vert \exp(i(h+\phi, f)) \vert \leq 1$ almost surely, we have that
\begin{align*}
    &\exp(i(h+\phi,f))\exp \bigg( -\frac{1}{2} \int_{D} m^2(z) :(h+\phi)^2(z): dz \bigg) \\
    &= \lim_{\eps \to 0} \exp(i(h+\phi,f)) \exp \bigg( -\frac{1}{2} \int_{D} m^2(z) \big( (h_{\eps}(z) + \phi_{\eps}(z))^2 - \EE[h_{\eps}(z)^2] \big) dz \bigg)
\end{align*}
where the limit in $\operatorname{L}^1(\PP)$. This yields that, almost surely,
\begin{align*}
    &\EE \bigg[\exp(i(h+\phi,f)) \exp \bigg( -\frac{1}{2}\int_{D} m^{2}(z) :(h+\phi)^2(z): dz \bigg) \,\vert\, \gamma([0,t]) \bigg] \\
    &= \lim_{\eps \to 0} \EE \bigg[ \exp(i(h+\phi,f)) \exp \bigg( -\frac{1}{2} \int_{D} m^2(z) \big( (h_{\eps}(z) + \phi_{\eps}(z))^2 - \EE[h_{\eps}(z)^2] \big) dz \bigg) \vert \gamma([0,t])\bigg]
\end{align*}
where the limit is in $\operatorname{L}^1(\PP)$. We recall that here, $h_{\eps}(z)$ denotes the circle average of $h$ as defined in Section \ref{sec_Wick} and $\phi_{\eps}(z)$ is the circle average approximation of $\phi(z)$ in the sense of \eqref{def_phi_eps}. Let us set $\EE_t[\cdot] := \EE[\cdot \vert \gamma([0,t])]$. As explained in Section \ref{sec_coupling_SLE}, under the (random) conditional law induced by $\EE_t$, $h+\phi= h_t + \phi_t$, where $h_t$ is a GFF with Dirichlet boundary conditions in $D_t$ and $\phi_{t}$ is the unique harmonic function with boundary conditions $-\lambda$ on $\partial D_t^{-}$ and $\lambda$ on $\partial D_t^{+}$. Therefore, almost surely,
\begin{align} \label{dec_EE_SLE}
    &\EE_t \bigg[ \exp(i(h+\phi,f)) \exp \bigg( -\frac{1}{2} \int_{D} m^2(z) :(h+\phi)^2(z): dz \bigg) \bigg] \nonumber \\
    &= \lim_{\eps \to 0} \EE_t \bigg[ \exp(i(h_t+\phi_t,f)) \exp \bigg( -\frac{1}{2} \int_{D} m^2(z) \big( (h_{t,\eps}(z) + \phi_{t,\eps}(z))^2 - \EE[h_{\eps}(z)^2] \big) dz \bigg)\bigg] \nonumber \\
    &= \lim_{\eps \to 0} \EE_t \bigg[ \exp(i(h_t+\phi_t,f)) \exp \bigg( -\frac{1}{2} \int_{D} m^2(z) \big( (h_{t,\eps}(z) + \phi_{t,\eps}(z))^2 - \EE_t[h_{t,\eps}(z)^2] \big) dz \bigg) \bigg] \nonumber \\ 
    &\times \exp \bigg( \frac{1}{2}\int_{D} m^2(z) \big( \EE[h_{\eps}(z)^2] - \EE_t[h_{t,\eps}(z)^2]\big) dz \bigg)
\end{align}
where the limit is in $\operatorname{L}^1(\PP)$. Since $\gamma([0,t])$ almost surely has Lebesgue measure $0$, we can rewrite the last equality as
\begin{align} \label{lim_with_int_over_Dt}
    &\EE_t \bigg[ \exp(i(h+\phi,f)) \exp \bigg( -\frac{1}{2} \int_{D} m^2(z) :(h+\phi)^2(z): dz \bigg)] \\
    &= \lim_{\eps \to 0} \begin{aligned}[t]&\EE_t \bigg[ \exp(i(h_t+\phi_t,f)) \exp \bigg( -\frac{1}{2} \int_{D_t} m^2(z) \big( (h_{t,\eps}(z) + \phi_{t,\eps}(z))^2 - \EE_t[h_{t,\eps}(z)^2] \big) dz \bigg) \bigg] \\ 
    &\times \exp \bigg( \frac{1}{2}\int_{D_t} m^2(z) \big( \EE[h_{\eps}(z)^2] - \EE_t[h_{t,\eps}(z)^2]\big) dz \bigg) \nonumber \end{aligned}
\end{align}
where the limit is in $\operatorname{L}^1(\PP)$. Since conditionally on $\gamma([0,t])$, $h_t$ is a GFF with Dirichlet boundary conditions and $\phi_t$ is almost surely a bounded harmonic function in $D_t$, by Lemma \ref{lemma_L2_cvg_Wick_boundary}, we have that, almost surely,
\begin{align} \label{lim_cond_EE_t_only}
    &\lim_{\eps \to 0} \EE_t \bigg[ \exp(i(h_t+\phi_t,f)) \exp \bigg( -\frac{1}{2} \int_{D_t} m^2(z) \big( (h_{t,\eps}(z) + \phi_{t,\eps}(z))^2 - \EE_t[h_{t,\eps}(z)^2] \big) dz \bigg) \bigg] \nonumber \\
    &= \EE_t \bigg[ \exp(i(h_t+\phi_t,f)) \exp \bigg( -\frac{1}{2} \int_{D_t} m^2(z) :(h_t+\phi_t)^2(z): dz \bigg) \bigg].
\end{align}
We thus see that to prove the proposition, it remains to show the following claim.

\begin{claim} \label{claim_CR_SLE}
Almost surely,
\begin{equation*}
    \lim_{\eps \to 0} \exp \bigg( \frac{1}{2}\int_{D_t} m^2(z) \big( \EE[h_{\eps}(z)^2] - \EE_t[h_{t,\eps}(z)^2]\big) dz \bigg) = \exp \bigg( \int_{D_t} \frac{m^2(z)}{4\pi} \log \frac{\CR(z, \partial D)}{\CR(z, \partial D_t)} dz \bigg).
\end{equation*}
\end{claim}

We postpone the proof of this claim to the end and first conclude the proof of Proposition \ref{prop_rep_SLE} based on it. From \eqref{lim_with_int_over_Dt}, we deduce that Claim \ref{claim_CR_SLE} together with \eqref{lim_cond_EE_t_only} imply that, almost surely,
\begin{align*}
     &\EE_t \bigg[ \exp(i(h+\phi,f)) \exp \bigg( -\frac{1}{2} \int_{D} m^2(z) :(h+\phi)^2(z): dz \bigg)\bigg] \nonumber \\
     &= \EE_t \bigg[ \exp(i(h_t+\phi_t,f)) \exp \bigg( -\frac{1}{2} \int_{D_t} m^2(z) :(h_t+\phi_t)^2(z): dz \bigg) \bigg] 
    \exp \bigg( \int_{D_t} \frac{m^2(z)}{4\pi} \log \frac{\CR(z, \partial D)}{\CR(z, \partial D_t)} dz \bigg).
\end{align*}
This is exactly the statement of Proposition \ref{prop_rep_SLE} and thus completes the proof.
\end{proof}

It now remains to prove Claim \ref{claim_CR_SLE}.

\begin{proof}[Proof of Claim \ref{claim_CR_SLE}]
Let $\eps > 0$ and set $D_{t,\eps} := \{ z \in D_t: \operatorname{dist}(z,\gamma([0,t])) > \eps \}$. Observe that, almost surely, if $z \in D_{t,\eps}$ is such that $\operatorname{dist}(z, \partial D) > \eps$, then
\begin{equation} \label{bound_diff_bulk}
    \EE[h_{\eps}(z)^2] - \EE_t[h_{t,\eps}(z)^2] = \frac{1}{2\pi} \log \frac{\CR(z, \partial D)}{\CR(z,\partial D_t)}.
\end{equation}
On the other hand, it is easy to see that there exists $K > 0$ such that for any $\eps > 0$, almost surely, for any $z \in D_{t,\eps}$ with $\operatorname{dist}(z, \partial D) \leq \eps$,
\begin{equation} \label{bound_diff_boundary}
    0 \leq \EE[h_{\eps}(z)^2] - \EE_t[h_{t,\eps}(z)^2] \leq \EE[h_{\eps}(z)^2] \leq K.
\end{equation}
Since almost surely for any $z \in D_t$, $\CR(z, \partial D) \geq \CR(z, \partial D_t)$ and since the function $m$ is bounded, we obtain from \eqref{bound_diff_bulk} and \eqref{bound_diff_boundary} that there exists $C > 0$ such that, almost surely, for any $z \in D_{t,\eps}$,
\begin{equation} \label{bound_diff_EE}
    m^2(z) \vert \EE[h_{\eps}(z)^2] - \EE_t[h_{t,\eps}(z)^2] \vert \leq C \bigg( \log \frac{\CR(z, \partial D)}{\CR(z,\partial D_t)} +1 \bigg).
\end{equation}
Observe also that, almost surely,
\begin{equation*}
    \lim_{\eps \to 0} m^2(z)(\EE[h_{\eps}(z)^2] - \EE_t[h_{t,\eps}(z)^2])\mathbb{I}_{D_{t,\eps}}(z) = \frac{m^2(z)}{2\pi}  \log \frac{\CR(z, \partial D)}{\CR(z,\partial D_t)}
\end{equation*}
pointwise in $D_t$. Moreover, the function $z \mapsto \log(\CR(z, \partial D)/\CR(z,\partial D_t))$ is almost surely integrable in $D_t$. Therefore, using the dominated convergence theorem with the bound \eqref{bound_diff_EE}, we obtain that, almost surely,
\begin{equation*}
    \lim_{\eps \to 0} \int_{D_{t,\eps}} m^2(z) (\EE[h_{\eps}(z)^2] - \EE_t[h_{t,\eps}(z)^2]) dz  = \int_{D_t} \frac{m^2(z)}{2\pi}  \log \frac{\CR(z, \partial D)}{\CR(z,\partial D_t)} dz.
\end{equation*}
To prove Claim \ref{claim_CR_SLE}, it remains to show that, almost surely,
\begin{equation} \label{cvg_0_not_Dt_eps}
    \lim_{\eps \to 0} \int_{D_t \setminus D_{t,\eps}} m^2(z) (\EE[h_{\eps}(z)^2] - \EE_t[h_{t,\eps}(z)^2]) dz =0.
\end{equation}
For this, we simply observe that the estimate \eqref{estimate_sup_variance} shows that there exists $C>0$ such that for any $0<\eps<1/2$ and any $z \in D_t$,
\begin{equation*}
    0 \leq \EE[h_{\eps}(z)^2] - \EE_t[h_{t,\eps}(z)^2] \leq \sup_{z \in D} \EE[h_{\eps}(z)^2] \leq C\log \eps^{-1}.
\end{equation*}
Together with the boundedness of $m$, this yields that, almost surely,
\begin{equation} \label{ineq_not_Dt_eps}
    \bigg \vert \int_{D_t \setminus D_{t,\eps}} m^2(z) (\EE[h_{\eps}(z)^2] - \EE_t[h_{t,\eps}(z)^2]) dz \bigg \vert \leq \overline{m}^2C (\log \eps^{-1})\Area(D \setminus D_{t,\eps}).
\end{equation}
Since $\gamma([0,t])$ almost surely has Hausdorff dimension $3/2$ \cite{dim_SLE}, we have that, almost surely, $\Area(D \setminus D_{t,\eps}) = O(\eps^{1/2})$. Therefore, the inequality \eqref{ineq_not_Dt_eps} implies the almost sure convergence \eqref{cvg_0_not_Dt_eps}, which completes the proof of the claim.
\end{proof}

\begin{remark}
The proof of Proposition \ref{prop_rep_SLE} actually shows that, $\PP$-almost surely,
\begin{align} 
	&\exp \bigg(-\frac{1}{2}\int_{D} m^2(z) :(h(z)+\phi(z))^2: dz \bigg) \nonumber \\
	&= \exp \bigg(-\frac{1}{2}\int_{D_t} m^2(z) :(h_t(z)+\phi_t(z))^2: dz + \int_{D_t} \frac{m^2(z)}{4\pi}\log \frac{\CR(z,\partial D)}{\CR(z,\partial D_t)}dz \bigg). \label{as_SLE_t}
\end{align}
Indeed, fix $\delta > 0$. Then, denoting by $X_{t,\eps}$ the $\eps$-approximation of the Wick square on the left-hand side of \eqref{lim_cond_EE_t_only} and $X_{t,0}$ its limit, we have that, for $\PP$-almost every $\gamma([0,t])$, $\PP_t(\vert e^{X_{t,\eps}} -e^{ X_{t,0}}\vert > \delta)$ converges to $0$ as $\eps$ tends to $0$. Therefore, by dominated convergence, $ e^{X_{t,\eps}}$ converges to $e^{ X_{t,0}}$ in $\PP$-probability as $\eps$ tends to $0$. Defining $(Y_{t,\eps})_{\eps > 0}$ as the random variables of Claim \ref{claim_CR_SLE} and $Y_{t,0}$ as their limit as $\eps$ tends to $0$ given by this claim, we can deduce from this and Claim \ref{claim_CR_SLE} that $e^{X_{t,\eps}}e^{Y_{t,\eps}}$ converges to $e^{X_t}e^{Y_t}$ in $\PP$-probability as $\eps$ tends to $0$. Since $e^{X_{t,\eps}}e^{Y_{t,\eps}}$ also converges to $\exp(-\frac{1}{2}\int_{D} m^2(z) :(h(z)+\phi(z))^2: dz)$ in $\PP$-probability as $\eps$ tends to $0$, we obtain the almost sure equality \eqref{as_SLE_t}. 
\end{remark}

Let us now give an expression for weighted conditional expectations of the form \eqref{weighted_EE_Gamma} when $t=\infty$. In this case, $D_{\infty}:=D \setminus \gamma([0,\infty))$ is composed of two simply connected components. Recall that we denote by $D_1$, respectively $D_2$, the component on the left, respectively right, of $\gamma([0,\infty))$. Moreover, under $\PP$, conditionally on $\gamma([0,\infty])$, $h+\phi=h_1+\phi_1+h_2+\phi_2$, where $h_1$, $h_2$, $\phi_1$ and $\phi_2$ are as described in Section \ref{sec_coupling_SLE}.

\begin{proposition} \label{prop_rep_SLE_infty}
Under the same assumptions as in Theorem \ref{theorem_mSLE_4}, let $f: D \to \mathbb{R}$ be a bounded and continuous function. Then, almost surely,
\begin{align*}
    &\EE \bigg[\exp(i(h+\phi,f)) \exp \bigg( -\frac{1}{2}\int_{D} m^{2}(z) :(h+\phi)^2(z): dz \bigg) \,\vert\, \gamma([0,\infty)) \bigg] \\ 
    &=\EE \bigg[ \exp(i(h_1+\phi_1,f)) \exp \bigg(-\frac{1}{2} \int_{D_1} m^2(z) :(h_1+\phi_1)^2(z): dz \bigg) \vert \gamma([0,\infty))\bigg] \\
    & \times \EE \bigg[ \exp(i(h_2+\phi_2,f)) \exp \bigg(-\frac{1}{2} \int_{D_2} m^2(z) :(h_2+\phi_2)^2(z): dz \bigg) \vert \gamma([0,\infty))\bigg] \\
    &\times \exp \bigg(\int_{D_1} \frac{m^2(z)}{4\pi} \log \frac{\CR(z,\partial D)}{\CR(z,\partial D_1)} dz + \int_{D_2} \frac{m^2(z)}{4\pi} \log \frac{\CR(z,\partial D)}{\CR(z,\partial D_2)} \bigg)
\end{align*}
where $\CR(z,\partial D)$, respectively $\CR(z,\partial D_j)$, $j=1,2$, is the conformal radius of $z$ in $D$, respectively $D_j$. Above, for $j=1,2$, the random variable $\int_{D_j} m^2(z) :(h_j+\phi_j)^2(z): dz$ is as defined via the notation \eqref{notation_Wick_boundary}.
\end{proposition}

\begin{proof}
This is left to the reader, since the proof is identical to but simpler than the proof of Proposition \ref{prop_rep_CLE}, or alternatively may be obtained by taking $t$ to $\infty$ in Proposition \ref{prop_rep_SLE}.
\end{proof}

\subsubsection{The conditional law of $h+\phi$ under $\tilde \PP$}

With Proposition \ref{prop_rep_SLE} in hands, let us now identify the conditional law of $h+\phi$ given $\gamma([0,t])$ under $\tilde \PP$, for $t \geq 0$. In this section, we keep the same notations as those adopted in Section \ref{sec_rdm_partition_SLE} for the slit domains $(D_t)_{t \geq 0}$, their boundary arcs $(\partial D_t^{+})_{t \geq 0}$ and $(\partial D_t^{-})_{t \geq 0}$ and the functions $(\phi_t)_{t \geq 0}$ and $(\phi_{t}^m)_{t \geq 0}$.

\begin{proposition} \label{prop_law_hSLE}
With the same assumptions as in Theorem \ref{theorem_mSLE_4}, let $t \geq 0$. Then, under $\tilde \PP$, conditionally on $\gamma([0,t])$, $h+\phi=h_t+\phi_t$ where $h_t+\phi_t$ has the law of a massive GFF in $D_t$ with mass $m$ and with boundary conditions $\lambda$ on $\partial D_t^{+}$ and $-\lambda$ on $\partial D_t^{-}$. In other words, under $\tilde \PP$, the conditional law of $h+\phi$ given $\gamma([0,t])$ is that of $h_t^m+\phi_t^m$ where $h_t^m$ has the law of a massive GFF in $D_t$ with mass $m$ and with Dirichlet boundary conditions and $\phi_t^m$ is the massive harmonic function in $D_t$ with boundary values $\lambda$ on $\partial D_t^{+}$ and $-\lambda$ on $\partial D_t^{-}$.
\end{proposition}

To prove Proposition \ref{prop_law_hSLE}, as mentioned in the introduction, we show a rigorous conditional version of the path-integral formalism for the massive GFF. Most of the technical work has already been done in the proof of Proposition \ref{prop_rep_SLE} and it now remains to apply the strategy developed there.

\begin{proof}[Proof of Proposition \ref{prop_law_hSLE}]
Let $t \geq 0$. To establish the proposition, we are going to compute the characteristic function of $(h+\phi, f)$ under $\tilde \PP$ given $\gamma([0,t])$, where $f$ is a smooth function with compact support in $D$. We have that, almost surely,
\begin{align} \label{charac_GFFgamma_t}
    \tilde \EE \big[ \exp(i(h+\phi,f)) \vert \gamma([0,t]) \big]
    = \frac{1}{\mathcal{Z}_t} \EE \bigg[ \exp(i(h+\phi,f)) \exp \bigg( -\frac{1}{2} \int_{D} m^2(z):(h+\phi)^2(z):dz \bigg) \vert \gamma([0,t]) \bigg]
\end{align}
where, as before, we have set
\begin{equation*}
    \mathcal{Z}_t := \EE \bigg[ \exp \bigg( -\frac{1}{2} \int_{D} m^2(z):(h+\phi)^2(z):dz \bigg) \vert \gamma([0,t]) \bigg].
\end{equation*}
By Proposition \ref{prop_rep_SLE}, we have that, almost surely,
\begin{align} \label{weighted_charac_GFFgamma_t}
    &\EE \bigg[ \exp(i(h+\phi,f)) \exp \bigg( -\frac{1}{2} \int_{D} m^2(z):(h+\phi)^2(z):dz \bigg) \vert \gamma([0,t]) \bigg] \nonumber \\
    &= \EE \bigg[ \exp(i(h_t+\phi_t,f)) \exp \bigg( -\frac{1}{2} \int_{D_t} m^2(z):(h_t+\phi_t)^2(z):dz \bigg) \vert \gamma([0,t]) \bigg]
    \exp \bigg( \int_{D_t} \frac{m^2(z)}{4\pi} \log \frac{\CR(z,\partial D)}{\CR(z,\partial D_t)}\bigg).
\end{align}
On the other hand, Proposition \ref{prop_rep_SLE} applied with $f \equiv 0$ yields that, almost surely,
\begin{equation} \label{dec_Zmt}
    \mathcal{Z}_t = \exp \bigg( \int_{D_t} \frac{m^2(z)}{4\pi} \log \frac{\CR(z,\partial D)}{\CR(z,\partial D_t)}\bigg)\mathcal{Z}_{\phi,t}
\end{equation}
where we have set
\begin{equation*}
    \mathcal{Z}_{\phi, t} := \EE \bigg[\exp \bigg(-\frac{1}{2} \int_{D_t} m^2(z) :(h_t+\phi_t)^2(z): dz \bigg)\vert \gamma([0,t]) \bigg].
\end{equation*}
Going back to \eqref{charac_GFFgamma_t}, we obtain from this that, almost surely,
\begin{align} \label{cond_EE_without_CR}
    &\tilde \EE \big[ \exp(i(h+\phi,f))\vert \gamma([0,t]) \big] \nonumber \\
    &= \frac{1}{\mathcal{Z}_{\phi, t}} \EE \bigg[ \exp(i(h_t+\phi_t,f)) \exp \bigg(-\frac{1}{2} \int_{D_t} m^2(z) :(h_t+\phi_t)^2(z): dz \bigg)\vert \gamma([0,t]) \bigg].
\end{align}
Indeed, the term involving the conformal radii in \eqref{weighted_charac_GFFgamma_t} cancels with the one stemming from the equality \eqref{dec_Zmt} for $\mathcal{Z}_t$. Observe now that by \eqref{RN_mGFF_phi}, conditionally on $\gamma([0,t])$,
\begin{equation*}
    \frac{1}{\mathcal{Z}_{\phi, t}} \exp \bigg(-\frac{1}{2} \int_{D_t} m^2(z) :(h_t+\phi_t)^2(z): dz \bigg)
\end{equation*}
is almost surely equal to the Radon-Nikodym derivative of the massive GFF in $D_t$ with mass $m$ and boundary conditions $\lambda$ on $\partial D_t^{+}$ and $-\lambda$ on $\partial D_t^{-}$ with respect to the GFF in $D_t$ with the same boundary conditions. We deduce from this and the equality \eqref{cond_EE_without_CR} that under $\tilde \PP$, conditionally on $\gamma([0,t])$, $h+\phi=h_t+\phi_t$ where $h_t+\phi_t$ has the law of a massive GFF with mass $m$ in $D_t$ and boundary conditions $\lambda$ on $\partial D_t^{+}$ and $-\lambda$ on $\partial D_t^{-}$.
\end{proof}

Let us now establish the conditional law of $h+\phi$ under $\tilde \PP$ conditionally on $\gamma([0,\infty))$. Below, the domains $D_j$ and the functions $\phi_j$ for $j=1,2$ are as described in Section \ref{sec_coupling_SLE}.

\begin{proposition} \label{prop_law_hSLE_infty}
With the same assumptions as in Theorem \ref{theorem_mSLE_4}, under $\tilde \PP$, conditionally on $\gamma([0,\infty))$, $h+\phi=h_1+\phi_1+h_2+\phi_2$ where for $j=1,2$, $h_j+\phi_j$ has the law of a massive GFF in $D_j$ with mass $m$ and boundary conditions $\phi_j$. Moreover, under $\tilde \PP$, conditionally on $\gamma([0,\infty))$, $h_1+\phi_1$ and $h_2+\phi_2$ are independent.
\end{proposition}

\begin{proof}
As in the proof of Proposition \ref{prop_rep_SLE_infty}, we set $\EE_{\infty}[\cdot] := \EE[\cdot \vert \gamma([0,\infty)])]$ and denote by $\PP_{\infty}$ the (random) conditional law induced by $\EE_{\infty}[\cdot]$. Let $f: D \to \mathbb{R}$ be a smooth function with compact support in $D$. To show Proposition \ref{prop_law_hSLE_infty}, we are going to compute the characteristic function of $(h+\phi, f)$ under $\tilde \PP$ given $\gamma([0,\infty))$. We have that, almost surely,
\begin{align} \label{EE_tilde_intro}
    \tilde \EE[\exp(i(h+\phi,f)) \vert \gamma([0,\infty))]
    =\frac{1}{\mathcal{Z}_{\infty}} \EE_{\infty} \bigg[ \exp(i(h+\phi, f)) \exp\bigg(-\frac{1}{2} \int_{D} m^2(z) :(h+\phi)^2(z): dz \bigg) \bigg]
\end{align}
where we have set
\begin{equation*}
    \mathcal{Z}_{\infty} := \EE_{\infty} \bigg[ \exp\bigg(-\frac{1}{2} \int_{D} m^2(z) :(h+\phi)^2(z): dz \bigg) \bigg].
\end{equation*}
By Proposition \ref{prop_rep_SLE_infty}, we have that, almost surely,
\begin{align} \label{weighted_infty}
    &\EE_{\infty} \bigg[ \exp(i(h+\phi, f)) \exp\bigg(-\frac{1}{2} \int_{D} m^2(z) :(h+\phi)^2(z): dz \bigg) \bigg] \nonumber \\
    &= \prod_{j=1}^2 \EE_{\infty} \bigg[ \exp(i(h_j+\phi_j,f)) \exp \bigg(-\frac{1}{2} \int_{D_j} m^2(z) :(h_j+\phi_j)^2(z): dz \bigg) \bigg] \exp \bigg( \int_{D_j} \frac{m^2(z)}{4\pi} \log \frac{\CR(z,\partial D)}{\CR(z,\partial D_j)} dz\bigg).
\end{align}
On the other hand, Proposition \ref{prop_rep_SLE_infty} applied with $f \equiv 0$ yields that, almost surely,
\begin{equation} \label{prod_dec_infty}
    \mathcal{Z}_{\infty} = \mathcal{Z}_{1,\infty}\mathcal{Z}_{2,\infty} \exp \bigg( \int_{D_1} \frac{m^2(z)}{4\pi} \log \frac{\CR(z,\partial D)}{\CR(z,\partial D_1)} dz + \int_{D_2} \frac{m^2(z)}{4\pi} \log \frac{\CR(z,\partial D)}{\CR(z,\partial D_2)} dz \bigg)
\end{equation}
where we have set, for $j=1,2$,
\begin{equation*}
    \mathcal{Z}_{j,\infty} = \EE_{\infty} \bigg[ \exp \bigg(-\frac{1}{2} \int_{D_j} m^2(z) :(h_j+\phi_j)^2(z): dz \bigg) \bigg].
\end{equation*}
Going back to \eqref{EE_tilde_intro}, we obtain that, almost surely,
\begin{align*}
    \tilde \EE[\exp(i(h+\phi,f)) \vert \gamma([0,\infty))]
    =\prod_{j=1}^2 \frac{1}{\mathcal{Z}_{j,\infty}} \EE_{\infty} \bigg[ \exp(i(h_j+\phi_j,f)) \exp \bigg(-\frac{1}{2} \int_{D_j} m^2(z) :(h_j+\phi_j)^2(z): dz \bigg) \bigg].
\end{align*}
This is because the term involving the conformal radii in \eqref{weighted_infty} cancels with that of the product decomposition \eqref{prod_dec_infty} of $\mathcal{Z}_{\infty}$. Observe now that by \eqref{RN_mGFF_phi}, conditionally on $\gamma([0,\infty))$, for $j=1,2$,
\begin{equation*}
    \frac{1}{\mathcal{Z}_{j,\infty}} \exp \bigg(-\frac{1}{2} \int_{D_j} m^2(z) :(h_j+\phi_j)^2(z): dz \bigg)
\end{equation*}
is almost surely equal to the Radon-Nikodym derivative of the massive GFF in $D_j$ with mass $m$ and boundary conditions $\phi_j$ with respect to the GFF in $D_j$ with boundary conditions $\phi_j$. We deduce from this that under $\tilde \PP$, conditionally on $\gamma([0,\infty))$, $h+\phi=h_1+\phi_1+h_2+\phi_2$, where $h_1+\phi_1$ and $h_2+\phi_2$ are as described in the statement of Proposition \ref{prop_law_hSLE_infty}. This completes the proof.
\end{proof}

\subsubsection{The marginal law of $\gamma$ under $\tilde \PP$} \label{sec_marginal_law_gamma}

To conclude the proof of Theorem \ref{theorem_mSLE_4}, we must identify the marginal of $\gamma$ under $\tilde \PP$. This is the content of the next proposition. 

\begin{proposition} \label{prop_law_gamma}
With the same assumptions as in Theorem \ref{theorem_mSLE_4}, under $\tilde \PP$, the marginal law of $\gamma$ is that of massive SLE$_4$ with mass $m$ in $D$ from $a$ to $b$.
\end{proposition}

The proof of Proposition \ref{prop_law_gamma} relies on the martingale characterization of massive SLE$_4$ established in \cite[Section~5]{mHE} and recalled just below. This characterization is similar in spirit to that of SLE$_4$ via a certain martingale observable and its proof uses the same kind of arguments. The characterization of massive SLE$_4$ stated just below is slightly weaker than that in \cite[Section~5]{mHE}, in the sense that the statement actually holds if the curve is only non-self-touching, and not necessarily simple. However, since this stronger assumption is satisfied here (by absolute continuity of $\tilde \PP$ with respect to $\PP$), this does not really matter.

\begin{proposition} \label{prop_mart_mSLE4}
Let $D \subset \mathbb{C}$ be an open, bounded and simply connected domain and let $a, b \in \partial D$. Let $\gamma: [0,\infty) \to D$ be a random simple curve from $a$ to $b$ which does not hit $\partial D$ except at its endpoints and denote by $(\mathcal{F}_t)_t$ the filtration generated by $\gamma$. Let $m: D \to \mathbb{R}_{+}$ be a bounded and continuous function. For $z \in D$, define the stopping time $\tau_{z}$ for $(\mathcal{F}_t)_t$ by $\tau_{z}:= \inf \{t \geq 0: z \in \gamma([0,t]) \}$. For $t \geq 0$, denote also by $H_t^m$ the unique massive harmonic function in $D_t:=D\setminus \gamma([0,t])$ with boundary conditions $1/2$ on $\partial D_t^{+}$ and $-1/2$ on $\partial D_t^{-}$. Assume that, for any $z \in D$, $(H_t^m(z), 0 \leq t \leq \tau_z)$ is a martingale with respect to the filtration $(\mathcal{F}_t)_t$. Then, when appropriately time-parametrized, $\gamma$ has the law of a massive SLE$_4$ curve in $D$ from $a$ to $b$ with mass $m$.
\end{proposition}

Before turning to the proof of Proposition \ref{prop_law_gamma} via the above characterization of massive SLE$_4$, we remark that Proposition \ref{prop_law_gamma} could also be proved by using Girsanov's theorem to compute the the driving function under $\tilde \PP$ of the curve $\gamma$ conformally mapped to the complex upper-half plane $\HH$. This would show that under $\tilde \PP$, $\gamma$ conformally mapped to $\HH$ has the same driving function as a massive SLE$_4$ curve conformally mapped to $\HH$, thus proving Theorem \ref{theorem_mSLE_4}. However, a lot of computations are necessary for this strategy to be successful and this is why we use instead the martingale characterization of massive SLE$_4$ given by Proposition \ref{prop_mart_mSLE4}.

\begin{proof}[Proof of Proposition \ref{prop_law_gamma}]
To show Proposition \ref{prop_law_gamma}, we wish to use the martingale characterization of massive SLE$_4$ given by Proposition \ref{prop_mart_mSLE4}. Observe first that since $\tilde \PP$ is absolutely continuous with respect to $\PP$, under $\tilde \PP$, the marginal law of $\gamma$ is such that $\gamma$ is almost surely a simple and continuous curve that almost surely does not touch $\partial D$ except at its endpoints. This follows from the fact that the marginal law of $\gamma$ under $\PP$ is that of SLE$_4$ in $D$ from $a$ to $b$ and that these properties are satisfied by SLE$_4$. Thus, the marginal law of $\gamma$ under $\tilde \PP$ is such that $\gamma$ is a random curve satisfying the first part of the assumptions of Proposition \ref{prop_mart_mSLE4}. To apply this proposition, it remains show that for any $z \in D$, $(H_t^m(z), 0 \leq t \leq \tau_z)$ is a martingale under $\tilde \PP$, where $H_t^m$ is defined as in Proposition \ref{prop_mart_mSLE4}. To establish this, we are going to use the fact stated above as Proposition \ref{prop_law_hSLE} that $\tilde \PP$ gives rise to a coupling between $\gamma$ and $h+\phi$, where, under $\tilde \PP$, the marginal law of $h+\phi$ is that of a massive GFF in $D$ with boundary conditions $\lambda$ on $\partial D^{+}$ and $-\lambda$ on $\partial D^{-}$. For $z \in D$ and $\eps > 0$, set
\begin{equation*}
    T_{\eps}(z) := \inf \{ t \geq 0: \operatorname{dist}(z, \gamma([0,t])) \leq \eps \}.
\end{equation*}
$T_{\eps}(z)$ is a stopping time for the filtration generated by $\gamma$. We denote by $\rho_{z,\eps}^m$ the massive harmonic measure with mass $m$ of $\partial B(z,\eps)$ seen from $z$. That is, for any $\tilde \partial \subset \partial B(z,\eps)$,
\begin{equation*}
    \int_{\tilde \partial} \rho_{z,\eps}^m(dw)  = \EE_{z}^{(m)} [\mathbb{I}_{\tau^{*} > \tau_{\eps}} \mathbb{I}_{B_{\tau_{\eps}} \in \tilde \partial}]
\end{equation*}
where under $\EE_{z}^{(m)}$, $B$ has the law of a massive (killed) Brownian motion with mass (killing rate) $m$ started at $z$, $\tau^{*}$ denotes its killing time and $\tau_{\eps}$ its first hitting time of $\partial B(z,\eps)$. Moreover, by definition, if a function $f: D \to \mathbb{R}$ is a bounded massive harmonic function in $D$, then for any $\eps > 0$ and any $z \in D$ such that $\operatorname{dist}(z,\partial D) > \eps$,
\begin{equation*}
    \int f(w) \rho_{z,\eps}^{m}(dw) = f(z).
\end{equation*}
Going back to the setting of Proposition \ref{prop_law_gamma}, observe that, almost surely, for any $t \geq 0$ and any $z \in D_t$,
\begin{equation*}
    H_t^m(z) = \frac{1}{2\lambda} \phi_t^m(z).
\end{equation*}
Indeed, the functions on each side of the equality are massive harmonic in $D_t$ with the same boundary values and therefore, they are equal. Now, let $z \in D$ and $\eps > 0$. To show that $(H_t^m(z), 0 \leq t \leq \tau_z)$ is a martingale under $\tilde \PP$, we are going to look at the "massive circle average" of the massive GFF $h + \phi$. This will allow us to express $H_t^m(z)$ as an observable of the field $h+\phi$ conditioned on $\gamma([0,t])$ under $\tilde \PP$.

This massive circle average is defined in a similar way as the circle average of the GFF (see Section \ref{sec_Wick}): this is the random variable $(h+\phi, \rho_{z,\eps}^m)$. Even though $\rho_{z,\eps}^m$ is not a smooth function, this is a well-defined random variable for the same reasons as in the case of the circle average of the GFF. Note also that $(h+\phi, \rho_{z,\eps}^m)$ has expectation $(\phi^m, \rho_{z,\eps}^m)$ with respect to $\tilde \PP$ since, under $\tilde \PP$, by Proposition \ref{prop_law_hSLE}, the marginal law of $h+\phi$ is that of a massive GFF in $D$ with boundary conditions $\lambda$ on $\partial D^{+}$ and $-\lambda$ on $\partial D^{-}$. Moreover, $(\phi^m, \rho_{z,\eps}^m)$ is equal to $\phi^m(z)$ if $\operatorname{dist}(z,\partial D) > \eps$ since $\phi^m$ is a massive harmonic function. 

We can then deduce from the above discussion and Proposition \ref{prop_law_hSLE} that, almost surely,
\begin{align*}
    \tilde \EE \big[ (h+\phi, \rho_{z,\eps}^m) \vert \gamma([0,t]) \big] \mathbb{I}_{t < T_{\eps}(z)} &= \tilde \EE \big[ (h_t+\phi_t, \rho_{z,\eps}^m) \vert \gamma([0,t]) \big] \mathbb{I}_{t < T_{\eps}(z)} = (\phi_t^m, \rho_{z,\eps}^m)\mathbb{I}_{t < T_{\eps}(z)} = \phi_t^m(z) \mathbb{I}_{t < T_{\eps}(z)}.
\end{align*}
The second equality follows from Proposition \ref{prop_law_hSLE} which shows that under $\tilde \PP$, conditionally on $\gamma([0,t])$, $h+\phi=h_t^m+\phi_t^m$ where $h_t^m$ has the law of a massive GFF in $D_t$ with mass $m$ and Dirichlet boundary conditions. This implies that, almost surely,
\begin{equation*}
    H_t^m(z) \mathbb{I}_{\{t < T_{\eps}(z)\}} = \frac{1}{2\lambda} \tilde \EE \big[ (h+\phi, \rho_{z,\eps}^m) \vert \gamma([0,t]) \big]\mathbb{I}_{\{t < T_{\eps}(z)\}}.
\end{equation*}
Taking $\eps \to 0$, it follows that $(H_t^m(z), 0 \leq t \leq \tau_z)$ is a martingale with respect to $(\mathcal{F}_t)_t$ under $\tilde \PP$, where $(\mathcal{F}_t)_t$ is the filtration generated by $\gamma$. By Proposition \ref{prop_mart_mSLE4}, this shows that under $\tilde \PP$, the marginal law of $\gamma$ is that of a massive SLE$_4$ with mass $m$ in $D$ from $a$ to $b$.
\end{proof}

We can now conclude the proof of Theorem \ref{theorem_mSLE_4}.

\begin{proof}[Proof of Theorem \ref{theorem_mSLE_4}]
$\tilde \PP$ is a well-defined probability measure which is absolutely continuous with respect to $\PP$. Moreover, by definition, the marginal law of $h$ under $\tilde \PP$ is that of a massive GFF in $D$ with mass $m$ and boundary conditions $\lambda$ on $\partial D^{+}$ and $-\lambda$ on $\partial D^{-}$. Proposition \ref{prop_law_hSLE} establishes that the conditional law under $\tilde \PP$ of $h+\phi$ given $\gamma([0,t])$, for $t \in (0,\infty)$, is as claimed in the statement of Theorem \ref{theorem_mSLE_4}. Proposition \ref{prop_law_hSLE_infty} shows the part of the statement about the conditional law of $h+\phi$ given $\gamma([0,\infty))$ under $\tilde \PP$. Finally, the fact that, under $\tilde \PP$, the marginal law of $\gamma$ is that of massive SLE$_4$ with mass $m$ in $D$ from $a$ to $b$ follows from Proposition \ref{prop_law_gamma}. This concludes the proof of Theorem \ref{theorem_mSLE_4}.
\end{proof}

\subsection{Some properties of massive SLE$_4$} \label{sec_cor_SLE4}

In this section, we use Theorem \ref{theorem_mSLE_4} to show two properties of massive SLE$_4$. The first one, stated as Corollary \ref{cor_RN_mSLE4} in the introduction, is an expression for the Radon-Nikodym derivative of the law of massive SLE$_4$ with respect to that of SLE$_4$. We keep here the same notations as in the introduction and Section \ref{sec_SLE}.

\begin{proof}[Proof of Corollary \ref{cor_RN_mSLE4}]
By definition of $\tilde \PP$, we have that, for $t \geq 0$,
\begin{equation*}
    \frac{\der \tilde \PP}{\der \PP} ((h+\phi, \gamma)) \bigg \vert_{\sigma(\gamma(s),s\leq t)} = \frac{1}{\mathcal{Z}} \EE \bigg[ \exp \bigg( -\frac{1}{2} \int_{D} m^2(z) :(h+\phi)^2(z): dz \bigg) \vert \gamma([0,t]) \bigg] = \frac{\mathcal{Z}_t}{\mathcal{Z}}
\end{equation*}
where $\mathcal{Z}$ is as in the statement of Theorem \ref{theorem_mSLE_4} and where, as before, we have set
\begin{equation*}
    \mathcal{Z}_t := \EE \bigg[ \exp \bigg( -\frac{1}{2} \int_{D} m^2(z) :(h+\phi)^2(z): dz \bigg) \vert \gamma([0,t]) \bigg].
\end{equation*}
In other words, for any random variable $Y$ which is $\sigma(\gamma(s),s\leq t)$-measurable, $\tilde \EE [Y] = \EE[\mathcal{Z}_t/\mathcal{Z}]$. Since under $\PP$, the marginal law of $\gamma$ is that of SLE$_4$ in $D$ from $a$ to $b$, see Section \ref{sec_coupling_SLE}, while, by Theorem \ref{theorem_mSLE_4}, under $\tilde \PP$, the marginal law of $\gamma$ is that of massive SLE$_4$ in $D$ from $a$ to $b$ with mass $m$, this implies that, almost surely,
\begin{equation} \label{RN_mSLE4_SLE4}
    \frac{\der \PP_{\operatorname{mSLE_4}}^{(D,a,b)}}{\der \PP_{\operatorname{SLE_4}}^{(D,a,b)}} \bigg \vert_{\sigma(\gamma(s), s \leq t)} = \frac{\mathcal{Z}_t}{\mathcal{Z}}.
\end{equation}
By Proposition \ref{prop_rep_SLE} applied with $f \equiv 0$, we have that, almost surely,
\begin{equation*}
    \mathcal{Z}_t = \exp \bigg( \int_{D_t} \frac{m^2(z)}{4\pi} \log \frac{\CR(z,\partial D)}{\CR(z,\partial D_t)} dz \bigg) \EE \bigg[ \exp\bigg( -\frac{1}{2} \int_{D_t} m^2(z) :(h_t+\phi_t)^2(z): dz \bigg) \vert \gamma([0,t]) \bigg].
\end{equation*}
To compute the conditional expectation in the above equality, since under $\PP$, conditionally on $\gamma([0,t])$, $h_t+\phi_t$ is a GFF with boundary conditions $\phi_t$ in $D_t$, we can use Lemma \ref{lemma_Zf} in the domain $D_t$. This yields that, almost surely,
\begin{align*}
    \mathcal{Z}_t = &\exp \bigg( -\frac{1}{2}\int_{D_t} m^2(z)\phi_t(z)\phi_t^m(z) dz + \int_{D_t} \frac{m^2(z)}{4\pi} \log \frac{\CR(z,\partial D)}{\CR(z,\partial D_t)} dz \bigg)\exp \bigg( \frac{1}{2} \mu_t(e^{-\langle \ell, m^2 \rangle}+\langle \ell, m^2 \rangle -1) \bigg).
\end{align*}
The statement of the corollary then follows from this equality and \eqref{RN_mSLE4_SLE4}.
\end{proof}

The second property of massive SLE$_4$ which follows from Theorem \ref{theorem_mSLE_4} is its conformal covariance. This property was already shown in \cite{mHE} by using the explicit expression of the driving function of massive SLE$_4$, when the curve is conformally mapped to the complex upper-half plane.

\begin{lemma} \label{lemma_conformal_cov_mSLE}
Let $D \subset \mathbb{C}$ be a bounded, open and simply connected domain and let $a, b \in \partial D$. Let $m: D \to \mathbb{R}_{+}$ be a bounded and continuous function. Let $\varphi: D \to \tilde D$ be a conformal map. If a curve $\gamma$ has the law of massive SLE$_4$ in $D$ from $a$ to $b$ with mass $m$, then $\varphi(\gamma)$ has the law of massive SLE$_4$ in $\tilde D$ from $\varphi(a)$ to $\varphi(b)$ with mass $m$ given by, for $w \in \tilde D$,
\begin{equation*}
    \tilde m^2(w) = \vert (\varphi^{-1})'(w) \vert ^2 m^2(\varphi^{-1}(w)). 
\end{equation*}
\end{lemma}

\begin{proof}
This follows from a change of variables. One can either make this change of variables at scale $\eps > 0$ in
\begin{equation*}
    \frac{1}{\mathcal{Z}} \EE \bigg[ \exp \bigg( -\frac{1}{2} \int_{D} m^2(z) \big( (h_{\eps}(z)+\phi_{\eps}(z))^2 -\EE[h_{\eps}(z)^2] \big) dz \bigg) \vert \gamma([0,t]) \bigg]
\end{equation*}
and then take a limit as $\eps$ to $0$ or use the expression for the Radon-Nikodym derivative of $\PP_{\operatorname{mSLE}_4}^{(D,a,b)}$ with respect to $\PP_{\operatorname{SLE}_4}^{(D,a,b)}$ given by Corollary \ref{cor_RN_mSLE4}. Since SLE$_4$ is conformally invariant, if this Radon-Nikodym transforms as claimed in the statement of the lemma under the action of a conformal map, then the conformal covariance of massive SLE$_4$ in the sense stated in the lemma follows. We choose this second strategy here as it is more direct. By conformal invariance of the Brownian loop measure, see Section \ref{sec_BLS}, it is easy to see that, for $t \geq 0$, almost surely,
\begin{equation*}
    \mu_{t} \big( e^{-\langle \ell, m^2 \rangle} + \langle \ell, m^2 \rangle -1 \big) = \tilde \mu_{\tilde D_t} \big(  e^{-\langle \ell, \tilde m^2 \rangle} + \langle \ell, \tilde m^2 \rangle -1 \big)
\end{equation*}
where we have set $\tilde D_t = \varphi(D_t)$ and $\tilde m$ is as in the statement of Lemma \ref{lemma_conformal_cov_mSLE}. By conformal invariance of harmonic functions and conformal covariance of massive harmonic functions, we also have that, for $t \geq 0$, almost surely,
\begin{align*}
    \int_{D_t} m^2(z) \phi_t(z) \phi_t^m(z) dz = \int_{\tilde D_t} \vert (\varphi^{-1})'(w) \vert^2 m^2(\varphi^{-1}(w)) \tilde \phi_t(w) \tilde \phi_t^{\tilde m}(w) dw 
    = \int_{\tilde D_t} \tilde m^2(w) \tilde \phi_t(w) \tilde \phi_t^{\tilde m}(w) dw
\end{align*}
where $\tilde \phi_t$, respectively $\tilde \phi_t^{\tilde m}$, the harmonic function, respectively massive harmonic function with mass $\tilde m$, with boundary values $\lambda$ on $\partial \tilde D_t^{+}$ and $-\lambda$ on $\partial \tilde D_t^{-}$. Moreover, we have that, for $t \geq 0$, almost surely,
\begin{equation*}
    \int_{D_t} m^2(z) \log \frac{\CR(z, \partial D)}{\CR(z, \partial D_t)} dz = \int_{\tilde D_t} \tilde m^2(w) \log \frac{\CR(w, \partial \tilde D)}{\CR(w, \partial \tilde D_t)} dw.
\end{equation*}
Note that taking $t=0$ in the above equalities gives the behavior of the partition function $\mathcal{Z}$ of Corollary \ref{cor_RN_mSLE4} under the action of the conformal map $\varphi$. By Corollary \ref{cor_RN_mSLE4} and conformal invariance of SLE$_4$, this completes the proof of the lemma.
\end{proof}

\section{Massive CLE$_4$ via change of measure} \label{sec_mCLE4}

In this section, we prove Theorem \ref{theorem_mCLE_4} and its corollaries, Corollary \ref{cor_RN_mCLE4} and Corollary \ref{cor_conformal_cov_mCLE4}. Theorem \ref{theorem_mCLE_4} is established in Section \ref{sec_ThCLE}. There, we first identify the conditional law under $\tilde \PP$ of the random variables $(\xi_j)_j$ given $\Gamma$. Then, we turn to the identification of the conditional law under $\tilde \PP$ of the fields $(h_{j}+\xi_{j})_j$ given $\Gamma$. We will explain how to use these results to prove Theorem \ref{theorem_mCLE_4}. In the proof, we will assume that the domain $D$ is bounded and that the mass $m: D \to \mathbb{R}_{+}$ is a bounded and continuous function. Conformal covariance then enables us to extend the results to pairs $(D,m)$ satisfying the more general assumptions of Theorem \ref{theorem_mCLE_4}. Finally, in Section \ref{sec_cor_CLE4}, we establish Corollary \ref{cor_RN_mCLE4} and Corollary \ref{cor_conformal_cov_mCLE4}.

\subsection{Level lines of the massive GFF} \label{sec_ThCLE}

In this section, to prove Theorem \ref{theorem_mCLE_4}, we investigate how the change of measure \eqref{RN_mCLE4} defining the probability measure $\tilde \PP$ in Theorem \ref{theorem_mCLE_4} affects the law of the field $h$. We first observe that the change of measure \eqref{RN_mCLE4} is well-defined. Indeed, this is just the change of measure \eqref{RN_mGFF_Dirichlet}, which extends to a change of measure on $(h,\Gamma)$ since $\Gamma$ is measurable with respect to $h$ under $\PP_{\operatorname{GFF}}$. This implies in particular that the marginal law of $h$ under $\tilde \PP$ is that of a massive GFF in $D$ with mass $m$ and Dirichlet boundary conditions.

The next step to prove Theorem \ref{theorem_mCLE_4} is to understand how the field $h$ behaves under $\tilde \PP$ when conditioning on $\Gamma$. By definition of $\tilde \PP$ via the change of measure \eqref{RN_mCLE4}, the conditional law of $h$ under $\tilde \PP$ given $\Gamma$ is such that, for $F: \mathcal{S}'(\mathbb{R}^2) \to \mathbb{R}$ a bounded and continuous function, almost surely, 
\begin{equation} \label{RN_Gamma}
    \tilde \EE[F(h) \vert \Gamma] = \frac{1}{\mathcal{Z}(\Gamma)} \EE \bigg[ F(h) \exp \bigg( -\frac{1}{2} \int_{D} m^2(z) :h^2(z): dz \bigg) \vert \Gamma \bigg]
\end{equation}
where we have set
\begin{equation} \label{def_rdm_partCLE}
    \mathcal{Z}(\Gamma) := \EE \bigg[ \exp \bigg( -\frac{1}{2}\int_{D} m^2(z) :h^2(z): dz \bigg) \vert \Gamma \bigg].
\end{equation}
Note that the random partition function $\mathcal{Z}(\Gamma)$ is almost surely strictly positive: this again follows from the fact that if $X$ is an almost surely strictly positive random variable and $\mathcal{F}$ is a $\sigma$-algebra, then $\EE[X \vert \mathcal{F}]$ is almost surely strictly positive. This implies that the right-hand side of \eqref{RN_Gamma} is well-defined and thus that also the conditional law of $h$ given $\Gamma$ under $\tilde \PP$ is well-defined.

Since under $\PP$, conditionally on $\Gamma$, $h=\sum_j h_j+\xi_j$ where $(h_j)_j$ and $(\xi_j)_j$ are as described in Section \ref{sec_coupling_CLE}, the random partition function $\mathcal{Z}(\Gamma)$, or more generally weighted conditional expectations of the form
\begin{equation} \label{weighted_EE_Gamma}
    \EE \bigg[ \exp(i(h,f)) \exp \bigg( -\frac{1}{2}\int_{D} m^2(z) :h^2(z): dz \bigg) \vert \Gamma \bigg]
\end{equation}
where $f: D \to \mathbb{R}$ is a bounded and continuous function, should have an expression in terms of the fields $(h_j)_j$ and the labels $(\xi_j)_j$. As a first step toward the identification of the conditional law of $h$ given $\Gamma$ under $\tilde \PP$, let us establish such an expression.

\subsubsection{The weighted conditional expectation}

Observe that by absolute continuity of $\tilde \PP$ with respect to $\PP$, under $\tilde \PP$, the marginal law of $\Gamma$ is such that $\Gamma$ is almost surely a countable collection of planar simple loops which do not touch each other or the boundary of $D$. Indeed, under $\PP$, the marginal law of $\Gamma$ is that of a CLE$_4$ in $D$, so that these properties are satisfied by $\Gamma$ under $\PP$.

With these preliminary remarks made, let us now show how to express the weighted conditional expectations of the form \eqref{weighted_EE_Gamma} in terms of the fields $(h_j)_j$ and the labels $(\xi_j)_j$.

\begin{proposition} \label{prop_rep_CLE}
Under the same assumptions as in Theorem \ref{theorem_mCLE_4} and with the same notations, let $f: D \to \mathbb{R}$ be a bounded and continuous function. Then, almost surely,
\begin{align*}
    &\EE \bigg[ \exp(i(h,f)) \exp \bigg( -\frac{1}{2}\int_{D} m^2(z) :h^2(z): dz \bigg) \vert \Gamma \bigg] \\
    &= \prod_{j} \begin{aligned}[t]&\EE \bigg[ \exp(i(h_j+\xi_j, f)) \exp \bigg( -\frac{1}{2} \int_{\Lj} m^2(z) :(h_j+\xi_j)^2(z): dz \bigg) \vert \Gamma \bigg] \\
    &\times \exp\bigg(\int_{\Lj} \frac{m^2(z)}{4\pi} \log \frac{\operatorname{CR}(z, \partial D)}{\operatorname{CR}(z, L_j)} dz \bigg)\end{aligned}
\end{align*}
where the product runs over the loops $(L_j)_j$ of $\Gamma$ and $\CR(z,\partial D)$, respectively $\CR(z,L_{j})$, denotes the conformal radius of $z$ in $D$, respectively in $\Lj$. Above, the fields $(h_j)_j$ and the random variables $(\xi_j)_j$ are as described in Section \ref{sec_coupling_CLE} and the random variables $(\int_{\Lj} m^2(z) :(h_j+\xi_j)^2(z): dz )_{j}$ are as defined via the notation \eqref{notation_Wick_boundary}.
\end{proposition}

\begin{proof}
Let $f: D \to \mathbb{R}$ be a bounded and continuous function. By Lemma \ref{lemma_Lp_cvg_boundary} and since almost surely $\vert \exp(i(h,f))\vert \leq 1$, we have that
\begin{align*}
    \exp(i(h,f))\exp \bigg(-\frac{1}{2} \int_{D} m^2(z):h(z)^2:dz \bigg)
    =\lim_{\eps \to 0} \exp(i(h,f)) \exp \bigg( -\frac{1}{2} \int_{D} m^2(z) \big(h_{\eps}(z)^2-\EE[h_{\eps}(z)^2] \big) dz \bigg)
\end{align*}
where the limit is in $\operatorname{L}^1(\PP)$. This yields that, almost surely,
\begin{align*}
    &\EE \bigg[ \exp(i(h,f)) \exp \bigg(-\frac{1}{2} \int_{D} m^2(z):h(z)^2:dz \bigg) \vert \Gamma \bigg] \\
    &= \lim_{\eps \to 0} \EE \bigg[ \exp(i(h,f)) \exp \bigg( -\frac{1}{2} \int_{D} m^2(z) \big(h_{\eps}(z)^2-\EE[h_{\eps}(z)^2] \big) dz \bigg) \vert \Gamma \bigg]
\end{align*}
where again the limit is in $\operatorname{L}^1(\PP)$. Using the fact that, as explained in Section \ref{sec_coupling_CLE}, under $\PP$, conditionally on $\Gamma$, $h=\sum_j h_j+\xi_j$ with $(h_j)_j$ and $(\xi_j)_j$ as described there, we obtain that, almost surely,
\begin{align} \label{dec_condEE}
    &\EE \bigg[  \exp(i(h,f)) \exp \bigg( -\frac{1}{2}\int_{D} m^2(z) :h(z)^2: dz \bigg) \vert \Gamma \bigg] \nonumber \\
    &\begin{aligned} =\lim_{\eps \to 0} \EE \bigg[  &\exp(i(\sum_j h_j+\xi_j,f))
    \exp \bigg(-\frac{1}{2} \int_{D} m^2(z)  \bigg(\bigg(\sum_{j} h_{j,\eps}(z)+\xi_{j,\eps}(z) \bigg)^2 - \sum_j \EE[h_{j,\eps}(z)^2 \vert \Gamma] \bigg) dz \bigg) \vert \Gamma \bigg] \\
    &\times \exp \bigg( \frac{1}{2} \int_{D} m^2(z) \bigg( \EE[h_{\eps}(z)^2] - \sum_j \EE[h_{j,\eps}(z)^2 \vert \Gamma] \bigg) dz \bigg),\end{aligned}
\end{align}
where the limit is in $\operatorname{L}^1(\PP)$. Above, for $\eps >0$, $z \in D$ and any $j$, we have set
\begin{equation} \label{def_xij_eps}
    \xi_{j,\eps}(z) := \xi_{j}(\mathbb{I}_{\Lj}, \rho_{\eps}^z) = \frac{\xi_{j}}{2\pi \eps}\mathcal{L}\big(\partial B(z,\eps) \cap \Lj \big)
\end{equation}
where $\mathcal{L}$ denotes the length with respect to the arc-length measure. As before, for any $j$, $h_{j,\eps}(z)$ stands for the random variable $(h_{j}, \rho_{\eps}^z)$. Now, for $\eps > 0$, let us define
\begin{align*}
    &I_{\eps} := -\frac{1}{2} \int_{D} m^2(z)  \bigg(\bigg(\sum_{j} h_{j,\eps}(z)+\xi_{j,\eps}(z) \bigg)^2 - \sum_j \EE[h_{j,\eps}(z)^2 \vert \Gamma] \bigg) dz, \\
    &J_{\eps} := \frac{1}{2} \int_{D} m^2(z) \bigg( \EE[h_{\eps}(z)^2] - \sum_j \EE[h_{j,\eps}(z)^2 \vert \Gamma] \bigg) dz.
\end{align*}
To prove  Proposition \ref{prop_rep_CLE}, we must control the limit of $I_{\eps}$ and $J_{\eps}$ as $\eps$ tends to $0$. We claim the following. Below, we denote by $\PP_{\Gamma}$ the (random) conditional law induced by $\EE_{\Gamma}[\cdot]:=\EE[\cdot \vert \Gamma]$.

\begin{claim} \label{claim_cvg_Wick_term}
For $\PP$-almost every $\Gamma$,
\begin{equation*}
    \lim_{\eps \to 0} I_{\eps} = -\frac{1}{2} \sum_{j} \int_{\Lj} m^2(z) :(h_{j}+\xi_j)^2(z): dz
\end{equation*}
where the limit is in $L^2(\PP_{\Gamma})$.
\end{claim}

\begin{claim} \label{claim_CR}
Almost surely,
\begin{equation*}
    \lim_{\eps \to 0} J_{\eps} = \sum_j \int_{\Lj} \frac{m^2(z)}{4\pi} \log \frac{\CR(z, \partial D)}{\CR(z,L_j)} dz
\end{equation*}
where the sum on the right-hand side is almost surely finite.
\end{claim}

We postpone the proof of these claims to the end and show how to prove Proposition \ref{prop_rep_CLE} based on them. On the one hand, we know that $e^{I_{\eps}}e^{J_{\eps}}$ converges to $\exp(-\frac{1}{2}\int_{D}m^2(z):h(z)^2:dz)$ in $\PP$-probabilty as $\eps$ tends to $0$. On the other hand, for any fixed $\delta>0$ and for $\PP$-almost every $\Gamma$, $\PP_{\Gamma}(\vert e^{I_{\eps}} - e^{I_0} \vert > \delta)$ converges to $0$ as $\eps$ tends to $0$, where $I_{0}$ denotes the limiting random variable in Claim \ref{claim_cvg_Wick_term}. By dominated convergence, this implies that $e^{I_{\eps}}$ converges to $e^{I_0}$ in $\PP$-probability as $\eps$ tends to $0$. Together with Claim \ref{claim_CR}, this in turn yields that $e^{I_{\eps}}e^{J_{\eps}}$ converges to $e^{I_{0}}e^{J_{0}}$ in $\PP$-probability as $\eps$ tends to $0$, where $J_0$ is the limiting random variable in Claim \ref{claim_CR}. Combining this with our first observation, we obtain that, $\PP$-almost surely,
\begin{equation*}
	\exp\bigg(-\frac{1}{2}\int_{D} m^2(z) :h(z)^2: dz \bigg) = e^{I_{0}}e^{J_{0}}.
\end{equation*}
Proposition \ref{prop_rep_CLE} then follows from this and the independence of $(h_j)_j$ and $(\xi_{j})_j$ under $\PP_{\Gamma}$.
\end{proof}

It now remains to prove Claim \ref{claim_cvg_Wick_term} and Claim \ref{claim_CR}.

\begin{proof}[Proof of Claim \ref{claim_cvg_Wick_term}]
We first observe that, almost surely,
\begin{align*}
    I_{\eps}&=-\frac{1}{2} \bigg( \int_{D} m^2(z)\sum_{j}(h_{j,\eps}(z)^2- \EE[h_{j,\eps}(z)^2 \vert \Gamma]) dz + 2\int_{D} m^2(z)\sum_j h_{j,\eps}(z)\xi_{j,\eps}(z) dz \\
    &+ \int_{D} m^2(z) \sum_j \xi_{j,\eps}(z)^2 dz + 2\int_{D} m^2(z) \sum_{j} h_{j,\eps}(z)\sum_{k: k \neq j} h_{k,\eps}(z) dz \\
    &+ 2\int_{D} m^2(z) \sum_{j}h_{j,\eps}(z)\sum_{k:k\neq j}\xi_{j,\eps}(z) dz +\int_{D} m^2(z) \sum_{j} \sum_{k: k \neq j} \xi_{j,\eps}(z)\xi_{k,\eps}(z) dz \bigg).
\end{align*}
Conditionally on $\Gamma$, the first term should give rise to a collection of Wick squares of the GFFs $(h_j)_j$ in the interiors of the loops $(L_j)_j$ while the second term should yield a collection of GFFs tested against the function $z \mapsto m^2(z)\xi_j$ inside each loop. The third term should in turn give rise to a collection of random variables $(\xi_j^2(m^2, \mathbb{I}_{\Lj}))_j$, one for each loop. Moreover, intuitively, the three last terms should all converge to $0$ since in the limit as $\eps$ tends to $0$, we should obtain random functions/distributions supported on the interior of different loops. Making these rigorous is the content of the following claims.

\begin{claim} \label{claim_sumhj_square}
For $\PP$-almost every $\Gamma$,
\begin{equation*}
    \lim_{\eps \to 0} \int_{D} m^2(z) \big( \sum_{j} h_{j,\eps}(z)^2 - \EE_{\Gamma}[h_{j,\eps}(z)^2]\big) dz = \sum_j (:h_j^2:, m^2)
\end{equation*}
where the limit is in $L^2(\PP_{\Gamma})$.
\end{claim}

\begin{claim} \label{claim_sumhjxi}
For $\PP$-almost every $\Gamma$,
\begin{equation*}
    \lim_{\eps \to 0} \int_{D} m^2(z) \sum_{j} h_{j,\eps}(z)\xi_{j,\eps}(z) dz = \sum_j \xi_{j}(h_{j}, m^2)
\end{equation*}
where the limit is in $L^2(\PP_{\Gamma})$.
\end{claim}

\begin{claim} \label{claim_neq_sumhjxi}
For $\PP$-almost every $\Gamma$,
\begin{equation*}
    \lim_{\eps \to 0} \int_{D} m^2(z)\sum_{j} h_{j,\eps}(z) \sum_{k:k \neq j}\xi_{k,\eps}(z) dz = 0
\end{equation*}
where the limit is in $L^2(\PP_{\Gamma})$.
\end{claim}

\begin{claim} \label{claim_sumhj_neq}
For $\PP$-almost every $\Gamma$,
\begin{equation*}
    \lim_{\eps \to 0} \int_{D} m^2(z) \sum_j h_{j,\eps}(z) \sum_{k: k\neq j} h_{k,\eps}(z) dz = 0
\end{equation*}
where the limit is in $L^{2}(\PP_{\Gamma})$.
\end{claim}

\begin{claim} \label{claim_sumxi_square}
For $\PP$-almost every $\Gamma$,
\begin{equation*}
    \lim_{\eps \to 0} \int_{D} m^2(z) \sum_{j} \xi_{j,\eps}(z)^2 dz = \sum_{j} \xi_j^2 \int_{\Lj} m^2(z) dz
\end{equation*}
where the limit is in $L^{2}(\PP_{\Gamma})$.
\end{claim}

\begin{claim} \label{claim_productxi_neq}
For $\PP$-almost every $\Gamma$,
\begin{equation*}
    \lim_{\eps \to 0} \int_{D} m^2(z) \sum_{j} \xi_{j,\eps}(z)\sum_{k:k\neq j}\xi_{k,\eps}(z) dz = 0
\end{equation*}
where the limit is in $L^{2}(\PP_{\Gamma})$.
\end{claim}
Claim \ref{claim_cvg_Wick_term} straightforwardly follows from these claims (recall the notation \eqref{notation_Wick_boundary}).
\end{proof}

\begin{proof}[Proof of Claim \ref{claim_sumhj_square}]
Let $\eps > 0$. Since $\Gamma$ almost surely has Lebesgue measure $0$, the integral over the domain $D$ can be decomposed as a sum integrals over the interiors $(\Lj)_j$ of the loops $(L_j)_j$ of $\Gamma$. This yields that, almost surely,
\begin{align*}
    &\int_{D} m^2(z) \big( \sum_{j} h_{j,\eps}(z)^2 - \EE_{\Gamma}[h_{j,\eps}(z)^2]\big) dz \\
    &= \sum_j \int_{\Lj} m^2(z) \big( h_{j,\eps}(z)^2-\EE_{\Gamma}[ h_{j,\eps}(z)^2]\big) dz + \int_{D} m^2(z) \sum_{k: k \neq j(z)} h_{k,\eps}(z)^2-\EE_{\Gamma}[ h_{k,\eps}(z)^2] dz
\end{align*}
where for $z \in D$, $j(z)$ denotes the index of the loop of $\Gamma$ surrounding $z$, which exists for almost every $z$. Therefore, using the inequality $(a+b)^2 \leq 2a^2 + 2b^2$, we have that, almost surely,
\begin{align} \label{ineq_Wick_eps}
    & \EE_{\Gamma} \bigg[ \bigg( \int_{D} m^2(z) \big( \sum_{j} h_{j,\eps}(z)^2 - \EE_{\Gamma}[h_{j,\eps}(z)^2]\big) dz - \sum_j (:h_j^2:, m^2) \bigg)^2 \bigg] \nonumber \\
    &\leq 2 \EE_{\Gamma} \bigg[ \bigg( \sum_j \int_{\Lj} m^2(z) \big( h_{j,\eps}(z)^2-\EE_{\Gamma}[ h_{j,\eps}(z)^2]\big) dz - \sum_j (:h_j^2:, m^2) \bigg)^2 \bigg] \nonumber \\
    &+ 2 \EE_{\Gamma} \bigg[ \bigg( \int_{D} m^2(z) \sum_{k: k \neq j(z)} h_{k,\eps}(z)^2-\EE_{\Gamma}[ h_{k,\eps}(z)^2] dz \bigg)^2\bigg].
\end{align}
Let us show that the two conditional expectations on the right-hand side of this inequality almost surely converge to $0$ as $\eps$ tends to $0$. We start with the first one. We have that, almost surely,
\begin{align} \label{EE_sum_square}
    &\EE_{\Gamma} \bigg[ \bigg( \sum_{j} \int_{\Lj} m^2(z) \big(h_{j,\eps}(z)^2 - \EE_{\Gamma}[h_{j,\eps}(z)^2]\big) dz - (:h_j^2:, m^2) \bigg)^2 \bigg] \nonumber \\
    &= \sum_{j} \EE_{\Gamma}\bigg[ \bigg( \int_{\Lj} m^2(z) \big(h_{j,\eps}(z)^2 - \EE_{\Gamma}[h_{j,\eps}(z)^2]\big) dz-(:h_j^2:, m^2) \bigg)^2 \bigg]
\end{align}
since the cross-terms vanish due to the conditional independence of the fields $(h_j)_j$. Now, observe that, almost surely, for any $j$,
\begin{equation*}
    \lim_{\eps \to 0} \EE_{\Gamma}\bigg[ \bigg( \int_{\Lj} m^2(z) \big(h_{j,\eps}(z)^2 - \EE_{\Gamma}[h_{j,\eps}(z)^2]\big) dz-(:h_j^2:, m^2) \bigg)^2\bigg] = 0.
\end{equation*}
This follows from the fact that, under $\PP_{\Gamma}$, $h_j$ is a GFF with Dirichlet boundary conditions in $\Lj$ and Lemma \ref{lemma_L2_cvg_Wick_boundary}. So, by the dominated convergence theorem, to prove that the conditional expectation on the last line of \eqref{EE_sum_square} almost surely converges to $0$ as $\eps \to 0$, it suffices to show that this conditional expectation is almost surely uniformly bounded in $\eps$. Using \eqref{EE_square_GFF} for the fields $(h_j)_j$ under $\PP_{\Gamma}$, we have that almost surely, for any $j$,
\begin{align*}
    &\EE_{\Gamma}\bigg[ \bigg( \int_{\Lj} m^2(z) \big(h_{j,\eps}(z)^2 - \EE_{\Gamma}[h_{j,\eps}(z)^2]\big) dz-(:h_j^2:, m^2) \bigg)^2 \bigg] \\
    &\leq 2\int_{\Lj \times \Lj} m^2(z)m^2(w) \tilde G_{j,\eps}(z,w) dzdw + 4 \int_{\Lj \times \Lj} m^2(z)m^2(w) G_j(z,w)^2 dzdw
\end{align*}
where $\tilde G_{j,\eps}$ denotes the covariance of the field $h_{j,\eps}^2 - \EE_{\Gamma}[h_{j,\eps}(\cdot)^2]$ under $\PP_{\Gamma}$ and $G_j$ is the Green function in $\Lj$. One can easily check that, almost surely, for any $z, w \in D$, $\tilde G_{j,\eps}(z,w) = 2\EE[h_{j,\eps}(z)h_{j,\eps}(w)]^2$. This implies that, almost surely, for any $z,w \in D$,
\begin{equation*}
    \tilde G_{j,\eps}(z,w) \leq 2G_{j}(z,w)^2 + O(1) \leq 2G_{D}(z,w)^2 + O(1),
\end{equation*}
where the implied constant is deterministic and independent of $j$. We deduce from this that, almost surely,
\begin{align} \label{ineq_EE_j_Gj}
    &\sum_j \EE_{\Gamma}\bigg[ \bigg( \int_{\Lj} m^2(z) \big(h_{j,\eps}(z)^2 - \EE_{\Gamma}[h_{j,\eps}(z)^2]\big) dz-(:h_j^2:, m^2) \bigg)^2\bigg] \nonumber \\
    &\leq 8\int_{D \times D} m^2(z)m^2(w) (2G_D(z,w)^2+O(1)) dzdw.
\end{align}
The final inequality follows from the fact that $G_{j}(z,w)=0$ if $j(z) \neq j(w)$ and $G_{D}(z,w)^2 \geq G_j(z,w)^2$ if $j(z)=j(w)$. We also used the fact that, almost surely, $\sum_j \Area(\Lj)^2 \leq \Area(D)^2$. As explained above, this shows that the first conditional expectation on the right-hand side of the inequality \eqref{ineq_Wick_eps} almost surely converges to $0$ as $\eps$ tends to $0$. Let us now show that the second conditional expectation on the right-hand side of \eqref{ineq_Wick_eps} also almost surely converges to $0$. For $\eps > 0$, define the set
\begin{equation} \label{def_Geps}
    G_{\eps} :=\{ z \in D: \operatorname{dist}(z, L(z)) \leq \eps \}
\end{equation}
where for $z \in D$, $L(z)$ denotes the loop of $\Gamma$ surrounding $z$, which almost surely exists. We then have that, almost surely,
\begin{align*}
    &\EE_{\Gamma} \bigg[ \bigg( \int_{D} m^2(z) \sum_{k: k \neq j(z)} h_{k,\eps}(z)^2-\EE_{\Gamma}[ h_{k,\eps}(z)^2] dz \bigg)^2 \bigg]
    = \EE_{\Gamma} \bigg[ \bigg( \int_{G_{\eps}} m^2(z) \sum_{k: k \neq j(z)} h_{k,\eps}(z)^2-\EE_{\Gamma}[ h_{k,\eps}(z)^2] dz \bigg)^2 \bigg].
\end{align*}
Indeed, for $z \in D \setminus G_{\eps}$, the integrand inside the conditional expectation is almost surely equal to $0$ as the fields $(h_{k})_{k \neq j(z)}$ are almost surely equal to $0$ in $\operatorname{Int}(L(z))$. Note also that the set of points not surrounded by any loop of $\Gamma$ almost surely has Lebesgue measure $0$, so we can ignore it when integrating. Pursuing our computations, we then have that, almost surely,
\begin{align*}
    &\EE_{\Gamma} \bigg[ \bigg( \int_{D} m^2(z) \sum_{k: k \neq j(z)} h_{k,\eps}(z)^2-\EE_{\Gamma}[ h_{k,\eps}(z)^2] dz \bigg)^2 \bigg] \\
    &= \EE_{\Gamma} \bigg[ \int_{G_{\eps} \times G_{\eps}} m^2(z) m^2(w) \sum_{\substack{k: k \neq j(z) \\ p: p \neq j(w)}} (h_{k,\eps}(z)^2-\EE_{\Gamma}[ h_{k,\eps}(z)^2]) (h_{p,\eps}(w)^2-\EE_{\Gamma}[ h_{p,\eps}(w)^2]) dz dw \bigg] \\
    &= \int_{G_{\eps} \times G_{\eps}} m^2(z) m^2(w) \sum_{\substack{k: k \neq j(z) \\ p: p \neq j(w)}} \EE_{\Gamma} \big[ (h_{k,\eps}(z)^2-\EE_{\Gamma}[ h_{k,\eps}(z)^2]) (h_{p,\eps}(w)^2-\EE_{\Gamma}[ h_{p,\eps}(w)^2])\big] dz dw,
\end{align*}
where the last equality follows from Fubini theorem applied term by term. In the sum above, by conditional independence of $h_k$ and $h_p$ given $\Gamma$, only the terms $k=p$ contribute. This yields that, almost surely,
\begin{align*}
    &\EE_{\Gamma} \bigg[ \bigg( \int_{D} m^2(z) \sum_{k: k \neq j(z)} h_{k,\eps}(z)^2-\EE_{\Gamma}[ h_{k,\eps}(z)^2] dz \bigg)^2 \bigg] \\
    &= \int_{G_{\eps} \times G_{\eps}} m^2(z) m^2(w) \sum_{k: k \neq j(z), j(w)} \EE_{\Gamma} \big[ (h_{k,\eps}(z)^2-\EE_{\Gamma}[ h_{k,\eps}(z)^2 ]) (h_{k,\eps}(w)^2-\EE_{\Gamma}[ h_{k,\eps}(w)^2]) \big] dz dw.
\end{align*}
Since conditionally on $\Gamma$, for each $k$, $h_k$ is a GFF with Dirichlet boundary conditions in $\Lk$, we obtain from this equality that, almost surely,
\begin{align*}
    &\EE_{\Gamma} \bigg[ \bigg( \int_{D} m^2(z) \sum_{k: k \neq j(z)} h_{k,\eps}(z)^2-\EE_{\Gamma}[ h_{k,\eps}(z)^2] dz \bigg)^2 \bigg]
    = \int_{G_{\eps} \times G_{\eps}} m^2(z) m^2(w) \sum_{k: k \neq j(z), k \neq j(w)} \tilde G_{k,\eps}(z,w) dz dw
\end{align*}
where, as before, for each $k$, $\tilde G_{k,\eps}$ denotes the covariance function of the field $h_{k,\eps}^2 - \EE_{\Gamma}[h_{k,\eps}(\cdot)^2]$ under $\PP_{\Gamma}$. For the same reasons as those in the first part of the proof, we get that, almost surely,
\begin{align*}
    \EE_{\Gamma} \bigg[ \bigg( \int_{D} m^2(z) \sum_{k: k \neq j(z)} h_{k,\eps}(z)^2-\EE_{\Gamma}[ h_{k,\eps}(z)^2] dz \bigg)^2\bigg]
    \leq \int_{G_{\eps} \times G_{\eps}} m^2(z) m^2(w) (2G_{D}(z,w)^2+O(1)) dz dw.
\end{align*}
Since the function $z, w \mapsto G_{D}(z,w)^2$ is integrable in $D \times D$ and $\Area(G_{\eps})$ almost surely converges to $0$ as $\eps$ tends to $0$, the right-hand side of the above inequality almost surely converges to $0$ as $\eps$ tends to $0$. Going back to \eqref{ineq_Wick_eps} and combining this with the first part of the proof, we deduce the claim.
\end{proof}

\begin{proof}[Proof of Claim \ref{claim_sumhjxi}]
Let $\eps > 0$. Since $\Gamma$ almost surely has Lebesgue measure $0$, the integral over $D$ can be written as a sum of integrals over the interiors of the loops of $\Gamma$. We then have that, almost surely,
\begin{align} \label{bound_sum_EE}
    &\EE_{\Gamma} \bigg[ \bigg( \int_{D} m^2(z) \sum_{j} h_{j,\eps}(z)\xi_{j,\eps}(z) dz- \sum_j \xi_{j}(h_{j}, m^2) \bigg)^2 \bigg] \nonumber \\
    &= \EE_{\Gamma}\bigg[  \bigg(\sum_{j} \int_{\Lj} m^2(z) h_{j,\eps}(z)\xi_{j,\eps}(z) dz- \sum_j \xi_{j}(h_{j}, m^2)
    + \sum_{j} \int_{\Lj} m^2(z)\sum_{k: k \neq j} h_{k,\eps}(z) \xi_{k,\eps}(z) dz \bigg)^2\bigg] \nonumber \\
    &\leq 2  \EE_{\Gamma}\bigg[  \bigg(\sum_{j} \int_{\Lj} m^2(z) h_{j,\eps}(z)\xi_{j,\eps}(z) dz- \sum_j \xi_{j}(h_{j}, m^2) \bigg)^2 \bigg] \nonumber \\
    &+2 \EE_{\Gamma} \bigg[ \bigg( \sum_{j} \int_{\Lj} m^2(z) \sum_{k: k \neq j} h_{k,\eps}(z) \xi_{k,\eps}(z) dz \bigg)^2 \bigg].
\end{align}
Moreover, we have that, almost surely,
\begin{align*}
    &\EE_{\Gamma}\bigg[  \bigg(\sum_{j} \int_{\Lj} m^2(z) h_{j,\eps}(z)\xi_{j,\eps}(z) dz- \sum_j \xi_{j}(h_{j}, m^2) \bigg)^2\bigg] \\
    &= \sum_j \EE_{\Gamma} \bigg[ \bigg( \int_{\Lj} m^2(z) h_{j,\eps}(z)\xi_{j,\eps}(z) dz -  \xi_{j}(h_{j}, m^2) \bigg)^2 \bigg].
\end{align*}
Indeed, by conditional independence of $(h_j)_j$ and $(\xi_j)_j$ given $\Gamma$, the cross-terms vanish. Going back to \eqref{bound_sum_EE}, this shows that, almost surely,
\begin{align} \label{sum_EE}
    &\EE_{\Gamma} \bigg[ \bigg( \int_{D} m^2(z) \sum_{j} h_{j,\eps}(z)\xi_{j,\eps}(z) dz- \sum_j \xi_{j}(h_{j}, m^2) \bigg)^2 \bigg] \nonumber \\
    &\leq 2 \sum_{j} \EE_{\Gamma} \bigg[ \bigg( \int_{\Lj} m^2(z)h_{j,\eps}(z)\xi_{j,\eps}(z) dz - \xi_j(h_j,m^2) \bigg)^2 \bigg] \nonumber \\
    &+2 \EE_{\Gamma} \bigg[ \bigg( \sum_{j} \int_{\Lj} m^2(z) \sum_{k: k \neq j} h_{k,\eps}(z) \xi_{k,\eps}(z) dz \bigg)^2 \bigg].
\end{align}
Let us now show that each of these conditional expectations almost surely converges to $0$ as $\eps \to 0$. For the first one, observe that for any $j$, almost surely,
\begin{equation*}
    \lim_{\eps \to 0} \EE_{\Gamma} \bigg[ \bigg( \int_{\Lj} m^2(z)h_{j,\eps}(z)\xi_{j,\eps}(z) dz - \xi_j(h_j,m^2) \bigg)^2 \bigg] = 0.
\end{equation*}
This follows from the fact that under $\PP$, conditionally on $\Gamma$, $h_j$ is a GFF with Dirichlet boundary conditions in the interior $\Lj$ of the loop $L_j$. Therefore, by the dominated convergence theorem, to prove that the first conditional expectation on the right-hand side of \eqref{sum_EE} almost surely converges to $0$ as $\eps \to 0$, it suffices to show that this conditional expectation is almost surely uniformly bounded in $\eps$. To this end, observe that, almost surely,
\begin{align} \label{ineq_hj_xij_eps}
    &\EE_{\Gamma} \bigg[ \bigg( \int_{\Lj} m^2(z)h_{j,\eps}(z)\xi_{j,\eps}(z) dz - \xi_j(h_j,m^2) \bigg)^2 \bigg] \nonumber \\
    &\leq 8\lambda^2\int_{\Lj} m^2(z)m^2(w) G_{j,\eps}(z,w) dz dw + 8\lambda^2 \int_{\Lj} m^2(z)m^2(w) G_{j}(z,w) dzdw
\end{align}
where $G_{j,\eps}$ denotes the covariance of the field $h_{j,\eps}$ when conditioning on $\Gamma$ (recall that under $\PP$, conditionally on $\Gamma$, $h_j$ is a GFF in $\Lj$ with Dirichlet boundary conditions) and $G_{j}$ is the Green function in $\Lj$. To obtain this inequality, we also used the fact that under $\PP$, almost surely, for any $j$, $\xi_{j}^2=4\lambda^2$, which implies that for any $z \in D$, almost surely, $\xi_{j,\eps}^2(z) \leq 4\lambda^2$. Moreover, almost surely, for any $z,w \in D$, we have that $G_{j,\eps}(z,w) \leq G_j(z,w) + O(1) \leq G_D(z,w) + O(1)$, where the implied constant is deterministic and independent of $j$. This yields that, almost surely,
\begin{align*}
    &\sum_j \EE_{\Gamma} \bigg[ \bigg( \int_{\Lj} m^2(z)h_{j,\eps}(z)\xi_{j,\eps}(z) dz - \xi_j(h_j,m^2) \bigg)^2 \bigg]
    \leq 16\lambda^2 \int_{D \times D} m^2(z)m^2(w)(G_{D}(z,w) + O(1)) dz dw,
\end{align*}
since, almost surely, $G_{j}(z,w)=0$ if $j(z) \neq j(w)$ and $G_{D}(z,w) \geq G_j(z,w)$ if $j(z)=j(w)$. This establishes the convergence to $0$ as $\eps \to 0$ of the first expectation on the right-hand side of \eqref{sum_EE}. Let us now turn to the second expectation on the right-hand side of \eqref{sum_EE}. Let us define, for $\eps > 0$, the set $G_{\eps}$ as in \eqref{def_Geps}. Observe that, almost surely,
\begin{align*}
    \EE_{\Gamma} \bigg[ \bigg( \sum_j \int_{\Lj} m^2(z) \sum_{k:k\neq j} h_{k,\eps}(z) \xi_{k,\eps}(z) dz\bigg)^2 \bigg]
    &= \EE_{\Gamma} \bigg[ \bigg( \int_{D} m^2(z) \sum_{j: z \notin \Lj} h_{j,\eps}(z)\xi_{j,\eps}(z) dz \bigg)^2  \bigg] \\
    &= \EE_{\Gamma} \bigg[ \bigg( \int_{G_{\eps}} m^2(z) \sum_{j: z \notin \Lj} h_{j,\eps}(z)\xi_{j,\eps}(z) dz \bigg)^2 \bigg].
\end{align*}
The last equality follows from the fact that if $z \in D \setminus G_{\eps}$, then almost surely $\sum_{j: z \notin \Lj} h_{j,\eps}(z)\xi_{j,\eps}(z) = 0$. This is because $h_{j}$ is almost surely $0$ outside of $\Lj$ and $\xi_{j,\eps}(z)$ is almost surely $0$ if $\operatorname{dist}(z, \Lj) > \eps$. Using the conditional independence of $(h_j)_j$ and $(\xi_j)_j$ together with the fact that conditionally on $\Gamma$, these random variables are centered, it then follows that, almost surely,
\begin{align*}
    &\EE_{\Gamma} \bigg[ \bigg( \int_{G_{\eps}} m^2(z) \sum_{j: z \notin \Lj} h_{j,\eps}(z)\xi_{j,\eps}(z) dz \bigg)^2 \bigg]\\
    &= \int_{G_{\eps} \times G_{\eps}} m^2(z)m^2(w) \sum_{j: z,w \notin \Lj} \EE_{\Gamma}[ h_{j,\eps}(z)h_{j,\eps}(w)] \EE_{\Gamma}[ \xi_{j,\eps}(z)\xi_{j,\eps}(w)] dz dw.
\end{align*}
Notice that almost surely, for any $z,w \in D$, $\EE_{\Gamma}[ \xi_{j,\eps}(z)\xi_{j,\eps}(w)] \leq 4\lambda^2$. This implies that, almost surely,
\begin{align} \label{ineq_zout}
    \EE_{\Gamma} \bigg[ \bigg( \int_{G_{\eps}} m^2(z) \sum_{j: z \notin \Lj} h_{j,\eps}(z)\xi_{j,\eps}(z) dz \bigg)^2 \bigg]  
    \leq 4\lambda^2\overline{m}^4 \int_{G_{\eps} \times G_{\eps}} \sum_{j: z,w \notin \Lj} \EE_{\Gamma}[ h_{j,\eps}(z)h_{j,\eps}(w)] dz dw.
\end{align}
Moreover, almost surely, for any $z, w \in D$,
\begin{equation} \label{ineq_Green_sum}
    \sum_{j: z,w \notin \Lj} \EE_{\Gamma}[ h_{j,\eps}(z)h_{j,\eps}(w)] \leq \EE[h_{\eps}(z)h_{\eps}(w)] \leq O(-\log \eps),
\end{equation}
where the right-most inequality follows from Cauchy-Schwarz inequality and the estimate \eqref{estimate_sup_variance}. The inequality \eqref{ineq_Green_sum} applied to the right-hand side of \eqref{ineq_zout} yields that, almost surely,
\begin{align} \label{bound_area_Geps}
    \EE_{\Gamma} \bigg[ \bigg( \sum_{j} \int_{\Lj} m^2(z) \sum_{k: k \neq j} h_{k,\eps}(z) \xi_{k,\eps}(z) dz \bigg)^2\ \bigg] \leq O(-\log \eps) \Area(G_{\eps})^2.
\end{align}
By \cite{dim_CLE, dim_CLE_2}, almost surely $\Area(G_\eps) = O(\eps^{2-\frac{15}{8}}) = O(\eps^{1/8})$, which implies that $O(-\log \eps) \Area(G_{\eps})^2$ almost surely converges to $0$ as $\eps \to 0$. This establishes the almost sure convergence to $0$ of the conditional expectation on the left-hand side of \eqref{bound_area_Geps} as $\eps$ tends to $0$ and completes the proof of the claim.
\end{proof}

\begin{proof}[Proof of Claim \ref{claim_neq_sumhjxi}]
Let $\eps > 0$. For $G_{\eps}$ defined as in \eqref{def_Geps}, notice that conditionally on $\Gamma$, if $z \in D \setminus G_{\eps}$, then, almost surely, $\sum_{j \neq k} h_{j,\eps}(z)\xi_{k,\eps}(z) = 0$. Indeed, since we are summing over indices $j \neq k$, conditionally on $\Gamma$, for $z \in D \setminus G_{\eps}$, if $z \in \Lj$, then $\sum_{k: k \neq j} \xi_{k,\eps}(z)$ and $\sum_{k: k \neq j} h_{k,\eps}(z)$ are almost surely equal to $0$. Therefore, it follows that, almost surely,
\begin{equation} \label{cvg_Geps}
    \EE_{\Gamma} \bigg[ \bigg( \int_{D} m^2(z)\sum_{j \neq k} h_{j,\eps}(z)\xi_{k,\eps}(z) dz \bigg)^2 \bigg] = \EE_{\Gamma} \bigg[ \bigg( \int_{G_{\eps}} m^2(z)\sum_{j \neq k} h_{j,\eps}(z)\xi_{k,\eps}(z) dz \bigg)^2\bigg].
\end{equation}
Let us now exhibit an almost sure upper bound for this conditional expectation. We have that, almost surely,
\begin{align*}
    &\EE_{\Gamma} \bigg[ \bigg( \int_{G_{\eps}} m^2(z)\sum_{j \neq k} h_{j,\eps}(z)\xi_{k,\eps}(z) dz \bigg)^2 \bigg]\\
    &=  \EE_{\Gamma} \bigg[ \int_{G_{\eps} \times G_{\eps}} m^2(z)m^2(w)\bigg(\sum_{j \neq k} h_{j,\eps}(z)\xi_{k,\eps}(z)\bigg) \bigg(\sum_{j \neq k} h_{j,\eps}(w)\xi_{k,\eps}(w)\bigg)  dz dw \bigg] \\
    &= \int_{G_{\eps} \times G_{\eps}} m^2(z)m^2(w) \EE_{\Gamma} \bigg[ \bigg(\sum_{j \neq k} h_{j,\eps}(z)\xi_{k,\eps}(z)\bigg) \bigg(\sum_{j \neq k} h_{j,\eps}(w)\xi_{k,\eps}(w)\bigg) \bigg] dz dw \\
    &= \binom{4}{2} \int_{G_{\eps} \times G_{\eps}} m^2(z)m^2(w) \EE_{\Gamma} \bigg[ \sum_{j \neq k} h_{j,\eps}(z)h_{j,\eps}(w)\xi_{k,\eps}(z) \xi_{k,\eps}(w)\bigg] dz dw.
\end{align*}
The last equality relies on the fact that the fields $(h_{j})_j$ and the labels $(\xi_j)_j$ are conditionally independent given $\Gamma$ and that all these random variables have mean $0$ conditionally on $\Gamma$. Using again this conditional independence property, we then obtain that, almost surely,
\begin{align} \label{ineq_Geps_h_xi_Gamma}
    \EE_{\Gamma} \bigg[ \bigg( \int_{G_{\eps}} m^2(z) \sum_{j \neq k} h_{j,\eps}(z)\xi_{k,\eps}(z) dz \bigg)^2 \bigg]
    \leq C \int_{G_{\eps} \times G_{\eps}} \sum_{j \neq k} \EE_{\Gamma}[ h_{j,\eps}(z)h_{j,\eps}(w)] \EE_{\Gamma}[ \xi_{k,\eps}(z) \xi_{k,\eps}(w)] dz dw,
\end{align}
where $C>0$ is a (non-random) constant depending only on $\overline{m}^2$. Observe that, for the same reasons as in \eqref{ineq_Green_sum}, for any $z,w \in G_{\eps}$, almost surely,
\begin{equation*}
    \sum_{j} \EE_{\Gamma}[ h_{j,\eps}(z)h_{j,\eps}(w)] \leq \EE[h_{\eps}(z)h_{\eps}(w)] = O(-\log \eps)
\end{equation*}
with constant independent of $z$ and $w$. Moreover, we have that, almost surely,
\begin{align} \label{EE_labels}
    \EE_{\Gamma}[ \xi_{k,\eps}(z) \xi_{k,\eps}(w)] &= 4\lambda^2 \EE_{\Gamma}[ (\mathbb{I}_{\operatorname{Int}(L_k)}, \rho_{\eps}^{z}) (\mathbb{I}_{\operatorname{Int}(L_k)}, \rho_{\eps}^{w})]
    = \frac{4\lambda^2}{4\pi^2\eps^2} \mathcal{L}(\partial B(z,\eps) \cap \operatorname{Int}(L_{k}))\mathcal{L}(\partial B(w,\eps) \cap \operatorname{Int}(L_{k})),
\end{align}
where as in \eqref{def_xij_eps}, $\mathcal{L}$ denote the arc-length measure. Observe also that almost surely, for any $k$ and any $w$, $(1/2\pi \eps) \operatorname{Length}(\partial B(w,\eps) \cap \operatorname{Int}(L_{k})) \leq 1 $. Moreover, it is easy to see that, since the loops of $\Gamma$ are almost surely disjoint, we have that, almost surely, $\frac{1}{2\pi\eps}\sum_{k} \mathcal{L}(\partial B(z,\eps) \cap \operatorname{Int}(L_{k})) \leq 1$. From \eqref{EE_labels} and these facts, we therefore obtain that, almost surely,
\begin{align*}
    \sum_{k} \EE_{\Gamma}[ \xi_{k,\eps}(z) \xi_{k,\eps}(w)] &= \frac{4\lambda^2}{4\pi^2\eps^2}\sum_k \mathcal{L}(\partial B(z,\eps) \cap \operatorname{Int}(L_{k}))\mathcal{L}(\partial B(w,\eps) \cap \operatorname{Int}(L_{k})) \\
    &\leq \frac{4\lambda^2}{2\pi\eps} \sum_{k} \mathcal{L}(\partial B(z,\eps) \cap \operatorname{Int}(L_{k})) \leq 4\lambda^2.
\end{align*}
In view of the inequality \eqref{ineq_Geps_h_xi_Gamma}, this yields that, almost surely,
\begin{equation*}
     \EE_{\Gamma} \bigg[ \bigg( \int_{G_{\eps}} m^2(z)\sum_{j \neq k} h_{j,\eps}(z)\xi_{k,\eps}(z) dz \bigg)^2\bigg] \leq \int_{G_{\eps} \times G_{\eps}} O(-\log \eps) dz dw.
\end{equation*}
By \cite{dim_CLE, dim_CLE_2}, $\Area(G_{\eps}) = O(\eps^{2-\frac{15}{8}})$ almost surely and therefore, the right-hand side of the above inequality almost surely converges to $0$ as $\eps \to 0$, which completes the proof of the claim.
\end{proof}

\begin{proof}[Proof of Claim \ref{claim_sumhj_neq}]
Let $\eps > 0$. With $G_{\eps}$ as in \eqref{def_Geps}, observe that if $z \in D \setminus G_{\eps}$, then $\sum_j h_{j,\eps}(z) \sum_{k: k\neq j} h_{k,\eps}(z) = 0$ almost surely. This is because for any $j$, the field $h_j$ is $0$ outside $\Lj$. Therefore, almost surely,
\begin{align} \label{cvg_neq_fields}
    &\EE_{\Gamma} \bigg[ \bigg( \int_{D} m^2(z) \sum_j h_{j,\eps}(z) \sum_{k: k\neq j} h_{k,\eps}(z) dz \bigg)^2 \bigg] \nonumber \\
    &= \EE_{\Gamma} \bigg[ \bigg( \int_{G_{\eps}} m^2(z) \sum_j h_{j,\eps}(z) \sum_{k: k\neq j} h_{k,\eps}(z) dz \bigg)^2\bigg] \nonumber \\
    &= \EE_{\Gamma} \bigg[ \int_{G_{\eps} \times G_{\eps}} m^2(z)m^2(w)\bigg( \sum_j h_{j,\eps}(z) \sum_{k: k\neq j} h_{k,\eps}(z) \bigg) \bigg( \sum_j h_{j,\eps}(w) \sum_{k: k\neq j} h_{k,\eps}(w)\bigg) dz dw \bigg] \nonumber \\
    &= \int_{G_{\eps}\times G_{\eps}} m^2(z)m^2(w)\EE_{\Gamma} \bigg[ \bigg( \sum_j h_{j,\eps}(z) \sum_{k: k\neq j} h_{k,\eps}(z) \bigg) \bigg( \sum_j h_{j,\eps}(w) \sum_{k: k\neq j} h_{k,\eps}(w)\bigg) \bigg] dz dw.
\end{align}
Observe that since the fields $(h_{j})_j$ are centered and conditionally independent given $\Gamma$, almost surely,
\begin{align*}
    &\EE_{\Gamma} \bigg[ \bigg( \sum_j h_{j,\eps}(z) \sum_{k: k\neq j} h_{k,\eps}(z) \bigg) \bigg( \sum_j h_{j,\eps}(w) \sum_{k: k\neq j} h_{k,\eps}(w)\bigg) \bigg]
    = \sum_{j \neq k} \EE_{\Gamma}[ h_{j,\eps}(z) h_{j,\eps}(w) ] \EE_{\Gamma}[ h_{k,\eps}(z) h_{k,\eps}(w)].
\end{align*}
Under $\PP_{\Gamma}$, for each $j$, $h_j$ is a GFF in $\Lj$ with Dirichlet boundary conditions and therefore, for the same reasons as in \eqref{ineq_Green_sum}, we have that, almost surely, for any $j$,
\begin{equation*}
    \sum_{k: k \neq j} \EE_{\Gamma}[ h_{k,\eps}(z) h_{k,\eps}(w)] \leq \EE[h_{\eps}(z)h_{\eps}(w)] \leq O(-\log \eps).
\end{equation*}
This yields that, almost surely,
\begin{equation} \label{EE_area_logsquare}
    \EE_{\Gamma} \bigg[ \bigg( \int_{D} m^2(z) \sum_j h_{j,\eps}(z) \sum_{k: k\neq j} h_{k,\eps}(z) dz \bigg)^2 \bigg] \leq O(\log(\eps)^2) \Area(G_{\eps})^2.
\end{equation}
By \cite{dim_CLE, dim_CLE_2}, $\Area(G_{\eps}) = O(\eps^{2-\frac{15}{8}})$ almost surely and therefore, $\log(\eps)^2\Area(G_{\eps})^2$ almost surely converges to $0$ as $\eps$ tends to $0$. By the inequality \eqref{EE_area_logsquare}, this completes the proof of the claim.
\end{proof}

\begin{proof}[Proof of Claim \ref{claim_sumxi_square}]
Observe first that since $D \setminus \Gamma$ almost surely has Lebesgue measure $0$, we have that, almost surely
\begin{align*}
    &\EE_{\Gamma} \bigg[ \bigg( \int_{D} m^2(z) \sum_{j} \xi_{j,\eps}(z)^2 dz - \sum_{j} \xi_j^2 \int_{\Lj} m^2(z) dz \bigg)^2 \bigg]\\
    &= \EE_{\Gamma} \bigg[ \bigg( \int_{D} m^2(z) \sum_{j} \xi_{j,\eps}(z)^2 dz - \int_{D} 4\lambda^2 m^2(z) dz \bigg)^2\bigg],
\end{align*}
where we have also used the fact that under $\PP_{\Gamma}$, for any $j$, $\xi_{j}^2=4\lambda^2$. Moreover, under $\PP_{\Gamma}$, for $z \in D \setminus G_{\eps}$, where $G_{\eps}$ is defined as in \eqref{def_Geps}, $\xi_{j,\eps}(z)^2 = 0$ unless $L_j=L(z)$, in which case $\xi_{j,\eps}(z)^2=4\lambda^2$. Here, as in the proof of Claim \ref{claim_sumhjxi}, for $z \in D$, we have denoted by $L(z)$ the loop of $\Gamma$ that surrounds $z$. Using Fubini-Tonelli theorem, this implies that, almost surely,
\begin{align*}
    \EE_{\Gamma} \bigg[ \bigg( \int_{D} m^2(z) \sum_{j} \xi_{j,\eps}(z)^2 dz - \sum_{j} \xi_j^2 \int_{\Lj} m^2(z) dz \bigg)^2 \bigg]
    = \EE_{\Gamma} \bigg[ \bigg( \int_{G_{\eps}} m^2(z) \sum_{j} \xi_{j,\eps}(z)^2 - 4\lambda^2 m^2(z) dz \bigg)^2 \bigg].
\end{align*}
By H\"older's inequality, we then have that, almost surely,
\begin{align*}
    \EE_{\Gamma} \bigg[ \bigg( \int_{D} m^2(z) \sum_{j} \xi_{j,\eps}(z)^2 dz - \sum_{j} \xi_j^2 \int_{\Lj} m^2(z) dz \bigg)^2 \bigg]
    \leq \Area(D) \EE_{\Gamma} \bigg[ \int_{G_{\eps}} m^4(z) \big( \sum_{j}  \xi_{j,\eps}(z)^2 - 4\lambda^2 \big)^2 dz \bigg].
\end{align*}
Using the inequality $(a-b)^2 \leq 2a^2 + 2b^2$ together with the boundedness of $m$, we then obtain that, almost surely,
\begin{align} \label{square_xi_gasket}
    &\EE_{\Gamma} \bigg[ \bigg( \int_{D} m^2(z) \sum_{j} \xi_{j,\eps}(z)^2 dz - \sum_{j} \xi_j^2 \int_{\Lj} m^2(z) dz \bigg)^2 \bigg] \nonumber \\
    &\leq \overline{m}^4\Area(D) \EE_{\Gamma} \bigg[ \int_{G_{\eps}} 2\bigg( \sum_{j}  \xi_{j,\eps}(z)^2 \bigg)^2 +32\lambda^4 dz \bigg] 
    = \overline{m}^4\Area(D) \int_{G_{\eps}} 2\EE_{\Gamma}\bigg[ \bigg( \sum_{j}  \xi_{j,\eps}(z)^2 \bigg)^2\bigg] +32\lambda^4 dz,
\end{align}
where the last equality follows from Fubini-Tonelli theorem. Notice that almost surely, for any $z \in D$,
\begin{equation*}
    \xi_{j,\eps}(z) = \xi_{j} (\mathbb{I}_{\Lj}, \rho_{\eps}^z) =\frac{\xi_j}{2\pi\eps}\mathcal{L}(\partial B(z,\eps) \cap \Lj),
\end{equation*}
where as in \eqref{def_xij_eps}, $\mathcal{L}$ denotes the length with respect to the arc-length measure. Moreover, as the loops are almost surely disjoint, we have that, almost surely, for any $z \in D$, 
\begin{equation} \label{sum_leq_1}
    \frac{1}{2\pi\eps}\sum_j \mathcal{L}(\partial B(z,\eps) \cap \Lj) \leq 1.
\end{equation}
This inequality implies that, almost surely, for any $z \in G_{\eps}$,
\begin{align*}
    \EE_{\Gamma}\bigg[ \bigg( \sum_{j}  \xi_{j,\eps}(z)^2 \bigg)^2\bigg] = 16\lambda^4 \EE_{\Gamma} \bigg[ \bigg( \frac{1}{2\pi \eps}\sum_{j} \mathcal{L}(\partial B(z,\eps) \cap \Lj) \bigg)^2 \bigg] \leq 16\lambda^4.
\end{align*}
Using this in \eqref{square_xi_gasket} yields that, almost surely,
\begin{align*}
    \EE_{\Gamma} \bigg[ \bigg( \int_{D} m^2(z) \sum_{j} \xi_{j,\eps}(z)^2 dz - \sum_{j} \xi_j^2 \int_{\Lj} m^2(z) dz \bigg)^2 \bigg] 
    \leq  C\Area(G_{\eps})
\end{align*}
where $C>0$ is a deterministic constant depending only on $D$, $\overline{m}^2$ and $\lambda^2$. By \cite{dim_CLE, dim_CLE_2}, $\Area(G_{\eps}) = O(\eps^{1/8})$. Therefore, we can deduce Claim \ref{claim_sumxi_square} from the above inequality.
\end{proof}

\begin{proof}[Proof of Claim \ref{claim_productxi_neq}]
Let $\eps > 0$. Observe that, under $\PP_{\Gamma}$, for $j\neq k$ and $z \in D \setminus G_{\eps}$, where $G_{\eps}$ is defined as in \eqref{def_Geps}, $\xi_{j,\eps}(z)\xi_{k,\eps}(z) = 0$. This implies that, almost surely,
\begin{align*}
    \EE_{\Gamma} \bigg[ \bigg( \int_{D} m^2(z) \sum_{j} \xi_{j,\eps}(z) \sum_{k: k \neq j} \xi_{j,\eps}(z) dz \bigg)^2 \bigg]
    = \EE_{\Gamma} \bigg[ \bigg( \int_{G_{\eps}} m^2(z) \sum_{j} \xi_{j,\eps}(z) \sum_{k: k \neq j} \xi_{j,\eps}(z) dz \bigg)^2 \bigg].
\end{align*}
Using H\"older inequality's and the boundedness of $m$, we then obtain that, almost surely,
\begin{align} \label{ineq_sum_square}
    \EE_{\Gamma} \bigg[ \bigg( \int_{D} m^2(z) \sum_{j} \xi_{j,\eps}(z) \sum_{k: k \neq j} \xi_{j,\eps}(z) dz \bigg)^2 \bigg]
    &\leq \overline{m}^4\Area(D) \EE_{\Gamma} \bigg[ \int_{G_{\eps}} \bigg( \sum_{j} \xi_{j,\eps}(z) \sum_{k: k \neq j} \xi_{j,\eps}(z) \bigg)^2 dz \bigg] \nonumber \\
    &= \overline{m}^4\Area(D) \int_{G_{\eps}} \EE_{\Gamma} \bigg[\bigg( \sum_{j} \xi_{j,\eps}(z) \sum_{k: k \neq j} \xi_{j,\eps}(z) \bigg)^2 \bigg] dz
\end{align}
where the last equality follows from Fubini-Tonelli theorem. Next, observe that, by the inequality \eqref{sum_leq_1} in the proof of Claim \ref{claim_sumxi_square}, almost surely, for any $z \in G_{\eps}$,
\begin{align*}
    \EE_{\Gamma} \bigg[ \bigg( \sum_{j,k} \xi_{j,\eps}(z) \xi_{k,\eps}(z) \bigg)^2\bigg]
    &= \EE_{\Gamma} \bigg[ \bigg( \sum_{j} \xi_{j,\eps}(z) \sum_{k} \frac{\xi_{k}}{2\pi\eps} \mathcal{L}\big( \partial B(z,\eps) \cap \Lk \big) \bigg)^2 \bigg] \\
    &\leq \EE_{\Gamma} \bigg[ \bigg( \sum_{j} 2\lambda \vert \xi_{j,\eps}(z) \vert \sum_{k} \frac{1}{2\pi\eps} \mathcal{L}\big( \partial B(z,\eps) \cap \Lk \big) \bigg)^2  \bigg] \\
    &\leq \EE_{\Gamma} \bigg[ \bigg( \sum_{j} 2\lambda \frac{\vert \xi_{j} \vert }{2\pi\eps} \mathcal{L}(\partial B(z,\eps) \cap \Lj) \bigg)^2 \bigg]
    \leq (4\lambda^2)^2.
\end{align*}
Using this in \eqref{ineq_sum_square} yields that, almost surely,
\begin{align*}
    \EE_{\Gamma} \bigg[ \bigg( \int_{D} m^2(z) \sum_{j} \xi_{j,\eps}(z) \sum_{k: k \neq j} \xi_{j,\eps}(z) dz \bigg)^2 \bigg] \leq 16\lambda^4\overline{m}^4\Area(D)\Area(G_{\eps}).
\end{align*}
By \cite{dim_CLE, dim_CLE_2}, $\Area(G_{\eps}) = O(\eps^{1/8})$ and therefore the claim follows from the above inequality.
\end{proof}

Let us now establish Claim \ref{claim_CR}.

\begin{proof}[Proof of Claim \ref{claim_CR}]
Let $\eps > 0$. We first observe that if $z \in D$ is such that $\operatorname{dist}(z,L(z))>\eps$ where as in the proof of Claim \ref{claim_sumhjxi}, $L(z)$ denotes the loop of $\Gamma$ surrounding $z$, then, almost surely,
\begin{equation*}
    \EE[h_{\eps}(z)^2] - \sum_{j} \EE_{\Gamma}[h_{j,\eps}(z)^2 ] = \EE[h_{\eps}(z)^2] - \EE_{\Gamma}[h_{j(z), \eps}(z)^2] = \frac{1}{2\pi} \log \frac{\CR(z, \partial D)}{\CR(z,L(z))},
\end{equation*}
where $h_{j(z)}$ denotes the field in the loop $L(z)$. This equality follows from the fact that under $\PP$, conditionally on $\Gamma$, for any $j\neq j(z)$, $h_{j,\eps}(z)=0$ almost surely since for $j \neq j(z)$, $h_j$ is almost surely equal to $0$ in $\operatorname{Int}(L(z))$. For $\eps > 0$, define the set $G_{\eps}$ as in \eqref{def_Geps}. Since the set $\{z \in D: z \text{ is not surrounded by a loop of $\Gamma$}\}$ almost surely has Lebesgue measure $0$, we then have that, almost surely,
\begin{align} \label{sum_CR}
    & \int_{D} m^2(z) (\EE[h_{\eps}(z)^2] - \sum_{j} \EE[h_{j,\eps}(z)^2 \vert \Gamma ]) dz \nonumber \\
    &=\int_{D \setminus G_{\eps}}  \frac{m^2(z)}{2\pi} \log \frac{\CR(z, \partial D)}{\CR(z,L(z))} dz
    +\int_{G_{\eps}} m^2(z) (\EE[h_{\eps}(z)^2] - \sum_{j} \EE_{\Gamma}[h_{j,\eps}(z)^2]) dz.
\end{align}
The second term in this sum almost surely converges to $0$ as $\eps$ tends to $0$. Indeed, almost surely,
\begin{align} \label{ineq_Geps_cov}
    0 \leq \int_{G_{\eps}}  m^2(z) (\EE[h_{\eps}(z)^2] - \sum_{j} \EE_{\Gamma}[h_{j,\eps}(z)^2]) dz \leq \overline{m}^2 \int_{G_{\eps}} \EE[h_{\eps}(z)^2] dz.
\end{align}
Above, the inequality on the left-hand side simply follows from the fact that, almost surely, for any $z \in D$, $\EE[h_{\eps}(z)^2] - \sum_{j} \EE_{\Gamma}[h_{j,\eps}(z)^2] \geq 0$. Moreover, the estimate \eqref{estimate_sup_variance} shows that there exists $C>0$ such that, for any $z \in D$, $\EE[h_{\eps}(z)^2] \leq C\log(\eps^{-1})$. Using this and the boundedness of $m$ in \eqref{ineq_Geps_cov} yields that, almost surely,
\begin{equation*}
    0 \leq  \int_{G_{\eps}}  m^2(z) (\EE[h_{\eps}(z)^2] - \sum_{j} \EE_{\Gamma}[h_{j,\eps}(z)^2]) dz \leq O(-\log\eps) \Area(G_{\eps}),
\end{equation*}
Since by \cite{dim_CLE, dim_CLE_2}, $\Area(G_{\eps})=O(\eps^{2-\frac{15}{8}})$ almost surely, we have that $(-\log\eps)\Area(G_{\eps})$ almost surely converges to $0$. This shows that the second term in the sum \eqref{sum_CR} indeed almost surely converges to $0$. For the first term in the sum \eqref{sum_CR}, by the monotone convergence theorem, almost surely,
\begin{equation*}
    \lim_{\eps \to 0} \int_{D \setminus G_{\eps}}  \frac{m^2(z)}{2\pi} \log \frac{\CR(z, \partial D)}{\CR(z,L(z))} dz = \int_{D}  \frac{m^2(z)}{2\pi} \log \frac{\CR(z, \partial D)}{\CR(z,L(z))} dz.
\end{equation*}
The random variable on the right-hand side is almost surely finite. Indeed, by \cite[Proposition~20]{BTLS}, for each $z \in D$, $\log \CR(z,\partial D) - \log \CR(z,L(z))$ has the law of $\tau_{-\pi, \pi}$, the first hitting time of $\{-\pi,\pi\}$ by a standard one-dimensional Brownian motion started at $0$. Moreover, for any $z \in D$, $\log \CR(z,\partial D) - \log \CR(z,L(z))$ is almost surely non-negative. Therefore, denoting by $T_{-\pi, \pi}$ the expectation of $\tau_{-\pi, \pi}$ and using Fubini-Tonelli theorem, we have that
\begin{equation*}
    0 \leq \EE \bigg[ \int_{D}  \frac{m^2(z)}{2\pi} \log \frac{\CR(z, \partial D)}{\CR(z,L(z))} dz \bigg] = \int_{D} \frac{m^2(z)}{2\pi} T_{-\pi, \pi} dz < \infty.
\end{equation*}
This implies that, almost surely,
\begin{equation*}
    \int_{D}  \frac{m^2(z)}{2\pi} \log \frac{\CR(z, \partial D)}{\CR(z,L(z))} dz < \infty.
\end{equation*}
Since $D \setminus \Gamma$ almost surely has Lebesgue measure $0$ and since the loops of $\Gamma$ are almost surely disjoint, the above integral over $D$ can be written as a sum of integrals over the interiors of the loops of $\Gamma$. This yields that, almost surely,
\begin{equation*}
    \int_{D}  \frac{m^2(z)}{2\pi} \log \frac{\CR(z, \partial D)}{\CR(z,L(z))} dz = \sum_{j} \int_{\Lj} \frac{m^2(z)}{2\pi} \log \frac{\CR(z,\partial D)}{\CR(z,L_j)} dz.
\end{equation*}
Going back to \eqref{sum_CR}, this completes the proof of the claim.    
\end{proof}

\subsubsection{The conditional law of $(\xi_j)_j$ under $\tilde \PP$}

Let us now show that under $\tilde \PP$, conditionally on $\Gamma$, the random variables $(\xi_j)_j$ are independent and Rademacher distributed on $\{-2\lambda, 2\lambda\}$.

\begin{proposition} \label{prop_xi_tildeP}
With the same assumptions as in Theorem \ref{theorem_mCLE_4}, under $\tilde \PP$, conditionally on $\Gamma$, the random variables $(\xi_j)_j$ are independent and almost surely, for any $j$,
\begin{equation*}
    \tilde \PP(\xi_{j} = -2\lambda \vert \Gamma) = \tilde \PP(\xi_{j} = 2\lambda \vert \Gamma) = \frac{1}{2}.
\end{equation*}
\end{proposition}

\begin{proof}
Observe that under $\PP$, almost surely, for any $j$, $\xi_j$ takes values in $\{-2\lambda, 2\lambda\}$. Since $\tilde \PP$ is absolutely continuous with respect to $\PP$, this also holds under $\tilde \PP$. To identify the law of $(\xi_j)_j$ given $\Gamma$ under $\tilde \PP$, it thus suffices to show that for any $a \in \mathbb{R}$, almost surely,
\begin{equation} \label{symmetry_tilde_PP}
    \tilde \EE[e^{ia\xi_j} \vert \Gamma] = \tilde \EE[e^{-ia\xi_j} \vert \Gamma].
\end{equation}
Indeed, it follows from this equality and the fact that almost surely $\xi_j \in \{-2\lambda, 2\lambda\}$ that, almost surely,
\begin{equation*}
    \tilde \PP(\xi_{j} = -2\lambda \vert \Gamma) = \tilde \PP(\xi_{j} = 2\lambda \vert \Gamma) = \frac{1}{2}.
\end{equation*}
Thus, let $a \in \mathbb{R}$. To establish \eqref{symmetry_tilde_PP}, we are going to make use of the symmetry properties of the Wick square and of the conditional laws of $(h_j)_j$ and $(\xi_{j})_j$ conditionally on $\Gamma$ under $\PP$. We first observe that, almost surely,
\begin{align*}
    \tilde \EE[e^{-ia\xi_j} \vert \Gamma] = \frac{1}{\mathcal{Z}(\Gamma)} \EE \bigg[ e^{-ia\xi_j} \exp \bigg(-\frac{1}{2} \int_{D} m^2(z) :h^2(z): dz \bigg) \vert \Gamma \bigg]
\end{align*}
where, as before, we have set
\begin{equation*}
    \mathcal{Z}(\Gamma) := \EE \bigg[ \exp \bigg(-\frac{1}{2} \int_{D} m^2(z) :h^2(z): dz \bigg) \vert \Gamma \bigg].
\end{equation*}
Proposition \ref{prop_rep_CLE} applied with $f \equiv 0$ shows that, almost surely,
\begin{equation} \label{prod_Z_Gamma_bis}
    \mathcal{Z}(\Gamma) = \prod_j \mathcal{Z}_j(\Gamma) \exp\bigg( \int_{\Lj} \frac{m^2(z)}{4\pi} \log \frac{\CR(z,\partial D)}{\CR(z,L_j)} dz \bigg)
\end{equation}
where the random variables $(\mathcal{Z}_{j}(\Gamma))_j$ are defined as 
\begin{equation*}
    \mathcal{Z}_{j}(\Gamma) := \EE \bigg[ \exp \bigg(-\frac{1}{2} \int_{\Lj} m^2(z) :(h_j+\xi_j)^2(z): dz \bigg) \vert \Gamma \bigg].
\end{equation*}
On the other hand, one can argue as in the proof of Proposition \ref{prop_rep_CLE} to show that, almost surely,
\begin{align} \label{numerator_xij}
    &\EE \bigg[ e^{-ia\xi_j} \exp \bigg(-\frac{1}{2} \int_{D} m^2(z) :h^2(z): dz \vert \Gamma \bigg] \nonumber \\
    &= \lim_{\eps \to 0} \EE \bigg[ e^{-ia\xi_j} \exp \bigg(-\frac{1}{2} \int_{D} m^2(z) :h_{\eps}^2(z): dz\bigg) \vert \Gamma \bigg] \nonumber \\
    &= \begin{aligned}[t]&\exp \bigg( \sum_j \int_{\Lj} \frac{m^2(z)}{4\pi} \log \frac{\CR(z,\partial D)}{\CR(z,L_j)} dz\bigg)\\
    &\times\EE \bigg[ e^{-ia\xi_j} \exp \bigg(-\frac{1}{2} \int_{\Lj} m^2(z) :(h+\xi_j)^2(z): dz \bigg) \vert \Gamma \bigg] \prod_{k:k\neq j} \mathcal{Z}_j(\Gamma).\end{aligned}
\end{align}
This yields that, almost surely,
\begin{equation} \label{charac_xij_simplified}
    \tilde \EE[e^{-ia\xi_j} \vert \Gamma] = \frac{1}{\mathcal{Z}_j(\Gamma)} \EE \bigg[ e^{-ia\xi_j} \exp \bigg(-\frac{1}{2} \int_{D} m^2(z) :(h_j+\xi_j)^2(z): dz \bigg) \vert \Gamma \bigg].
\end{equation}
Indeed, the terms involving the conformal radii stemming from the product decomposition \eqref{prod_Z_Gamma_bis} of $\mathcal{Z}(\Gamma)$ cancel those appearing in \eqref{numerator_xij}. Since under $\PP$, conditionally on $\Gamma$, $\xi_j$ has the same law as $-\xi_j$, we then obtain that, almost surely,
\begin{equation*}
     \tilde \EE[e^{-ia\xi_j} \vert \Gamma] = \frac{1}{\mathcal{Z}_j(\Gamma)} \EE \bigg[ e^{ia\xi_j} \exp \bigg(-\frac{1}{2} \int_{D} m^2(z) :(h_j-\xi_j)^2(z): dz \bigg) \vert \Gamma \bigg].
\end{equation*}
It is also easy to see that if $h$ is a GFF with Dirichlet boundary conditions and $c \in \mathbb{R}$, then
\begin{equation*}
    \int_{D} m^2(z) :(h-c)^2(z): dz \overset{\operatorname{law}}{=} \int_{D} m^2(z) :(-h+c)^2(z): dz.
\end{equation*}
This simply follows from the fact that $-h$ has the same law as $h$ and that we can write
\begin{align*}
    \int_{D} m^2(z) :(h-c)^2(z): dz &= \lim_{\eps \to 0} \int_{D} m^2(z) ((h_{\eps}(z)-c)^2 - \EE[h_{\eps}(z)^2]) dz \\
    &= \lim_{\eps \to 0} \int_{D} m^2(z) ((-h_{\eps}(z)+c)^2 - \EE[h_{\eps}(z)^2]) dz
    = \int_{D} m^2(z) :(-h+c)^2(z): dz
\end{align*}
where the limit is in $L^2(\PP)$. Therefore, we have that, almost surely,
\begin{equation*}
     \tilde \EE[e^{-ia\xi_j} \vert \Gamma] = \frac{1}{\mathcal{Z}_j(\Gamma)} \EE \bigg[ e^{ia\xi_j} \exp \bigg(-\frac{1}{2} \int_{D} m^2(z) :(-h_j+\xi_j)^2(z): dz \bigg) \vert \Gamma \bigg].
\end{equation*}
Moreover, under $\tilde \PP$, conditionally on $\Gamma$, $h_j$ has the same law as $-h_j$. This implies that, almost surely,
\begin{equation} \label{symmetry_xij_final}
    \tilde \EE[e^{-ia\xi_j} \vert \Gamma] = \frac{1}{\mathcal{Z}_j(\Gamma)} \EE \bigg[ e^{ia\xi_j} \exp \bigg(-\frac{1}{2} \int_{D} m^2(z) :(h_j+\xi_j)^2(z): dz \bigg) \vert \Gamma \bigg].
\end{equation}
On the other hand, the same arguments as those used to prove \eqref{charac_xij_simplified} show that, almost surely,
\begin{equation} \label{charac_xij_+a}
    \tilde \EE[e^{ia\xi_j} \vert \Gamma] = \frac{1}{\mathcal{Z}_j(\Gamma)} \EE \bigg[ e^{ia\xi_j} \exp \bigg(-\frac{1}{2} \int_{D} m^2(z) :(h_j+\xi_j)^2(z): dz \bigg) \vert \Gamma \bigg].
\end{equation}
We can deduce from this equality and \eqref{symmetry_xij_final} that \eqref{symmetry_tilde_PP} holds, thus identifying the conditional law of $\xi_j$ given $\Gamma$ under $\tilde \PP$. 

Let us now show that under $\tilde \PP$, the random variables $(\xi_j)_j$ are conditionally independent given $\Gamma$. To this end, conditionally on $\Gamma$, let $(a_j)_j$ be a collection of real numbers, one for each loop. We then have that, almost surely,
\begin{align*}
    \tilde \EE \big[ \prod_j \exp(ia_j\xi_j) \vert \Gamma \big] = \frac{1}{\mathcal{Z}(\Gamma)} \EE \bigg[ \big( \prod_j \exp(ia_j\xi_j) \big) \exp \bigg( -\frac{1}{2} \int_{D} m^2(z):h^2(z):dz \bigg) \vert \Gamma \bigg]
\end{align*}
where, as before, $\mathcal{Z}(\Gamma)$ is defined as in \eqref{def_ZGamma}. Arguing as in the proof of Proposition \ref{prop_rep_CLE} and using the conditional independence of $(\xi_j)_j$ given $\Gamma$ under $\PP$, we obtain that, almost surely,
\begin{align*}
    \tilde \EE \big[ \prod_j \exp(ia_j\xi_j) \vert \Gamma \big]
    &= \lim_{\eps \to 0} \frac{1}{\mathcal{Z}(\Gamma)} \EE \bigg[ \big( \prod_j \exp(ia_j\xi_j) \big) \exp \bigg( -\frac{1}{2} \int_{D} m^2(z)\big(h_{\eps}(z)^2 - \EE[h_{\eps}(z)^2]\big) dz \bigg) \vert \Gamma \bigg] \\
    &= \begin{aligned}[t]&\frac{1}{\mathcal{Z}(\Gamma)} \prod_j \EE \bigg[ e^{ia_j\xi_j} \exp \bigg( -\frac{1}{2} \int_{\Lj} m^2(z) :(h+\xi_j)^2(z): dz \bigg) \vert \Gamma \bigg] \\
    &\times \exp \bigg( \int_{\Lj} \frac{m^2(z)}{4\pi} \log \frac{\CR(z,\partial D)}{\CR(z,L_j)} dz \bigg)\end{aligned}
\end{align*}
where in the second equality, the limit is in $\operatorname{L}^1(\PP)$. The product decomposition \eqref{prod_Z_Gamma_bis} of $\mathcal{Z}(\Gamma)$ then yields that, almost surely,
\begin{align*}
    \tilde \EE \big[ \prod_j \exp(ia_j\xi_j) \vert \Gamma \big] 
    = \prod_j \frac{1}{\mathcal{Z}_j(\Gamma)} \EE \bigg[ e^{ia_j\xi_j} \exp \bigg( -\frac{1}{2} \int_{\Lj} m^2(z) :(h+\xi_j)^2(z): dz \bigg) \vert \Gamma \bigg].
\end{align*}
Together with \eqref{charac_xij_+a}, this implies that, almost surely, $\tilde \EE \big[ \prod_j \exp(ia_j\xi_j) \vert \Gamma \big]  = \prod_j \tilde \EE \big[ \exp(ia_j\xi_j) \vert \Gamma \big]$, which shows that under $\tilde \PP$, the random variables $(\xi_j)$ are conditionally independent given $\Gamma$.
\end{proof}

\subsubsection{The conditional law of $h$ under $\tilde \PP$}

Having identified the conditional law given $\Gamma$ of the random variables $(\xi_j)_j$ under $\tilde \PP$, let us now show that under $\tilde \PP$, conditionally on $\Gamma$, the field $h$ can be decomposed as claimed in Theorem \ref{theorem_mCLE_4}. We will then establish Theorem \ref{theorem_mCLE_4}.

\begin{proposition} \label{prop_h_tildeP}
With the same assumptions as in Theorem \ref{theorem_mCLE_4}, under $\tilde \PP$, conditionally on $\Gamma$,
\begin{equation*}
    h = \sum_j h_j+\xi_j
\end{equation*}
where the sum runs over the loops $(L_j)_j$ of $\Gamma$ and the fields $(h_j+\xi_j)_j$ are independent fields whose law is, for each $j$, that of a massive GFF with mass $m$ and boundary conditions $\xi_j$ in the interior of the loop $L_j$.
\end{proposition}

To prove Proposition \ref{prop_h_tildeP}, as discussed in the introduction, we use a rigorous conditional version of the path-integral formalism for the massive GFF. This strategy involves a lot of technicalities but most of them have been taken care of in the proof of Proposition \ref{prop_rep_CLE}.

\begin{proof}[Proof of Proposition \ref{prop_h_tildeP}]
To establish the conditional law of $h$ given $\Gamma$ under $\tilde \PP$, we are going to compute its characteristic function. To this end, let $f$ be a smooth function with compact support in $D$. We have that, almost surely,
\begin{align} \label{EE_Gamma_intro}
    \tilde \EE \big[ \exp(i(h,f)) \vert \Gamma \big]
    = \frac{1}{\mathcal{Z}(\Gamma)} \EE \bigg[ \exp(i(h,f)) \exp \bigg( -\frac{1}{2} \int_{D} m^2(z):h^2(z): dz \bigg) \vert \Gamma \bigg],
\end{align}
where, as before, we have set
\begin{equation} \label{def_ZGamma}
    \mathcal{Z}(\Gamma) := \EE \bigg[ \exp \bigg( -\frac{1}{2} \int_{D} m^2(z):h^2(z): dz \bigg) \vert \Gamma \bigg].
\end{equation}
By Proposition \ref{prop_rep_CLE}, almost surely,
\begin{align} \label{prod_EE_h}
    &\EE \bigg[ \exp(i(h,f)) \exp \bigg( -\frac{1}{2} \int_{D} m^2(z):h^2(z): dz \bigg) \vert \Gamma \bigg] \nonumber \\
    &= \begin{aligned}[t]\prod_{j} &\EE \bigg[ \exp(i(h_j+\xi_j, f) \exp \bigg( -\frac{1}{2} \int_{\Lj} m^2(z) :(h_j+\xi_j)^2(z): dz \bigg) \vert \Gamma \bigg]\\
    & \times \exp\bigg(\int_{\Lj} \frac{m^2(z)}{4\pi} \log \frac{\operatorname{CR}(z, \partial D)}{\operatorname{CR}(z, L_j)} dz \bigg).\end{aligned}
\end{align}
On the other hand, Proposition \ref{prop_rep_CLE} applied with $f \equiv 0$ shows that, almost surely,
\begin{align} \label{prod_Z_Gamma}
    \mathcal{Z}(\Gamma) = \prod_{j} \mathcal{Z}_j(\Gamma)\exp\bigg(\int_{\Lj} \frac{m^2(z)}{4\pi} \log \frac{\operatorname{CR}(z, \partial D)}{\operatorname{CR}(z, L_j)} dz \bigg)
\end{align}
where we have set
\begin{equation} \label{def_Zj_Gamma}
    \mathcal{Z}_j(\Gamma) = \EE \bigg[ \exp \bigg( -\frac{1}{2} \int_{\Lj} m^2(z) :(h_j+\xi_j)^2(z): dz \bigg) \vert \Gamma \bigg].
\end{equation}
In view of \eqref{EE_Gamma_intro}, we deduce from this that, almost surely,
\begin{align*}
    \tilde \EE \big[ \exp(i(h,f)) \vert \Gamma \big]
    =\prod_j \frac{1}{\mathcal{Z}_j(\Gamma)} \EE \bigg[ \exp(i(h_j+\xi_j, f) \exp \bigg( -\frac{1}{2} \int_{\Lj} m^2(z) :(h_j+\xi_j)^2(z): dz \bigg) \vert \Gamma \bigg].
\end{align*}
Indeed, the terms involving the conformal radii in the product decomposition \eqref{prod_Z_Gamma} of $\mathcal{Z}(\Gamma)$ cancel with those appearing in \eqref{prod_EE_h}. Now, observe that by \eqref{RN_mGFF_phi}, conditionally on $\Gamma$, for any $j$,
\begin{equation*}
    \frac{1}{\mathcal{Z}_j(\Gamma)} \exp\bigg( -\frac{1}{2}\int_{\Lj} m^2(z) :(h_j+\xi_j)^2(z): dz \bigg)
\end{equation*}
is almost surely equal to the Radon-Nikodym derivative of the massive GFF with mass $m$ and boundary conditions $\xi_j$ in $\Lj$ with respect to the GFF with boundary conditions $\xi_j$ in $\Lj$. Going back to \eqref{EE_Gamma_intro}, this implies that under $\tilde \PP$, conditionally on $\Gamma$, $h=\sum_j h_j+\xi_j$ where $h_{j}+\xi_j$ is a massive GFF in $\Lj$ with boundary conditions $\xi_j$. 

Let us now establish the conditional independence of the fields $(h_j+\xi_j)_j$ given $\Gamma$ under $\tilde \PP$. To this end, under $\tilde \PP$, conditionally on $\Gamma$, let $(f_j)_j$ be a collection of smooth functions such that for each $j$, $f_j$ has compact support in $\Lj$. We then have that, almost surely,
\begin{align*}
    \tilde \EE \big[ \prod_j \exp(i(h_j+\xi_j, f_j)) \vert \Gamma \big]
    = \frac{1}{\mathcal{Z}(\Gamma)} \EE \bigg[ \big( \prod_j \exp(i(h_j+\xi_j, f_j)) \big) \exp \bigg(-\frac{1}{2}\int_{D} m^2(z) :h^2(z): dz \bigg) \vert \Gamma \bigg]
\end{align*}
where $\mathcal{Z}(\Gamma)$ is as in \eqref{def_ZGamma}. Since, almost surely, $\vert \prod_j \exp(i(h_j+\xi_j, f_j)) \vert \leq 1$, for the same reasons as those explained in the first part of the proof, we have that, almost surely,
\begin{align*}
    \tilde \EE \big[ \prod_j \exp(i(h_j+\xi_j, f_j)) \vert \Gamma \big]
    = \prod_j \frac{1}{\mathcal{Z}_j(\Gamma)} \EE \bigg[ \exp(i(h_j+\xi_j, f_j)) \exp \bigg( -\frac{1}{2} \int_{\Lj} m^2(z) :(h_j+\xi_j)^2(z): dz \bigg)  \vert \Gamma \bigg].
\end{align*}
On the other hand, using once again the same reasoning as in the first part of the proof, it is easy to see that almost surely, for any $j$,
\begin{align} \label{prod_sum_EE}
    \tilde \EE \big[ \exp(i(h_j+\xi_j, f_j)) \vert \Gamma \big]
    = \frac{1}{\mathcal{Z}_j(\Gamma)} \EE \bigg[ \exp(i(h_j+\xi_j, f_j)) \exp \bigg( -\frac{1}{2} \int_{\Lj} m^2(z) :(h_j+\xi_j)^2(z): dz \bigg) \vert \Gamma \bigg].
\end{align}
Using this equality, we can then deduce from \eqref{prod_sum_EE} that under $\tilde \PP$, conditionally on $\Gamma$, the fields $(h_j+\xi_j)_j$ are independent. This completes the proof of Proposition \ref{prop_h_tildeP}.
\end{proof}

We can now conclude the proof of Theorem \ref{theorem_mCLE_4}.

\begin{proof}[Proof of Theorem \ref{theorem_mCLE_4}]
$\tilde \PP$ defined via the Radon-Nikodym derivative \eqref{RN_mCLE4} is a well-defined probability measure which is absolutely continuous with respect to $\PP$. Moreover, by definition, the marginal law under $\tilde \PP$ of $h$ is that of a massive GFF in $D$ with mass $m$ and Dirichlet boundary conditions. Proposition \ref{prop_h_tildeP} and Proposition \ref{prop_xi_tildeP} show the first three bullet points of the statement of Theorem \ref{theorem_mCLE_4} regarding the conditional law of $h$ given $\Gamma$ under $\tilde \PP$. To conclude the proof of Theorem \ref{theorem_mCLE_4}, it remains to prove the last bullet point. This simply follows from the fact that under $\PP$, conditionally on $\Gamma$, for each $j$, $\xi_j$ is almost surely measurable with respect to $h_j+\xi_j$. Since $\tilde \PP$ is absolutely continuous with respect to $\PP$, the same holds true under $\tilde \PP$. This gives the joint law of $(h_j+\xi_j, \xi_j)$ under $\tilde \PP$ conditionally on $\Gamma$. This also concludes the proof of Theorem \ref{theorem_mCLE_4}.
\end{proof}

 \subsection{Some properties of massive CLE$_4$} \label{sec_cor_CLE4}

In this section, we show two properties of massive CLE$_4$ that can be established thanks to Theorem \ref{theorem_mCLE_4}. The first one, stated as Corollary \ref{cor_RN_mCLE4} in the introduction, is an expression for the Radon-Nikodym derivative of the law massive CLE$_4$ with respect to that of CLE$_4$.

\begin{proof}[Proof of Corollary \ref{cor_RN_mCLE4}]
The proof is very similar to that of Corollary \ref{cor_RN_mSLE4}, using that under $\PP$, conditionally on $\Gamma$, $(h_j)_j$ and $(\xi_j)$ are independent. We leave the details to the reader.
\end{proof}

The second property of massive CLE$_4$ that follows from Theorem \ref{theorem_mCLE_4} is its conformal covariance, as stated in Corollary \ref{cor_conformal_cov_mCLE4} and shown below.

\begin{proof}[Proof of Corollary \ref{cor_conformal_cov_mCLE4}]
The proof of this corollary follows the same strategy as that of Lemma \ref{lemma_conformal_cov_mSLE} and we leave the details to the reader.
\end{proof}

\section{Massive CLE$_4$, the massive Brownian loop soup and the massive GFF}

In this section, we prove Theorem \ref{prop_mCLE4_mBLS}. In Section \ref{sec_mBLS}, we first define the massive Brownian loop soup and show that its law is absolutely continuous with respect to that of the Brownian loop soup. Then, in Section \ref{sec_mCLE4_mBLS}, we use this result together with Theorem \ref{theorem_mCLE_4} to deduce Theorem \ref{prop_mCLE4_mBLS}.

\subsection{The massive Brownian loop soup} \label{sec_mBLS}

The massive Brownian loop soup is a massive version of the Brownian loop soup introduced by Camia in \cite{Camia_BLS, Hausdorff_mBLS}. To define it, one first defines a massive version $\mu_{D}^{(m)}$ of the Brownian loop measure, called the massive Brownian loop measure, via the following change of measure:
\begin{equation*}
	\mu_{D}^{(m)}(d\ell):= \exp \bigg( -\int_{0}^{\tau(\ell)} m^2(\ell(t)) dt \bigg) \mu_{D}(d\ell).
\end{equation*}
Here, we assume that the domain $D$ and the mass function $m$ satify the assumptions of Theorem \ref{prop_mCLE4_mBLS} and we recall that $\mu_{D}$ was defined in Section \ref{sec_BLS}. The measure $\mu_{D}^{(m)}$ is then a measure on unrooted loops in $D$ and it enjoys the same restriction property as $\mu_{D}$. However, the presence of a mass breaks the conformal invariance of the measure: $\mu_{D}^{(m)}$ is conformally covariant, in the following sense. If $f: D \to \Omega$ is a conformal map, then $\mu_{\Omega}^{(\tilde m)} =  f \circ \mu_{D}^{(m)}$ where, for $z \in \Omega$, $\tilde m^2(z) = \vert (f^{-1})'(z)\vert^2m^2(f^{-1}(z))$.

The massive Brownian loop soup in $D$ with mass $m$ and intensity $\alpha > 0$ is then defined analogously to the Brownian loop soup in $D$ with intensity $\alpha$ (see Section \ref{sec_BLS}): this is the Poisson point process with intensity measure $\alpha \mu_{D}^{(m)}$. The massive Brownian loop soup is actually absolutely continuous with respect to the Brownian loop soup and the corresponding Radon-Nikodym can be explicitly computed, as shown below.

\begin{lemma} \label{lemma_abscont_mBLS}
Let $D \subset \mathbb{C}$ and $m:D \to \mathbb{R}_{+}$ be as in Theorem \ref{prop_mCLE4_mBLS}. Let $\alpha > 0$ and denote by $\PP_{\alpha}$, respectively $\PP_{\alpha}^{(m)}$, the law of the Brownian loop soup, respectively massive Brownian loop soup with mass $m$, in $D$ with intensity $\alpha$. Then
\begin{equation} \label{RN_mBLS_BLS}
\frac{\der \PP_{\alpha}^{(m)}}{\der \PP_{\alpha}}(\LL) = \frac{1}{\mathcal{Z}}\exp \bigg( -\int_{D} m^2(z) :L(z): dz \bigg)
\end{equation} 
where $\mathcal{Z}$ is a normalization constant and, under $\PP_{\alpha}$, $:L:$ is the renormalized occupation-time field of $\LL$ (see Section \ref{sec_BLS}). 
\end{lemma}

\begin{proof}
The proof goes along the same lines as that of the "if" part of the main theorem of \cite{Poisson_abs}. Let $\alpha >0$. As a preliminary remark, observe that the change of measure \eqref{RN_mBLS_BLS} is well-defined. Indeed, $\int_{D} m^2(z) :L(z): dz$ is a limit in $L^2(\PP_{\alpha})$ and is therefore measurable with respect to $\LL$ while  finiteness of the normalization constant $\mathcal{Z}$ is shown in \cite[Theorem~8]{LeJan}.

Let us denote by $\mathcal{X}(D)$ the space of unrooted loops contained in $D$. We turn $\mathcal{X}(D)$ into a metric space by endowing with the metric $d_{\mathcal{X}}$ induced by the norm, for $\ell \in \mathcal{X}(D)$,
\begin{equation*}
    \| \ell \|_{\mathcal{X}} := \tau(\ell) + \sup_{t \in [0,\tau(\ell)]} \vert \ell(t) \vert
\end{equation*}
where $\tau(\ell)$ is the lifetime of the loop $\ell$. See \cite[Section~5.1]{book_Lawler} for more details about these definitions. Let $f: \mathcal{X}(D) \to \mathbb{R}_{+}$ be a nonnegative function with compact support in $\mathcal{X}(D)$. To prove Lemma \ref{lemma_abscont_mBLS}, we must show that
\begin{equation} \label{Laplace_mBLS}
    \EE_{\alpha}^{(m)}[\exp(-\langle \LL, f \rangle)] = \exp \bigg( -\alpha \int  (1 - e^{-f(\ell)}) \mu_{D}^m(d\ell) \bigg)
\end{equation}
or, in words, that the Laplace transform of $\LL$ under $\PP_{\alpha}^{(m)}$ is the same as that of a massive Brownian loop in $D$ with mass $m$ and intensity $\alpha$. Above, the random variable $\langle \LL, f \rangle$ is defined as $\langle \LL, f \rangle := \sum_{\ell \in \LL} f(\ell)$. For ease of notations, let us set
\begin{equation*}
    \phi(\ell) := \exp \bigg( -\int_{0}^{\tau(\ell)} m^2(\ell(t)) dt \bigg)
\end{equation*}
and note that $\mu_{D}^{m}(d\ell) = \phi(\ell)\mu_{D}(d\ell)$. Next, we observe that in the integral on the right-hand side of \eqref{Laplace_mBLS}, the integrand is $0$ outside the support of $f$. We are going to use this to decompose this integral as a sum over well-chosen compact sets of loops. To this end, let us define the sets of loops, for $n \geq 1$,
\begin{equation*}
	K_n := \{\ell \in \mathcal{X}(D): \overline{\ell} \subset \overline{D}, \tau(\ell) \in [2^{-n}, 2^{n}] \}.
\end{equation*}
Observe that the set $(K_n)_n$ are such that $K_{n} \subset \operatorname{Int}(K_{n+1})$ and $\mathcal{X}(D) = \cup_n K_n$. Moreover, we claim the following.

\begin{claim} \label{claim_compactness_Kn}
The sets $(K_n)_n$ are compact sets of $(\mathcal{X}(D), d_{\mathcal{X}})$.
\end{claim}

We postpone the proof of Claim \ref{claim_compactness_Kn} to the end and continue the proof of the equality \eqref{Laplace_mBLS}. By \cite[Exercise~3.8.C]{Engelking}, Claim \ref{claim_compactness_Kn} and the covering property of the sets $(K_n)$ imply that $(\mathcal{X}(D), d_{\mathcal{X}})$ is hemicompact: for any compact set $K \subset \mathcal{X}(D)$, there exists $n \in \mathbb{N}$ such that $K \subset K_n$, see for example \cite[Section~17.I]{Willard}. Therefore, since $f$ has compact support, for $n$ large enough,
\begin{equation*}
    \int (1 - e^{-f(\ell)}) \mu_{D}^m(d\ell) = \int_{K_n} (1 - e^{-f(\ell)}) \phi(\ell)\mu_{D}(d\ell).
\end{equation*}
Note that this integral is finite since loops in $K_n$ have lifetime at least $2^{-n}$ (the divergence of the total mass of $\mu_D$ is only due to short loops).  For $n \geq 1$, let us further set $B_n = \{ \ell \in \mathcal{X}(D): \overline{\ell} \subset \overline{D}, \tau(\ell) \geq 2^{-n}\}$ and observe that $K_n \subset B_n$. Using the fact that $f$ is compactly supported in $K_n \subset B_n$,
\begin{align}
	&\exp \bigg( -\alpha \int_{K_n} (1 - e^{-f(\ell)}) \phi(\ell)\mu_{D}(d\ell) \bigg) \nonumber \\
	&= \exp \bigg( - \alpha \int_{B_n} (1 - e^{-f(\ell)}) \phi(\ell)\mu_{D}(d\ell) \bigg) \nonumber \\
	&= \EE_{\alpha}\bigg[ \exp(-\langle \mathcal{L}_{B_n}, f\rangle + Y_{n}(\LL))\bigg] \label{Laplace_final}
\end{align}
where for a set $B$ and a function $g$, $\langle \mathcal{L}_{B}, g \rangle := \sum_{\ell \in \mathcal{L}: \ell \in B} g(\ell)$ and where we have set, for $n \geq 1$,
\begin{equation*}
	Y_n(\LL) := \langle \LL_{B_n}, \log \phi \rangle + \alpha \int_{B_n} (1-\phi(\ell)) \mu_{D}(d\ell).
\end{equation*}
Note that in this definition, the integral $\int_{B_n} (1-\phi(\ell)) \mu_{D}(d\ell)$ is finite since it is bounded above by $2\mu_{D}(B_{n})$, which is finite because loops in $B_n$ have lifetime at least $2^{-n}$. The equality \eqref{Laplace_final} implies that to establish \eqref{Laplace_mBLS}, it suffices to show that the random variable $\exp(-\langle \mathcal{L}_{B_n}, f\rangle + Y_{n}(\LL))$ on the right-hand side of \eqref{Laplace_final} converges in $L^1(\PP_{\alpha})$ as $n \to \infty$ to $\exp(-\langle \mathcal{L}, f\rangle)R(\LL)$ where $R(\LL)$ denotes the random variable on the right-hand side of \eqref{RN_mBLS_BLS} (with the normalization constant $\mathcal{Z}$ included). 

To this end, observe that $\exp(-\langle \mathcal{L}_{B_n}, f \rangle)$ almost surely converges to $\exp(-\langle \mathcal{L}, f \rangle)$ and moreover, since $f$ is non-negative, almost surely, for any $n$, $\exp(-\langle \mathcal{L}_{B_n}, f \rangle) \leq 1$. By dominated convergence, this implies that $\exp(-\langle \mathcal{L}_{B_n}, f \rangle)$ converges in $L^1(\PP_{\alpha})$ to $\exp(-\langle \mathcal{L}, f \rangle)$. Therefore, using once again that, almost surely, for any $n$, $\exp(-\langle \mathcal{L}_{\tilde B_n}, f \rangle) \leq 1$, the following claim then yields convergence of $\exp(-\langle \mathcal{L}_{B_n}, f\rangle + Y_{n}(\LL))$ to $\exp(-\langle \mathcal{L}, f\rangle)R(\LL)$ in $L^1(\PP_{\alpha})$, and thus concludes the proof of Lemma \ref{lemma_abscont_mBLS}.

\begin{claim} \label{claim_cvg_Yn_COM}
The random variables $(\exp(Y_n(\LL)))_n$ converge in $L^1(\PP_{\alpha})$ to the random variable on the right-hand side of \eqref{RN_mBLS_BLS}.
\end{claim}
\end{proof}

We now turn to the proof of Claim \ref{claim_compactness_Kn} and Claim \ref{claim_cvg_Yn_COM}.

\begin{proof}[Proof of Claim \ref{claim_compactness_Kn}]
Let $n \geq 1$ and let $(\ell_k)_k$ be a sequence of loops in $K_n$. To show that $K_n$ is compact in $(\mathcal{X}(D), d_{\mathcal{X}})$, we are going to show that there exists a subsequence $(k(p))_{p}$ such that $(\ell_{k(p)})_p$ converges in $(\mathcal{X}(D), d_{\mathcal{X}})$ to a limit $\ell^{*}$ with $\ell^{*} \in K_n$. To this end, observe that, for any $k$, since $\ell_k \in K_n$ and $\ell_k$ is contained in the bounded domain $D$,
\begin{equation*}
    2^{-n} \leq \| \ell_k \|_{\mathcal{X}} = \tau(\ell_k) + \sup_{t} \vert \ell(t) \vert \leq 2^{n} + C_D
\end{equation*}
where $C_D > 0$ is a finite constant that depends only on $D$. Therefore, $(\ell_k)_k$ is bounded in $(\mathcal{X}(D), d_{\mathcal{X}})$, which implies that there exists a subsequence $(k(p))_p$ such $(\ell_{k(p)})_p$ converges in $(\mathcal{X}(D), d_{\mathcal{X}})$ to a limit $\ell^{*}$. To see that $\ell^{*}$ belongs to $K_n$, observe first that the function $\tau: \ell \mapsto \tau(\ell)$ is continuous in $(\mathcal{X}(D), d_{\mathcal{X}})$ and therefore, the fact that for any $p$, $2^{-n} \leq \tau(\ell_{k(p)}) \leq 2^n$ implies that $2^{-n} \leq \tau(\ell^{*}) \leq 2^n$. Moreover, we have that ${\ell^{*}} \subset \overline{D}$ since for any $p$, $\ell_{k(p)}$ belongs to $\mathcal{X}(D)$ and hence $\overline \ell^{*} \subset \overline{D}$.
\end{proof}

\begin{proof}[Proof of Claim \ref{claim_cvg_Yn_COM}]
Let us set, for $n \geq 1$,
\begin{equation*}
    \tilde Y_n(\LL) := \langle \LL_{B_n}, \log \phi \rangle - \alpha \mu_{D}(\mathbb{I}_{B_n}\log \phi).
\end{equation*}
Note that, in this definition, the loop measure term is finite. Indeed, we have that
\begin{equation} \label{finite_1stmoment_Bn}
	\mu_{D}(\vert \log \phi \vert \mathbb{I}_{B_n}) = \mu_{D} \bigg(\mathbb{I}_{\ell \in B_n} \int_{0}^{\tau(\ell)} m^2(\ell(t))dt \bigg) \leq \int_{2^{-n}}^{\infty} \int_{D} m^2(z)p_t(z,z) dz dt < \infty
\end{equation}
as $D$ is bounded (here, $p_t$ is the heat kernel in $D$ with Dirichlet boundary conditions on $\partial D$). Now, observe that, almost surely,
\begin{equation*}
    \exp(Y_n(\LL)) = \exp(\tilde Y_n(\LL))\exp\bigg(\alpha \int_{B_n} (1-\phi(\ell)+\log \phi(\ell))\mu_{D}(d\ell)\bigg)
\end{equation*}
where finiteness of the integral $\int_{B_n} (1-\phi(\ell)+\log \phi(\ell))\mu_{D}(d\ell)$ follows from \eqref{finite_1stmoment_Bn} and the fact that, for any $\ell$, $\vert 1-\phi(\ell)+\log \phi(\ell)\vert \leq 2\vert \log \phi(\ell)\vert$. Therefore, to prove the claim, it suffices to show that
\begin{align}
    &\lim_{n \to \infty} \exp(\tilde Y_n(\LL))\exp\bigg(\alpha \int_{\tilde B_n} (1-\phi(\ell)+\log \phi(\ell))\mu_{D}(d\ell)\bigg) \nonumber \\
    &= \frac{1}{\mathcal{Z}}\exp \bigg(-\int_{D} m^2(z) :L(z): dz \bigg) \label{lim_good_L}
\end{align}
where the limit is in $L^1(\PP_{\alpha})$. To establish this, we first show that $(\exp(Y_n(\LL)))_n$ is a bounded martingale in $L^1(\PP_{\alpha})$, which implies that it has a limit in $L^1(\PP_{\alpha})$. To identify this limit with the random variable on the right-hand side of \eqref{lim_good_L}, we then prove that $(\exp(Y_n(\LL)))_n$ almost surely converges to this random variable.

The fact that $(\exp(Y_n(\LL)))_n$ is a martingale follows from the Poissonian nature of $\LL$, as we now explain. For $n\geq 2$, let us set $B_1' = B_1$ and $B_n' = B_n \setminus B_{n-1}$. We then have that, for any $n \geq 2$, almost surely,
\begin{align*}
    &\EE_{\alpha}[\exp(Y_n(\LL)) \vert \exp(Y_{n-1}(\LL))] \\
	&= \exp(Y_{n-1}(\LL)) \EE_{\alpha}[\exp(Y_n(\LL) - Y_{n-1}(\LL)) \vert \exp(Y_{n-1}(\LL))]\\
    &= \exp(Y_{n-1}(\LL)) \exp\bigg(\alpha \int_{B_n'} (1- \phi(\ell)) \mu_{D}(d\ell) \bigg) \EE_{\alpha}[\exp(\langle \LL_{ B_n'}, \log \phi \rangle ) \vert \exp(Y_{n-1}(\LL))]\\
    &= \exp(Y_{n-1}(\LL)) \exp\bigg(\alpha \int_{B_n'} (1- \phi(\ell)) \mu_{D}(d\ell) \bigg) \EE_{\alpha}[\exp(\langle \LL_{B_n'}, \log \phi \rangle )]\\
    &=\exp(Y_{n-1}(\LL))
\end{align*}
where in the penultimate equality, we used the fact that $\exp(\langle \LL_{B_n'}, \log \phi \rangle)= \exp(\langle \LL_{B_n}, \log \phi \rangle - \langle \LL_{B_{n-1}}, \log \phi \rangle)$ is independent of $\exp(Y_{n-1}(\LL))$ since $\LL$ is a Poisson point process. This shows that $(\exp(Y_n(\LL)))_n$ is a martingale. Boundedness in $L^1(\PP_{\alpha})$ is straightforward to see since for each $n$, $\exp(Y_n(\LL))$ is almost surely a non-negative random variable and
\begin{equation*}
    \EE_{\alpha}[\exp(Y_n(\LL))] = \exp \bigg( \alpha \int_{B_n} (1-\phi(\ell)) \mu_{D}(d\ell) \bigg)\EE_{\alpha}[\exp(\langle \LL_{B_n}, \log \phi \rangle)] = 1.
\end{equation*}
Having established that $(\exp(Y_n(\LL))_n$ converges in $L^1(\PP_{\alpha})$, let us now show that it also converges almost surely to the random variable on the right-hand side of \eqref{lim_good_L}. To this end, we first apply the dominated convergence theorem to the loop measure term. Observe that, $\mu_{D}$-almost everywhere,
\begin{align*}
    \lim_{n \to \infty} \mathbb{I}_{B_n} (1-\phi(\ell)+\log \phi(\ell)) = 1-\phi(\ell)+\log \phi(\ell).
\end{align*}
Moreover, there exists a constant $C>0$ such that, for any loop $\ell$ with $\phi(\ell) \geq 1/2$, $\vert 1-\phi(\ell)+\log \phi(\ell) \vert \leq C(\sqrt{\phi(\ell)} - 1)^2$. Therefore, setting $E:= \{ \ell \in \mathcal{X}(D): \phi(\ell) \leq 1/2\}$, we have that, for any $n \geq 1$, $\mu_{D}$-almost everywhere,
\begin{equation*}
    \mathbb{I}_{B_n \setminus E} \vert 1-\phi(\ell)+\log \phi(\ell) \vert \leq \mathbb{I}_{B_n \setminus E} C(\sqrt{\phi(\ell)} - 1)^2.
\end{equation*}
Moreover, it is easy to see that for a loop $\ell$, $(\sqrt{\phi(\ell)} - 1)^2 \leq (\log \phi(\ell))^2/4$ and $\mu_{D}((\log \phi(\ell))^2)= \mu_{D}(\langle m^2, \ell \rangle^2) < \infty$ by \cite[Lemma~2]{LeJan} since $D$ is a bounded domain. This shows that, for any $n \geq 1$, $\mu_{D}$-almost everywhere,
\begin{align*}
    \mathbb{I}_{B_n} \vert 1-\phi(\ell)+\log \phi(\ell) \vert &\leq \mathbb{I}_{B_n \setminus E} C(\sqrt{\phi(\ell)} - 1)^2 + \mathbb{I}_{E} \vert 1-\phi(\ell)+\log \phi(\ell) \vert\\
    &\leq C(\log \phi(\ell))^2 + \mathbb{I}_{E} \vert 1-\phi(\ell)+\log \phi(\ell) \vert
\end{align*}
and the right-hand side is $\mu_D$-integrable (loops in $E$ have lifetime at least $\overline{m}^{-2}\log(2)$). Therefore, by dominated convergence, we obtain that
\begin{equation*}
    \lim_{n \to \infty} \exp\bigg(\alpha \int_{B_n} (1-\phi(\ell)+\log \phi(\ell))\mu_{D}(d\ell))\bigg) = \exp(\alpha \mu_{D}(1-\phi(\ell)+\log \phi(\ell))) = \frac{1}{\mathcal{Z}}
\end{equation*}
where the rightmost equality follows from \cite[Theorem~8]{LeJan}. On the other hand, as recalled in Section \ref{sec_BLS}, we have that $\tilde Y_n(\LL)$ converges $\PP_{\alpha}$-almost surely to $-\int_{D} m^2(z) :L(z):dz$ as $n \to \infty$. Therefore, by continuity of $x \mapsto e^{x}$, we have that, $\PP_{\alpha}$-almost surely,
\begin{equation*}
    \lim_{n \to \infty} \exp(\tilde Y_n(\LL)) = \exp \bigg( -\int_{D} m^2(z) :L(z): dz \bigg).
\end{equation*}
As explained above, this yields \eqref{lim_good_L} and concludes the proof of Claim \ref{claim_cvg_Yn_COM}.
\end{proof}

\subsection{Construction of massive CLE$_4$ via a massive Brownian loop soup} \label{sec_mCLE4_mBLS}

We now turn to the proof of Theorem \ref{prop_mCLE4_mBLS}. This theorem is in fact a consequence of the following result, which we show below.

\begin{proposition} \label{prop_mGFF_mBLS}
Let $D \subset \mathbb{C}$ and $m: D \to \mathbb{R}_{+}$ be as in Theorem \ref{prop_mCLE4_mBLS}. Denote by $\PP$ the law of the coupling $(h,\LL,\Gamma)$ of Proposition \ref{prop_coupling_GFF_BLS_CLE} between a GFF $h$ in $D$ with Dirichlet boundary conditions, a Brownian loop soup $\LL$ in $D$ with intensity $1/2$ and a CLE$_4$ $\Gamma$ in $D$. Define a new probability measure $\tilde \PP$ via
\begin{equation} \label{COM_mGFF_mBLS}
    \frac{\der \tilde \PP}{\der \PP}((h,\mathcal{L},\Gamma)) := \frac{1}{\mathcal{Z}}\exp \bigg( -\int_{D} m^2(z) :L(z): dz \bigg)
\end{equation}
where, under $\PP$, $\int_{D} m^2(z) :L(z): dz$ is the renormalized occupation-time field of $\LL$ tested against the function $m^2$ and $\mathcal{Z}$ is a normalization constant. Then, under $\tilde \PP$,
\begin{itemize}
	\item the marginal law of $h$ is that of a massive GFF in $D$ with mass $m$ and Dirichlet boundary conditions;
	\item the marginal law of $\Gamma$ is that of a massive CLE$_4$ in $D$ with mass $m$;
	\item the loops of $\Gamma$ are the level lines of $h$ as in Theorem \ref{theorem_mCLE_4};
	\item the marginal law of $\LL$ is that of a massive Brownian loop soup in $D$ with mass $m$ and intensity $1/2$;
	\item the loops of $\Gamma$ are almost surely equal to the outer boundaries of the outermost clusters of $\LL$.
\end{itemize}
\end{proposition}

\begin{proof}
As noted in the proof of Lemma \ref{lemma_abscont_mBLS}, the change of measure \eqref{COM_mGFF_mBLS} is well-defined. The first, second and third items in the statement of Proposition \ref{prop_mGFF_mBLS} are consequences of the definition of $\tilde \PP$ and Theorem \ref{theorem_mCLE_4}. Indeed, under $\PP$, $\frac{1}{2}\int_{D} m^2(z) :h(z)^2: dz = \int_{D} m^2(z) :L(z): dz$ and therefore, by Theorem \ref{theorem_mCLE_4}, $(h,\Gamma)$ has the required joint law under $\tilde \PP$. The fourth item is in turn a direct consequnece of Lemma \ref{lemma_abscont_mBLS}. The final item follows by absolute continuity, since the equality holds almost surely under $\PP$.
\end{proof}

\section{Appendix} \label{sec_app_1}

In this section, we prove Lemma \ref{lemma_L2_cvg_Wick_boundary} and Lemma \ref{lemma_Lp_cvg_boundary}. We also show the almost sure strict positiveness of the random variables on the right-hand side of \eqref{RN_mGFF_Dirichlet} and \eqref{RN_mGFF_phi}.

In what comes next, we will repeatedly make use of the following estimate. Let $D \subset \mathbb{C}$ be a bounded and simply connected domain and let $h$ be a GFF in $D$ with Dirichlet boundary conditions. Then, there exists $C > 0$ such that for any $0<\eps<1/2$,
\begin{equation} \label{estimate_sup_variance}
    \sup_{z \in D} \EE[h_{\eps}(z)^2] \leq C\log(\eps^{-1}).
\end{equation}
This inequality follows from \cite[Lemma~3.5]{GMC_var}.

Let us now turn to the proof of Lemma \ref{lemma_L2_cvg_Wick_boundary}. To prove this result, we first show the following auxiliary lemma.

\begin{lemma} \label{lemma_EE_Wick_Markov}
Let $D \subset \mathbb{C}$ be a bounded and simply connected domain and let $h$ be a GFF in $D$ with Dirichlet boundary conditions. Let $0 < \delta < \eps <1/2$. Then, for any $z,w \in D$ such that $\max(\operatorname{dist}(z,\partial D), \operatorname{dist}(w,\partial D)) > \eps$ and $\vert z -w \vert > 2\eps$,
\begin{equation} \label{Wick_Markov}
    \EE[(:h_{\eps}(z)^2: - :h_{\delta}(z)^2:)(:h_{\eps}(w)^2: - :h_{\delta}(w)^2:)] = 0.
\end{equation}
Moreover, there exists $K>0$ such that for any $0 < \delta < \eps <1/2$ and any $z,w \in D$,
\begin{equation} \label{estimate_difference_wick}
    \EE[(:h_{\eps}(z)^2: - :h_{\delta}(z)^2:)(:h_{\eps}(w)^2: - :h_{\delta}(w)^2:)] \leq K (\log \delta)^2.
\end{equation}
\end{lemma}

\begin{proof}
Let us first prove \eqref{Wick_Markov}. Let $0 < \delta < \eps <1/2$ and let $z,w \in D$ be such that $\vert z -w \vert > 2\eps$ and assume without loss of generality that $\operatorname{dist}(z,\partial D) > \eps$. Applying the Markov property in the ball $B(z,2\eps)$ and using that
\begin{equation*}
    (W_r)_{r \geq 0} := (h_{2\eps e^{-r}}(z) - h_{2\eps}(z))_{r \geq 0}
\end{equation*}
is a Brownian motion, we see that for some centered random variables $X, \tilde X$ (not independent but independent of $W$),
\begin{equation*}
    (:h_{\eps}(z)^2: - :h_{\delta}(z)^2:)(:h_{\eps}(w)^2: - :h_{\delta}(w)^2:) = \big( (W_t -t^2) - (W_s - s^2) + 2X(W_t-W_s)\big) \tilde X.
\end{equation*}
Taking expectation and using independence, we get zero, which completes the proof of \eqref{Wick_Markov}.

Let us now prove the estimate \eqref{estimate_difference_wick}. Let $0 < \delta < \eps < 1/2$ and let $z,w \in D$. Observe first that
\begin{equation} \label{upper_square_wick_difference}
    \EE [(:h_{\eps}(z)^2: - :h_{\delta}(z)^2:)^2] \leq 2 \EE[(:h_{\eps}(z)^2:)^2] + 2\EE[(:h_{\delta}(z)^2:)^2].
\end{equation}
Moreover, since $h_{\eps}(z)$ is a Gaussian random variable, we have that $\EE[(:h_{\eps}(z)^2:)^2] = \EE[h_{\eps}(z)^4] - 2\EE[h_{\eps}(z)^2]^2 + \EE[h_{\eps}(z)^2]^2 = 2\EE[h_{\eps}(z)^2]^2$. Going back to \eqref{upper_square_wick_difference}, we deduce from this and the estimate \eqref{estimate_sup_variance} that
\begin{equation*}
    \EE [(:h_{\eps}(z)^2: - :h_{\delta}(z)^2:)^2] \leq 4C^2 \log(\delta)^2.
\end{equation*}
The estimate \eqref{estimate_difference_wick} then simply follows from this bound and the Cauchy-Schwarz inequality.
\end{proof}

With Lemma \ref{lemma_EE_Wick_Markov} at hands, let us now prove Lemma \ref{lemma_L2_cvg_Wick_boundary}. We first deal with the convergence in $L^2(\PP)$ of $\int_{D} m^2(z) :h_{\eps}(z)^2:dz$, that is with the case $\phi \equiv 0$.

\begin{lemma} \label{lemma_L2_cvg_Wick}
Let $D \subset \mathbb{C}$ be a bounded, open and simply connected domain and let $m: D \to \mathbb{R}_{+}$ be a bounded and continuous function. Let $h$ be a GFF with Dirichlet boundary conditions in $D$. Then, as $\eps \to 0$, the random variable
\begin{equation*}
    \int_{D} m^2(z) :h_{\eps}(z)^2:dz
\end{equation*}
converges in $\operatorname{L}^2(\PP)$ to a limit denoted by $\int_{D} m^2(z):h(z)^2: dz$. Moreover, for any $b \in (0, 2-\operatorname{dim}_{H}(\partial D))$, there exists $C>0$ such that for any $0 < \eps < 1/2$,
\begin{equation} \label{rate_cvg_L2_Wick}
    \bigg \| \int_{D} m^2(z) :h(z)^2: dz - \int_{D} m^2(z) :h_{\eps}(z)^2: dz \bigg\|_{\operatorname{L}^2(\PP)} \leq C\eps^{b}.
\end{equation}
\end{lemma}

\begin{proof}[Proof of Lemma \ref{lemma_L2_cvg_Wick}]
For conciseness, let us set $\alpha := 2-\operatorname{dim}_H(\partial D)$. We are going to show that the sequence
\begin{equation*}
    \bigg( \int_{D} m^2(z) :h_{\eps}(z)^2:dz, \eps > 0 \bigg)
\end{equation*}
is a Cauchy sequence in $L^{2}(\PP)$ and thus has a limit in $L^2(\PP)$ as $\eps$ tends to $0$. To this end, let $0< \delta < \eps < 1/2$. Then, by Fubini-Tonelli theorem, we have that
\begin{align} \label{difference_L2_k}
    &\EE \bigg[ \bigg \vert \int_{D} m^2(z) :h_{\eps}(z)^2: dz - \int_{D} m^2(z) :h_{\delta}(z)^2: dz \bigg \vert^2 \bigg] \nonumber \\
    &= \int_{D \times D} m^2(z)m^2(w) \EE[(:h_{\eps}(z)^2: - :h_{\delta}(z)^2:)(:h_{\eps}(w)^2: - :h_{\delta}(w)^2:)] dw dz.
\end{align}
Indeed, when we expand the product of the difference of the Wick squares, we obtain a sum of non-negative terms and we can apply Fubini-Tonelli theorem to each of these terms separately. With $D_{\eps} :=\{ z \in D: \operatorname{dist}(z, \partial D) \geq \eps \}$, we decompose the region of integration into $D \times D_{\eps}$, $D_{\eps} \times D$ and $(D \setminus D_{\eps}) \times (D_{\eps} \setminus D_{\eps})$. For the integral over $D \times D_{\eps}$, Lemma \ref{lemma_EE_Wick_Markov} shows that only the points $z,w \in D \times D_{\eps}$ such that $\vert z -w \vert \leq 2\eps$ contribute. Using the estimate \eqref{estimate_difference_wick} of Lemma \ref{lemma_EE_Wick_Markov}, this yields that
\begin{align} \label{bound_int_1}
    &\int_{D} m^2(z) \int_{B(z,2\eps) \cap D_{\eps}} m^2(w) \EE[(:h_{\eps}(z)^2: - :h_{\delta}(z)^2:)(:h_{\eps}(w)^2: - :h_{\delta}(w)^2:)] dw dz \nonumber \\
    &\leq \int_{D} m^2(z) \int_{B(z,2\eps) \cap D_{\eps}} m^2(w) K (\log\delta)^2 dw dz \nonumber \\
    &\leq O(\eps^2(\log \delta)^2)
\end{align}
where the implied constant depends only on $D$ and $m^2$. The integral over $D_{\eps} \times D$ can be bounded in a similar way. We are now left to control the integral over $(D \setminus D_{\eps}) \times (D \setminus D_{\eps})$. To do so, we can simply use the estimate \eqref{estimate_difference_wick} of Lemma \ref{lemma_EE_Wick_Markov} to obtain the bound
\begin{align*}
    \int_{(D \setminus D_{\eps}) \times (D \setminus D_{\eps})} m^2(z)m^2(w) \EE[(:h_{\eps}(z)^2: - :h_{\delta}(z)^2:)(:h_{\eps}(w)^2: - :h_{\delta}(w)^2:)] dw dz
    \leq O(\eps^{2\alpha} \log(\delta)^2).
\end{align*}
Going back to \eqref{difference_L2_k}, since $\alpha \in (0,1]$, this bound together with the bound \eqref{bound_int_1} imply that
\begin{align} \label{decay_L2}
    &\EE \bigg[ \bigg \vert \int_{D} m^2(z) :h_{\eps}(z)^2: dz - \int_{D} m^2(z) :h_{\delta}(z)^2: dz \bigg \vert^2 \bigg]
    \leq  O(\eps^{2\alpha}(\log \delta)^2).
\end{align}
This shows that $(\int_{D} m^2(z) :h_{\eps}(z)^2: dz, \eps > 0)$ is a Cauchy sequence in $L^2(\PP)$ and therefore has a limit in $L^2(\PP)$, that we denote by $\int_{D} m^2(z) :h(z)^2: dz$.

It now remains to prove the claim about the rate of convergence of $\int_{D} m^2(z) :h_{\eps}(z)^2: dz$ to $\int_{D} m^2(z) :h(z)^2: dz$. Let $0 < \eps <1/2$ and let $k \in \mathbb{N}$ be such that $\eps \in [2^{-(k+1)}, 2^{-k})$. For $n \in \mathbb{N}$, set $a_n = 2^{-n}$. Using the triangle inequality and the inequality \eqref{decay_L2}, we have that, for some constant $C > 0$,
\begin{align*}
    &\bigg \| \int_{D} m^2(z) :h(z)^2: dz - \int_{D} m^2(z) :h_{\eps}(z)^2: dz\bigg \|_{L^2(\PP)}\\
    &\leq \bigg \| \int_{D} m^2(z) :h_{a_{k+1}}(z)^2: dz - \int_{D} m^2(z) :h_{\eps}(z)^2: dz\bigg \|_{L^2(\PP)} \\
    &+\sum_{n=k+1}^{\infty}  \bigg \| \int_{D} m^2(z) :h_{a_{n+1}}(z)^2: dz - \int_{D} m^2(z) :h_{a_n}(z)^2: dz\bigg \|_{L^2(\PP)} \\
    &\leq C2^{-k\alpha}(k+1) + C\sum_{n=k+1}^{\infty} (a_{n}^{2\alpha} (\log(a_{n+1})^2))^{1/2}.
\end{align*}
Let $\eta \in (0, \alpha)$. Then, we have that
\begin{equation*}
    \sum_{n=k}^{\infty} 2^{-n\alpha}(n+1) \leq 2^{-(\alpha - \eta)k} \sum_{n=k}^{\infty} 2^{-n\eta}(n+1) \leq 2^{(\alpha - \eta)}\eps^{(\alpha - \eta)} \sum_{n=0}^{\infty} 2^{-n\eta}(n+1)
\end{equation*}
where the sum on the right-hand side is finite. This concludes the proof of \eqref{rate_cvg_L2_Wick} and thus of the lemma.
\end{proof}

Let us now turn to the proof of Lemma \ref{lemma_L2_cvg_Wick_boundary}.

\begin{proof}[Proof of Lemma \ref{lemma_L2_cvg_Wick_boundary}]
For conciseness, let us set $\alpha := 2 - \operatorname{dim}_{H}(\partial D) \in (0,1]$. Observe first that, almost surely, for any $0 <\eps <1/2$,
\begin{align*}
    \int_{D} m^2(z) :(h_{\eps}(z)+\phi_{\eps}(z))^2:dz
    = \int_{D} m^2(z) :h_{\eps}(z)^2: dz + \int_{D} 2m^2(z) h_{\eps}(z)\phi_{\eps}(z) dz + \int_{D} m^2(z) \phi_{\eps}(z)^{2} dz.
\end{align*}
Therefore, by Lemma \ref{lemma_L2_cvg_Wick}, it suffices to prove that
\begin{equation} \label{lim_L2_hphi}
    \lim_{\eps \to 0} \int_{D} m^2(z) h_{\eps}(z)\phi_{\eps}(z) dz = (h, m^2\phi)
\end{equation}
where the limit is in $L^{2}(\PP)$ and that
\begin{equation} \label{lim_phi_square}
    \lim_{\eps \to 0} \int_{D} m^2(z) \phi_{\eps}(z)^2 dz = \int_{D} m^2(z)\phi(z)^2dz.
\end{equation}
Let us start with \eqref{lim_phi_square}. We note that, by harmonicity of $\phi$,
\begin{equation*}
    \bigg \vert \int_{D} m^2(z) \phi(z)^2 dz - \int_{D} m^2(z) \phi_{\eps}(z)^2 dz \bigg \vert = \bigg \vert \int_{D \setminus D_{\eps}} m^2(z) (\phi(z)^2 - \phi_{\eps}(z)^2) dz \bigg \vert.
\end{equation*}
By assumption on $\phi$ and $\operatorname{dim}_{H}(\partial D)$, this yields that, for any $b \in (0,\alpha]$,
\begin{equation} \label{decay_harmonic_part}
     \bigg \vert \int_{D} m^2(z) \phi(z)^2 dz - \int_{D} m^2(z) \phi_{\eps}(z)^2 dz \bigg \vert \leq 2b_{\phi}^2\int_{D \setminus D_{\eps}} m^2(z) dz \leq O(\eps^b),
\end{equation}
from which \eqref{lim_phi_square} follows. Let us now turn to \eqref{lim_L2_hphi}. One could argue directly because for each $\eps > 0$, we have a Gaussian random variable. However, since we also want to prove the rate of decay \eqref{rate_cvg_L2_Wick_boundary}, we are in fact going to show that, for any $b \in (0,\alpha)$, there exists $C_1>0$ such that for any $0<\eps<1/2$,
\begin{equation} \label{decay_hphi_L2}
     \bigg \| 2(h, m^2\phi) - 2\int_{D} m^2(z) h_{\eps}(z)\phi_{\eps}(z) dz  \bigg \|_{L^2(\PP)} \leq C_{1} \eps^{b}.
\end{equation}
Let $0 < \eps < 1/2$ and let $k \in \mathbb{N}$ be such that $\eps \in [2^{-(k+1)}, 2^{-k})$. As before, for $n \in \mathbb{N}$, set $a_n=2^{-n}$. By the triangular inequality, we have that
\begin{align} \label{triang_ineq_hphi}
    &\bigg \| 2(h, m^2\phi) - 2\int_{D} m^2(z) h_{\eps}(z)\phi_{\eps}(z) dz  \bigg \|_{L^2(\PP)} \nonumber \\
    &\leq 2\bigg \| \int_{D} m^2(z) h_{a_{k+1}}(z)\phi_{a_{k+1}}(z) dz -  \int_{D} m^2(z) h_{\eps}(z)\phi_{\eps}(z) dz\bigg \|_{L^2(\PP)} \nonumber \\
    &+ 2\sum_{n=k+1}^{\infty} \bigg \| \int_{D} m^2(z) h_{a_{n+1}}(z)\phi_{a_{n+1}}(z) dz - \int_{D} m^2(z) h_{a_n}(z)\phi_{a_{n}}(z) dz  \bigg \|_{L^2(\PP)}.
\end{align}
To show \eqref{decay_hphi_L2}, we are going to upper bound each term in the above sum. To this end, let $n \geq k$. We have that
\begin{align*}
    &\bigg \| \int_{D} m^2(z) h_{a_{n+1}}(z)\phi_{a_{n+1}}(z) dz - \int_{D} m^2(z) h_{a_{n}}(z)\phi_{a_{n}}(z) dz  \bigg \|_{L^2(\PP)}^2 \\
    &=\EE \bigg[ \int_{D \times D} m^2(z)m^2(w) ( h_{a_{n+1}}(z)\phi_{a_{n+1}}(z) - h_{a_{n}}(z)\phi_{a_{n}}(z))(h_{a_{n+1}}(w)\phi_{a_{n+1}}(w) - h_{a_{n}}(w)\phi_{a_{n}}(w)) dw dz \bigg].
\end{align*}
Set $D_n := \{ z \in D: \operatorname{dist}(z, \partial D) > a_n \}$ and observe that, by harmonicity of $\phi$, for any $z \in D_n$, $\phi_{a_n}(z)=\phi(z)=\phi_{a_{n+1}}(z)$. We decompose the region of integration into $D_n \times D$, $D \times D_n$ and $(D \setminus D_n) \times (D \setminus D_n)$. To bound the integral over $D_n \times D$, if $(z,w) \in D_n \times D$ are such that $\vert z -w \vert > 2a_n$, then we can use the Markov property in $B(z,2a_n)$ and the harmonicity of $\phi$ to obtain that
\begin{align*}
    &\EE[( h_{a_{n+1}}(z)\phi_{a_{n+1}}(z) - h_{a_{n}}(z)\phi_{a_{n}}(z))(h_{a_{n+1}}(w)\phi_{a_{n+1}}(w) - h_{a_{n}}(w)\phi_{a_{n}}(w))] \\
    &= \phi(z) \EE[( h_{a_{n+1}}(z) - h_{a_{n}}(z))(h_{a_{n+1}}(w)\phi_{a_{n+1}}(w) - h_{a_{n}}(w)\phi_{a_{n}}(w))] = 0.
\end{align*}
Therefore, only points $z,w$ such that $\vert z - w \vert \leq 2a_n$ contribute to the integral over $D_n \times D$. Note also that, by harmonicity, for any $n \in \mathbb{N}$ and any $z \in D$, $\vert \phi_{a_n}(z)\vert \leq b_{\phi}$. Therefore, using the estimate \eqref{estimate_sup_variance}, Cauchy-Schwarz inequality and the boundedness of $m$, we obtain that
\begin{align*}
    \int_{D_n \times D} m^2(z)m^2(w)\phi(z) \EE[(h_{a_{n+1}}(z)- h_{a_{n}}(z))(h_{a_{n+1}}(w)\phi_{a_{n+1}}(w)- h_{a_{n}}(w)\phi_{a_n}(w))] dw dz
    \leq O(a_n^{2} \log a_{n+1}^{-1}).
\end{align*}
The integral over $D \times D_n$ can be treated similarly. To control the integral over $(D \setminus D_n) \times (D \setminus D_n)$, we simply use the estimate \eqref{estimate_sup_variance}, Cauchy-Schwarz inequality, the uniform boundedness of $\phi_{a_n}$ in $n$ and $z$ and the boundedness of $m$ and $D$ to get that
\begin{align*}
    &\int_{(D \setminus D_n) \times (D \setminus D_n)} m^2(z)m^2(w) \EE[(h_{a_{n+1}}(z)\phi_{a_{n+1}}(z)- h_{a_{n}}(z)\phi_{\eps_p}(z))(h_{a_{n+1}}(w)\phi_{a_{n+1}}(w)- h_{a_{n}}(w)\phi_{a_n}(w))] dw dz \\
    &\leq O(a_n^{2\alpha} \log(a_{n+1}^{-1})).
\end{align*} 
Going back to \eqref{triang_ineq_hphi}, we can deduce from the above bounds that for some constant $C>0$,
\begin{align*}
    \bigg \| 2(h, m^2\phi) - 2\int_{D} m^2(z) h_{\eps}(z)\phi_{\eps}(z) dz  \bigg \|_{L^2(\PP)} \leq 2C2^{-k\alpha}\sqrt{k+1} + 2C \sum_{n=k+1}^{\infty} a_n^{\alpha} \sqrt{(\log(a_{n+1}^{-1}))}.
\end{align*}
For the same reasons as those explained in the proof of Lemma \ref{lemma_L2_cvg_Wick}, this bound yields the inequality \eqref{decay_hphi_L2}. Moreover, by the triangular inequality, the statement of the lemma about the rate of convergence in $L^2(\PP)$ follows from Lemma \ref{lemma_L2_cvg_Wick}, \eqref{decay_hphi_L2} and \eqref{decay_harmonic_part}.
\end{proof}

As a corollary of Lemma \ref{lemma_L2_cvg_Wick}, we obtain an explicit expression for the variance of the Wick square of the GFF. This expression is as mentioned in Remark \ref{remark_variance_Wick}.

\begin{lemma} \label{lemma_variance_Wick}
Under the same assumptions as in Lemma \ref{lemma_L2_cvg_Wick},
\begin{equation*}
    \EE \bigg[ \bigg( \int_{D} m^2(z) :h(z)^2: dz \bigg)^2 \bigg] = 2\int_{D \times D} m^2(z)m^2(w)G_{D}(z,w)^2 dw dz.
\end{equation*}
\end{lemma}

\begin{proof}
From the $L^2(\PP)$ convergence result established in Lemma \ref{lemma_L2_cvg_Wick}, we have that
\begin{align*}
    \EE \bigg[ \bigg( \int_{D} m^2(z) :h(z)^2: dz \bigg)^2 \bigg]
    &= \lim_{\eps \to 0} \EE \bigg[ \bigg( \int_{D} m^2(z) :h_{\eps}(z)^2: dz \bigg)^2 \bigg] \\
    &= \lim_{\eps \to 0}  \int_{D \times D} m^2(z) m^2(w) \EE[:h_{\eps}(z)^2::h_{\eps}(w)^2:] dw dz
\end{align*}
where in the last equality, the application of Fubini theorem can be justified as in the proof of Lemma \ref{lemma_L2_cvg_Wick}. Writing, for $z \in D$,
\begin{equation*}
    :h_{\eps}(z)^2: = \EE[h_{\eps}(z)^2]H_2\bigg( \frac{h_{\eps}(z)}{\EE[h_{\eps}(z)^2]^{1/2}}\bigg)
\end{equation*}
where $H_2(x)=x^2-1$ is the Hermite polynomial of degree 2, \cite[Theorem~I.3]{SimonEQFT} shows that $\EE[:h_{\eps}(z)^2::h_{\eps}(w)^2:] = 2 \EE[h_{\eps}(z)h_{\eps}(w)]^2$. This yields that
\begin{equation*}
    \EE \bigg[ \bigg( \int_{D} m^2(z) :h(z)^2: dz \bigg)^2 \bigg] = \lim_{\eps \to 0}  \int_{D \times D} m^2(z) m^2(w) 2 \EE[h_{\eps}(z)h_{\eps}(w)]^2 dwdz.
\end{equation*}
Moreover, it is easy to see that the integral on the above right-hand side converges to
\begin{equation*}
    \int_{D \times D} m^2(z) m^2(w) 2 G_{D}(z,w)^2 dwdz,
\end{equation*}
which completes the proof the lemma.
\end{proof}

Let us now establish Lemma \ref{lemma_Lp_cvg_boundary}. The main idea behind the proof is similar to \cite{Nelson_2, Nelson_1}, which show finiteness of the partition function of the $P(\Phi)_2$-theory, see also \cite[Chapter~V]{SimonEQFT}. Here, we must adapt these ideas to approximations of the GFF by circle-averages and consider the fact that we are dealing with a massless GFF in a domain with boundary instead of a massive GFF on the plane (with the interaction term being restricted to a finite subset).

\begin{proof}[Proof of Lemma \ref{lemma_Lp_cvg_boundary}]
Fix $p \in [1,\infty)$. To show the lemma, we are first going to show that
\begin{equation} \label{lim_proba_intro}
    \lim_{\eps \to 0} \exp \bigg(-\frac{p}{2} \int_{D} m^2(z) :(h_{\eps}(z)+\phi_{\eps}(z))^2:dz \bigg) = \exp \bigg(-\frac{p}{2} \int_{D} m^2(z) :(h+\phi)^2(z): dz \bigg)
\end{equation}
in probability. Then, we will prove that there exists $M>0$ such that for any $0 < \eps < 1/2$,
\begin{equation} \label{uniform_bound_intro}
    \EE \bigg[ \bigg(  \exp \bigg(-\frac{p}{2} \int_{D} m^2(z) :(h_{\eps}(z)+\phi_{\eps}(z))^2:dz \bigg) \bigg)^2 \bigg] \leq M.
\end{equation}
The lemma follows from these two facts: indeed, the upper bound \eqref{uniform_bound_intro} establishes that the sequence
\begin{equation*}
    \bigg( \exp \bigg(-\frac{p}{2} \int_{D} m^2(z) :(h_{\eps}(z)+\phi_{\eps}(z))^2:dz \bigg), 0< \eps < 1/2 \bigg)
\end{equation*}
is uniformly integrable and therefore that it converges in $L^1(\PP)$ to its limit in probability, which is given by \eqref{lim_proba_intro}. Noting that
\begin{equation*}
    \EE \bigg[ \exp \bigg(-\frac{p}{2} \int_{D} m^2(z) :(h_{\eps}(z)+\phi_{\eps}(z))^2:dz \bigg) \bigg] = \EE \bigg[ \bigg(  \exp \bigg(-\frac{1}{2} \int_{D} m^2(z) (h_{\eps}(z)+\phi_{\eps}(z))^2:dz \bigg) \bigg)^p \bigg]
\end{equation*}
and similarly for $\exp(-\int_{D} m^2(z):(h+\phi)^2(z):dz)$, we then obtain the lemma.

The fact that \eqref{lim_proba_intro} holds is a consequence of Lemma \ref{lemma_L2_cvg_Wick_boundary}. Indeed, if we consider the mass $z \mapsto (p/2)m^2(z)$, then this lemma shows that
\begin{equation*}
     \lim_{\eps \to 0} \frac{p}{2} \int_{D} m^2(z) :(h_{\eps}(z)+\phi_{\eps}(z))^2:dz = \frac{p}{2} \int_{D} m^2(z) :(h+\phi)^2(z): dz
\end{equation*}
where the limit is in $L^2(\PP)$. From this and the continuity of the function $x \mapsto e^{-x}$, we then easily deduce \eqref{lim_proba_intro}.

Let us now show the uniform bound \eqref{uniform_bound_intro}. Let $0<\eps < 1/2$ and for conciseness, set
\begin{align*}
    Y_{\eps} := p\int_{D} m^2(z) :(h_{\eps}(z)+\phi_{\eps}(z))^2: dz, \quad Y := p\int_{D} m^2(z) :(h+\phi)^2(z): dz.
\end{align*}
We first observe that for that almost surely, for any $0 <\eps < 1/2$ and any $z \in D$,
\begin{equation*}
    h_{\eps}(z)^2 + 2h_{\eps}(z)\phi_{\eps}(z) \geq h_{\eps}(z)^2 - 2\vert h_{\eps}(z) \vert b_{\phi} \geq -b_{\phi}^2.
\end{equation*}
The rightmost inequality is obtained by noting that the minimal value of the polynomial $x \mapsto x^2-2\vert x \vert b_{\phi}$ is $-b_{\phi}^2$ (reached at $x=\pm b_{\phi}$). It follows from this that, almost surely, for any $\eps > 0$,
\begin{equation*}
    \int_{D} m^2(z):(h_{\eps}(z)+\phi_{\eps}(z))^2: dz \geq -2b_{\phi}^2\int_{D} m^2(z) dz -\int_{D} m^2(z)\EE[h_{\eps}(z)^2] dz.
\end{equation*}
Together with the estimate \eqref{estimate_sup_variance}, this yields that there exists $K>0$ such that, almost surely, for any $0 < \eps < 1/2$,
\begin{equation*}
    \int_{D} m^2(z):(h_{\eps}(z)+\phi_{\eps}(z))^2: dz \geq -K\log(\eps^{-1}).
\end{equation*}
Above, the constant $K$ depends only on $D$, $m^2$ and $b_{\phi}$. Using this almost sure lower bound, we then have that, for some constant $\tilde C > 0$,
\begin{align} \label{bound_expectation_1}
    &\EE \bigg[ \bigg(  \exp \bigg(-\frac{p}{2} \int_{D} m^2(z) :(h_{\eps}(z)+\phi_{\eps}(z))^2:dz \bigg) \bigg)^2 \bigg] \nonumber \\
    &= \EE \big[ \exp \big(-(Y_{\eps}-Y) \big) \exp \big(-Y \big) \mathbb{I}_{\{Y_{\eps} - Y > -1\}} \big] + \EE \big[ \exp \big(-(Y_{\eps}-Y) \big) \exp \big(-Y \big) \mathbb{I}_{\{Y_{\eps} - Y \leq -1\}} \big] \nonumber \\
    &\leq \exp(1) \EE[\exp(-Y)] + \exp\big(\tilde C \|m^2\|_{L^1(D)} \log(\eps^{-1}) \big) \PP(Y_{\eps} - Y \leq -1) \nonumber \\
    &\leq \exp(1) \EE[\exp(-Y)] + \exp\big(\tilde C \|m^2\|_{L^1(D)} \log(\eps^{-1}) \big) \PP(\vert Y_{\eps} - Y \vert \geq 1),
\end{align}
where $\EE[\exp(-Y)]$ is finite by \cite{LeJan}. We thus see that to prove the uniform bound \eqref{uniform_bound_intro}, we must show that the probability $\PP(\vert Y_{\eps} - Y \vert \geq 1)$ decays fast enough as $\eps \to 0$. This will be achieved thanks to Lemma \ref{lemma_L2_cvg_Wick_boundary}. Indeed, by Markov inequality, for any $q \in [2,\infty)$,
\begin{equation*}
    \PP(\vert Y_{\eps} - Y \vert \geq 1) \leq \EE[\vert Y_{\eps} - Y \vert^q].
\end{equation*}
Moreover, as both $Y_{\eps}$ and $Y$ are random variables in $\Gamma_{0}(\mathcal{H}) \oplus \Gamma_{1}(\mathcal{H}) \oplus \Gamma_{2}(\mathcal{H})$ where $\mathcal{H}$ is the Gaussian Hilbert space corresponding to the GFF $h$, by \cite[Theorem~I.22]{SimonEQFT}, we have that
\begin{equation*}
    \| Y_{\eps} - Y \|_{L^q(\PP)} \leq (q-1) \| Y_{\eps} - Y \|_{L^2(\PP)}.
\end{equation*}
Together with Lemma \ref{lemma_L2_cvg_Wick_boundary}, this yields that, for any $b \in (0,\alpha)$, there exists a constant $C>0$ such that for any $0 < \eps < 1/2$,
\begin{equation*}
    \PP(\vert Y_{\eps} - Y \vert \geq 1) \leq (q-1)^q \EE[\vert Y_{\eps} - Y \vert^2]^{q/2} \leq C^q(q-1)^q \eps^{qb}.
\end{equation*}
We are now going to choose $q$ in a $\eps$-dependent way. Namely, set
\begin{equation*}
    q= \eps^{-b/3}.
\end{equation*}
Then, $(q-1)^q\leq q^q = \eps^{-qb/3}$. Moreover, there exists $\tilde \eps_{1} >0$ such that for any  $0 < \eps < \tilde \eps_{1}$, $C^q \leq \eps^{-qb/3}$. Therefore, for any $0 < \eps < \tilde \eps_{1}$,
\begin{equation*}
    \PP(\vert Y_{\eps} - Y \vert \geq 1) \leq \eps^{-qb/3} \eps^{-qb/3} \eps^{bq} = \eps^{qb/3}.
\end{equation*}
Thanks to our choice of $q$, it then easy to see that there exists $\tilde \eps_{2} >0$ such that for any $0 < \eps < \tilde \eps_{2}$,
\begin{equation*}
    \PP(\vert Y_{\eps} - Y \vert \geq 1) \leq \exp(-\eps^{-b/3}).
\end{equation*}
Going back to \eqref{bound_expectation_1}, we obtain that for any $0 < \eps < \tilde \eps_{2}$,
\begin{equation*}
    \EE [ \exp (-Y_{\eps}) ] \leq \exp(1) \EE[\exp(-Y)] + \exp\big(\tilde C \|m^2\|_{L^1(D)} \log(\eps^{-1}) \big)\exp(-\eps^{-b/3}).
\end{equation*}
Since $b > 0$, the second summand on the right-hand side of the above inequality converges to $0$ as $\eps \to 0$ and is therefore bounded by some constant $C>0$ independent of $\eps$. This shows that, for any $0 < \eps < \tilde \eps_{2}$,
\begin{equation*}
    \EE [ \exp (-Y_{\eps}) ] \leq \exp(1) \EE[\exp(-Y)] + C.
\end{equation*}
We deduce from this the uniform bound \eqref{uniform_bound_intro}, which as explained above, concludes the proof of the lemma.
\end{proof}

Finally, let us show the almost sure strict positiveness of the random variable on the right-hand side of \eqref{RN_mGFF_phi}. Note that by taking $\phi \equiv 0$, the lemma below implies that the random variable on the right-hand side of \eqref{RN_mGFF_Dirichlet} is also almost surely strictly positive.

\begin{lemma} \label{lemma_non0}
Almost surely,
\begin{equation*}
    \exp \bigg(-\frac{1}{2} \int_{D} m^2(z) :(h+\phi)^2(z): dz \bigg) > 0.
\end{equation*}
\end{lemma}

\begin{proof}
This lemma follows from the fact that
\begin{equation*}
    \PP \bigg( \exp \bigg(-\frac{1}{2} \int_{D} m^2(z) :(h+\phi)^2(z): dz \bigg) > 0\bigg) = \PP \bigg(\frac{1}{2} \int_{D} m^2(z) :(h+\phi)^2(z): dz < \infty \bigg) = 0
\end{equation*}
since the random variable $\int_{D} m^2(z) :(h+\phi)^2(z): dz$ is in $L^2(\PP)$.
\end{proof}

\bibliography{massiveSLE4_arxiv}
\end{document}